\documentclass[a4paper,12pt,leqno, twoside]{article}
\usepackage{amssymb}
\usepackage{amsthm}
\usepackage[all]{xy} 
\usepackage[mathscr]{eucal}
\usepackage{hyperref}

\def\UM{{\mathbb{U}}}

\def\GM{{\mathbb{G}}}
\def\PM{{\mathbb{P}}}

\def\QM{{\mathbb{Q}}}
\def\FM{{\mathbb{F}}}
\def\ZM{{\mathbb{Z}}}

\def\AM{{\mathbb{A}}}

\def\sG{{\mathfrak s}}

\def\SG{{\mathfrak S}}

\def\AC{{\mathcal A}}

\def\SC{{\mathcal S}}
\def\CC{{\mathcal C}}
\def\HC{{\mathcal H}}
\def\NC{{\mathcal N}}

\def\RC{{\mathcal R}}

\def\OC{{\mathcal O}}
\def\MC{{\mathcal M}}

\def\PC{{\mathcal P}}

\def\LC{{\mathcal L}}
\def\EC{{\mathcal E}}

\def\VC{{\mathcal V}}

\def\BC{{\mathcal B}}

\def\YC{{\mathcal Y}}

\def\simto{\buildrel\hbox{\tiny{$\sim$}}\over\longrightarrow}

\def\leq{\leqslant}
\def\geq{\geqslant}
\def\injo{\hookrightarrow}
\def\id{\mathop{\mathrm{Id}}\nolimits}

\def\ba{\backslash}
\def\wt{\widetilde}
\def\wh{\widehat}
\def\o#1{\overline{#1}}

\def\application#1#2#3#4#5{\begin{array}{rcl}
                            #1 \;\;\; #2 & \to &  #3 \\
                              #4 & \mapsto & #5 
                            \end{array}}

\def\To#1{\buildrel\hbox{\tiny{$#1$}}\over\longrightarrow}
\def\to{\rightarrow}

\def\ker{\mathop{\hbox{\sl ker}\,}}
\def\coker{\mathop{\hbox{\sl coker}\,}}
\def\im{\mathop{\hbox{\sl im}\,}}

\def\Hom{\mathop{\hbox{\rm Hom}}\nolimits}
\def\Aut{\mathop{\hbox{\rm Aut}}\nolimits}

\def\Rep{{\rm {R}ep}}

\def\Irr#1#2{\mathop{{\rm Irr}_{#1}\left(#2\right)}}

\def\ind#1#2#3{\hbox {\rm Ind}_{#1}^{#2}\>\!\left(#3\right)}  

\def\cind#1#2#3{\hbox {\rm ind}_{#1}^{#2}\>\!\left(#3\right)} 
\def\ip#1#2#3{\hbox {\sl i}_{#1}^{#2}\>\!(#3)}  
\def\Ip#1#2{\hbox {\sl i}_{#1}^{#2}}

\def\dim{\mathop{\mbox{\rm dim}}\nolimits}
\def\dim{{\rm dim}}

\def\colim{{\rm colim}}

\def\ini{\setcounter{equation}{\value{subsubsection}}\addtocounter{subsubsection}{1}}

\makeatletter
\renewcommand{\subsubsection}{\@startsection{subsubsection}{3}{\parindent}{-\baselineskip}{-0.01\baselineskip}{\bf}}
\renewcommand*{\@seccntformat}[1]{%
  \csname the#1\endcsname\
}
\makeatother

\def\alin#1{\setcounter{equation}{0}\subsubsection{\it  #1}. --- }

\newtheoremstyle{th}
  {\baselineskip}{.5\baselineskip}{\itshape}
  {\parindent}{\bf}
  { ---}{.5em}{}

\newtheoremstyle{def}
  {\baselineskip}{\baselineskip}{}
  {\parindent}{\bf}
  {--}{.5em}{}

\newtheoremstyle{th*}
  {.5\baselineskip}{.5\baselineskip}{\itshape}
  {\parindent}{\bf}
  { ---}{.5em}{}

\newtheoremstyle{remark*}
  {.5\baselineskip}{.5\baselineskip}{}
  {\parindent}{\bf}
  {--}{.5em}{}

\newtheoremstyle{remark}
  {.5\baselineskip}{.5\baselineskip}{}
  {\parindent}{\bf}
  { ---}{.5em}{}

\swapnumbers
\theoremstyle{th}
\newtheorem{theo}[subsubsection]{\it Theorem.\bf}
\newtheorem{lemme}[subsubsection]{\it Lemma.\bf}
\newtheorem{prop}[subsubsection]{\it Proposition.\bf}
\newtheorem{coro}[subsubsection]{\it Corollary.\bf}

\newtheorem{DEf}[subsubsection]{\it D{e}finition.\bf}
\newtheorem{fact}[subsubsection]{\it Fact\bf}
\theoremstyle{def}

\theoremstyle{remark}
\newtheorem{exam}[subsubsection]{\it Example.\bf}   

\theoremstyle{th*}
\newtheorem*{thm}{\it Theorem.}
\newtheorem*{lem}{\it Lemma.}
\newtheorem*{pro}{\it Proposition.}
\newtheorem*{cor}{\it Corollary.}
\newtheorem*{con}{\it Conjecture.}
\newtheorem*{defn}{\it Definition.}
\newtheorem*{fac}{\it Fact.}

\theoremstyle{remark*}

\newcommand{\findem}{\hfill$\Box$\par\medskip}
\newcommand{\dem}{\indent {\it Preuve :} \rm }

\title{Equivalences of tame blocks for $p$-adic linear groups}
\author{Jean-Fran\c{c}ois Dat\footnote{The author thanks the Institut
    Universitaire de France and the  ANR project  ANR-14-CE25-0002 PerCoLaTor for their
  support.}}
\date{March 2016}

\AtEndDocument{{\footnotesize
  \textsc{Jean-Fran\c{c}ois Dat.}  Institut de Math\'ematiques de Jussieu-PRG, Universit\'e Pierre et Marie Curie, 4
    place Jussieu, 75005 Paris, France.
  \textit{e-mail :} \texttt{jean-francois.dat@imj-prg.fr}
}}

\begin{document}
\maketitle
\bibliographystyle{plain}

\def\la{\langle}
\def\ra{\rangle}
\def\knr{{\wh{K^{nr}}}}
\def\ka{\wh{K^{ca}}}

\def\dd{D_d^\times}
\def\mdro{\MC_{Dr,0}}
\def\mdrn{\MC_{Dr,n}}
\def\mdr{\MC_{Dr}}
\def\mlto{\MC_{LT,0}}
\def\mltn{\MC_{{\rm LT},n}}
\def\mltno{\MC_{{\rm LT},n}^{(0)}}
\def\mlt{\MC_{\rm LT}}
\def\mltK{\MC_{LT,K}}
\def\LJ{{\rm LJ}}
\def\JL{{\rm JL}}
\def\SL{{\rm SL}}
\def\GL{{\rm GL}}

\def\Ql{\QM_{\ell}}
\def\Zl{\ZM_{\ell}}
\def\Fl{\FM_{\ell}}
\def\oQl{\o\QM_{\ell}}
\def\oZl{\o\ZM_{\ell}}
\def\bZl{\o\ZM_{\ell}}
\def\bQl{\o\QM_{\ell}}
\def\oFl{\o\FM_{\ell}}
\def\mltnc{\wh\MC_{{\tiny{\rm LT}},n}}
\def\mltnco{\wh\MC_{{\tiny{\rm LT}},n}^{(0)}}

\def\sp{{\rm sp}}
\def\Ens{{\SC ets}}
\def\Coef{{\rm Coef}}

\begin{abstract}
 Let $p$ and $\ell$ be two distinct primes, $F$ a $p$-adic field and
 $n$ an integer. We show that any level $0$ block of the category of
 smooth $\oZl$-valued representations of $ \GL_{n}(F)$ is equivalent
 to the unipotent block of an appropriate product of
 $\GL_{n_{i}}(F_{i})$. 
To this end, we first show that the level $0$ category of $\GL_{n}(F)$ is equivalent to a
category of ``modules'' over a certain $\oZl$-algebra ``with many
objects'' whose definition only involves $n$ and the residue field
of $F$. Then we use fine properties of  Deligne-Lusztig cohomology to
split this algebra and produce suitable Morita equivalences.
\end{abstract}

\tableofcontents
\section{Introduction}

Let $F$ be a $p$-adic field and let $R$ be a commutative ring in which $p$ is invertible. 
For $\mathbf{G}$ a reductive group
over $F$, we put $G:=\mathbf{G}(F)$ and  we denote by
$\Rep_{R}(G)$ the abelian category of smooth 
 representations of $G$ with coefficients in $R$, and by $\Irr{R}{G}$ the set of
 isomorphism classes of simple objects in $\Rep_{R}(G)$. 
In this paper we focus on the following situation :
\begin{itemize}
\item $\mathbf{G}$ is of $\GL$-type, by which we mean that
$\mathbf{G}$ is  a product of restrictions of scalars of general linear  groups. 
\item $R=\oZl$, for $\ell$ a prime number coprime to $p$, occasionnally
  $R=\oFl$ or $\oQl$.
\end{itemize}

\subsection{Main results}

\alin{The block decomposition} \label{mainv1}
For $\mathbf{G}=\GL_{n}$, Vign\'eras \cite{VigInduced} has obtained a
decomposition of the category $\Rep_{\oFl}(G)$ as a product of blocks, indexed by the
so-called ``inertial classes of supercuspidal pairs'' $[M,\pi]$ with $M$ a Levi subgroup
and $\pi\in\Irr\oFl M$  supercuspidal.
This decomposition  was further lifted to a decomposition of $\Rep_{\oZl}(G)$ by
Helm in \cite{HelmBernstein}. A simple object of $\Rep_{\oZl}(G)$ belongs to
$\Rep_{[M,\pi]}(G)$ if and only if it has ``mod $\ell$ inertial supercuspidal support''
equal to  $[M,\pi]$, in the sense of \cite[Def. 4.10]{HelmBernstein}.

At the moment,  Vigneras blocks and Helm blocks are not well understood.
Actually even the structure of the \emph{principal} block of
$\Rep_{\oZl}(G)$ (the one that contains the trivial representation, and corresponds to the
supercuspidal pair $[T,1]$ with $T$ a maximal torus) is
still a mystery.

However, it is expected that any Vigneras or Helm block is  equivalent to the principal block
of another group $\mathbf{G'}$ of $\GL$-type.
One consequence of this paper is the following result in this direction.

\begin{thm}
If $\mathbf{G}$ is of $\GL$-type,  any level $0$ block of $\Rep_{\oZl}(G)$ is equivalent to the
principal block of $\Rep_{\oZl}(G')$ for an appropriate $\mathbf{G'}$ of $\GL$-type.
\end{thm}
Here, a block of $\Rep_{\oZl}(G)$ has level $0$  if any object in it
has non-zero invariants under the first congruence subgroup of some parahoric group in
$G$. In terms of supercuspidal pair $[M,\pi]$, this means that $\pi$ has level $0$ (or
``depth $0$'') in the usual sense.

In this result, the group $\mathbf{G'}$ may not be unique. This is already the case for 
$\oQl$-coefficients since the Iwahori-Hecke algebra of $\GL_{n}(F)$ only
depends on the residual field of $F$.
 However, we have explained in \cite{Datfuncto}  how a very natural $\mathbf{G'}$ stands out when we take the
Langlands-dual point of view.

\alin{Dual parametrization of blocks} \label{dual_param_blocks}
Let $W_{F}$ denote the absolute Weil group of $F$, and let $I_{F}$ be
its inertia subgroup. We will denote by $I_{F}^{\ell}$  the prime-to-$\ell$ radical of
$I_{F}$, \emph{i.e.} the kernel of the canonical map $I_{F}\To{}\ZM_{\ell}(1)$.
We have observed in \cite[\S 1.2.1]{Datfuncto} that
Vign\'eras' Langlands correspondence for supercuspidal $\oFl$-representations  of
$\GL_{n}(F)$  yields a bijection $[M,\pi] \leftrightarrow \phi$ between inertial classes of
supercuspidal pairs and isomorphism classes of  semi-simple continous representations 
$I_{F}^{\ell}\To{}\GL_{n}(\oQl)$ that can be extended to $W_{F}$.

More generally, let $^{L}\mathbf{G}=\hat\mathbf{G}\rtimes W_{F}$ denote ``the'' $L$-group
of a connected reductive group over $F$ and let $\Phi(\mathbf{G},\oQl)\subset
H^{1}(W'_{F},\hat\mathbf{G}(\oQl))$ denote the set
of admissible Weil-Deligne parameters for $\mathbf{G}$.
When $\mathbf{G}$  is of $\GL$-type, 
we again refer to  \cite[\S 1.2.1]{Datfuncto} for a bijection $[M,\pi]\leftrightarrow \phi$
 where now $\phi$ belongs to the set 
$$\Phi_{\ell-\rm inert}(\mathbf{G},\oQl)= \im( \Phi(\mathbf{G},\oQl) \To{\rm rest}
H^{1}(I_{F}^{\ell},\hat\mathbf{G}(\oQl))).$$
Therefore, to any $\phi\in \Phi_{\ell-\rm inert}(\mathbf{G},\oQl)$ is attached a block
$\Rep_{\phi}(G)$ of $\Rep_{\oZl}(G)$. The irreducible $\oQl$-representations $\pi$ in this block
are those whose Langlands parameter $\varphi_{\pi}$ satisfies $(\varphi_{\pi})_{|I_{F}^{\ell}}\sim\phi$.
In this parametrization, the \emph{principal block} corresponds to the \emph{trivial}
parameter $\phi : i\mapsto (1,i)$, 
while \emph{level $0$ blocks} correspond to 
\emph{tame} parameters $\phi$, in the sense that the restriction $\phi_{|P_{F}}$ to the wild inertia
subgroup $P_{F}$ is trivial.

\alin{The transfer conjecture} \label{transfer}
Having a dual parametrization suggests that some kind of  ``Langlands transfer'' principle
for blocks might exist. 
So, let us consider  a $L$-homomorphism $\xi:{^{L}\mathbf{G'}}\To{}{^{L}\mathbf{G}}$ of groups of
$\GL$-type,  fix an admissible parameter $\phi': I_{F}^{(\ell)}\To{}
{^{L}\mathbf{G'}}$ and put $\phi:=\xi\circ\phi'$. 
Let us denote by $C_{\hat\mathbf{G}}(\phi)$ the centralizer of the image of $\phi$.
We have formulated the following conjecture in \cite[\S 1.2.2]{Datfuncto}
\begin{con}
  Suppose that $\xi$ induces an isomorphism $C_{\hat \mathbf{G}'}(\phi')\simto C_{\hat
    \mathbf{G}}(\phi)$, and also that the projection of $\xi(W_{F})$ to $\hat G(\oQl)$ is bounded. 
Then there is an 
equivalence of categories $\Rep_{\phi'}(G')\simto \Rep_{\phi}(G)$ that interpolates the
Langlands transfer $\xi_{*}$ on irreducible $\oQl$-representations.
\end{con}

\alin{Main results}\label{main}
In this paper we 
construct an equivalence $\Rep_{\phi'}(G')\simto \Rep_{\phi}(G)$ in the following two cases.
  \begin{enumerate}
  \item $\phi'=1$ and $\xi$ is a base change $L$-morphism 
$^{L}\GL_{n}\To{} {^{L}{\rm  Res}_{F'|F}(\GL_{n})}$ with $F'$ totally ramified.
  \item $\phi'$ is tame and $\xi$ is an automorphic induction $L$-morphism 
${^{L}{\rm  Res}_{F'|F}(\GL_{n'})} \To{} {^{L}\GL_{n}}$, where $F'$ is unramified and $n=n'[F:F']$.
  \end{enumerate}

We have explained in \cite[\S 2.4]{Datfuncto} how these two particular cases imply the following
weak form of the above conjecture for tame parameters and tame $L$-morphisms.
\begin{thm} 
  Let $\xi$ be as in the conjecture, and suppose that both $\xi$ and $\phi'$ are
  tame\footnote{this means that the restrictions of $\xi$ and $\phi'$ to $P_{F}$ are $\hat\mathbf{G}$-conjugate to the map $\gamma\mapsto (1,\gamma)$}. 
Then there is  an equivalence of categories $\Rep_{\phi'}(G')\simto  \Rep_{\phi}(G)$. 
\end{thm}
This result, in turn, implies the following more precise version of Theorem \ref{mainv1}, for which
we refer again to \cite[\S 1.2.3]{Datfuncto}. For any $\phi\in\Phi_{\ell-\rm
  inert}(\mathbf{G},\oQl)$ there is a unique connected reductive group $\mathbf{G}_{\phi}$
of $\GL$-type over $F$, such that $\hat\mathbf{G}_{\phi}=C_{\hat\mathbf{G}}(\phi)$ and 
$^{L}\mathbf{G}_{\phi}$ is isomorphic to the normalizer $\NC_{^{L}\mathbf{G}}(\phi)$ of the image
of $\phi$ in $^{L}\mathbf{G}$. Moreover $\mathbf{G}_{\phi}$ is unramified if $\phi$ is tame.

\begin{cor}
  For any tame parameter $\phi$ of a group $\mathbf{G}$ of $\GL$-type, there is an
  equivalence of categories $\Rep_{1}(G_{\phi})\simto \Rep_{\phi}(G)$.
\end{cor}

What the theorem and the corollary are missing at the moment is the compatibility with the
transfer maps on irreducible $\oQl$-representations. This would follow from a natural and hoped-for
compatibility with parabolic induction of our explicit  constructions in cases i) and ii)
above, that we haven't proved yet.

\subsection{Outline of the construction} \label{subsecoutline}

Here we outline the construction of the desired equivalences in the two settings i) and
ii) of \ref{main}.
Note that the analog of i) for $\oQl$-coefficients boils down to the fact that the
  Iwahori-Hecke algebra only depends on the residue field of $F$. The analog of ii) for
  $\oQl$-coefficients is a little more complicated and uses the computation of Hecke
  algebras for level $0$ simple types (in the sense of Bushnell and Kutzko). 
These methods seem difficult to adapt to $\oZl$-coefficients at the moment,
 because there is no known connection between blocks and categories of modules over
  sufficiently explicit algebras.
Our method is therefore completely different. It relies on modular  Deligne-Lusztig
theory, that provides a crucial input from representation theory of finite linear groups, and on
coefficient systems on the building, that provide the necessary link between finite and
$p$-adic linear groups. Technically, we replace ``Hecke algebras'' by
``Hecke algebras with many objects'', see \ref{into_Hecke_alg}.

\alin{Tame inertial parameters of $\GL_{n}(F)$ and semisimple conjugacy
  classes in $\GL_{n}(\FM)$} \label{param-conjcl} 
Recall that the quotient  $I_{F}/P_{F}$ is procyclic of pro-order
prime to $p$. More precisely, let us fix an isomorphism $\iota:\,(\QM/\ZM)_{p'}\simto
\o\FM^{\times}$, where $\FM=\FM_{q}$ is the residue field of $F$. This lifts uniquely to an embedding
$\iota:\,(\QM/\ZM)_{p'}\injo \o F^{\times}$, and provides us with a generator
of $I_{F}/P_{F}=\ZM_{p'}(1)$, namely the system $\gamma:=(\iota(\frac 1 N))_{(N,p)=1}$.
Therefore, a continuous morphism $\phi:\,I_{F}/P_{F}\To{} \GL_{n}(\oQl)$ (with discrete topology on
target) is given by the image $s=\phi(\gamma)$, which has  finite and prime-to-$p$ order. Such a morphism extends to
$W_{F}$ if and only if $s^{q}$ is conjugate to $s$. Hence
a tame parameter $\phi\in\Phi_{\rm inert}(\GL_{n},\oQl)$ is given by a conjugacy class in
$\GL_{n}(\oQl)$ stable under $q^{th}$-power. Such a conjugacy class is determined by its
characteristic polynomial $P_{\phi}(X)\in\oQl[X]$, which has the form
$P_{\phi}(X)=\prod_{i=1}^{n}(X-\zeta_{i})=\prod_{i=1}^{n}(X-\zeta_{i}^{q})$ with $\zeta_{i}$ a
$p'$-root of unity. 

Let us now choose an embedding  $\jmath:\,(\QM/\ZM)_{p'}\injo
\oZl^{\times}$. Via $\iota\circ\jmath^{-1}$ we can transfer
$P_{\phi}(X)$ to  a polynomial in $\o\FM[X]$, and the functional equation
of $P_{\phi}(X)$ says that
 this polynomial is in fact in $\FM[X]$. Therefore it corresponds to a semisimple
conjugacy class $s$ in $\GL_{n}(\FM)$.  

In the same way, a tame $\ell$-inertia parameter
$I_{F}^{\ell}\To{}\GL_{n}(\oQl)$ corresponds to a semisimple conjugacy class
$s$ in $\GL_{n}(\FM)$ whose elements have prime-to-$\ell$ order. This fits into a diagram
$$\xymatrix{
\Phi_{\rm inert}(\GL_{n},\oQl) \ar[r]^-{\sim} \ar[d]_{\rm rest}
& 
\{\hbox{semi-simple conjugacy classes in }\GL_{n}(\FM)\} \ar[d]^{\ell'-\rm part}
\\
\Phi_{\ell-\rm inert}(\GL_{n},\oQl) \ar[r]^-{\sim}
&
\{\hbox{$\ell'$-semi-simple conjugacy classes in }\GL_{n}(\FM)\}.
}$$

\alin{Semisimple conjugacy classes and idempotents in  $\oZl\GL_{n}(\FM)$}
Green's classification of  irreducible representations of the finite
group $\GL_{n}(\FM)$ provides us with a
surjective map $\Irr{\oQl}{\GL_{n}(\FM)}\To{} \GL_{n}(\FM)^{\rm
  ss}_{/\rm conj}$, which a priori depends on our choices of $\iota$ and
$\jmath$. This map  can also be understood in terms of
 Deligne-Lusztig ``twisted'' induction, and its fibers are called
 ``(Lusztig) rational series'' (see section \ref{sec:finite-linear-groups} for a brief account). 
Let $s$ be a semisimple conjugacy class in $\GL_{n}(\FM)$ of order
prime to $\ell$.  Brou\'e and Michel have proved that the central idempotent 
$$e_{s}:=\sum_{(s')_{\ell'}=s}\sum_{\pi\mapsto s'}e_{\pi}$$ of
$\oQl\GL_{n}(\FM)$ actually lies in $\oZl\GL_{n}(\FM)$. Here $e_{\pi}\in\oQl\GL_{n}(\FM)$
is the central idempotent attached to the irreducible representation $\pi$.

\alin{$\Rep_{\phi}(G)$ as a category of coefficient systems} \label{intro_coeff}
Denote by $BT$ the reduced building of $\GL_{n}(F)$ with its simplicial structure. 
A vertex $x$ of $BT$ corresponds to an homothety class of $\OC_{F}$-lattices, so its compact
stabilizer comes with a reduction map $G_{x}^{0}\To{}\GL_{n}(\FM)$ well defined up to inner
automorphism of the target. 
Denote by $s$ the conjugacy class in $\GL_{n}(\FM)$ that corresponds to $\phi$ as in
\ref{param-conjcl}. 
 
Taking inverse image of $e_{s}$ by the reduction map, we get
a central idempotent $e_{x}^{\phi}$ in the smooth distribution algebra
$\HC_{\oZl}(G_{x}^{0})$.
More generally, the compact fixator $G_{\sigma}^{0}$ of a facet comes with a reduction map 
to a standard Levi subgroup of $\GL_{n}(\FM)$, and  we can again define a central idempotent $e_{\sigma}^{\phi}$ in $\HC_{\oZl}(G^{0}_{\sigma})$.

Consider the category ${\rm Coef}_{e^{\phi}}^{G}(BT)$ of all $G$-equivariant coefficient systems 
in $\oZl$-modules  $(\VC_{\sigma})_{\sigma\in BT_{\bullet}}$ on $BT$ such that for all $\tau\leq\sigma$, the structural map
$\VC_{\sigma}\To{}\VC_{\tau}$ is an isomorphism $\VC_{\sigma}\simto
e_{\sigma}^{\phi}\VC_{\tau}$. Using results of Meyer-Solleveld \cite{MS1} and Wang \cite{Wang3} we
prove in part \ref{sec:level-0-blocks} that the functor which takes $V\in\Rep_{\oZl}(G)$ to the
$G$-equivariant coefficient system 
 $(e_{\sigma}^{\phi}V)_{\sigma\in BT_{\bullet}}$ induces an equivalence of categories
$$ \Rep_{\phi}(G) \simto {\rm Coef}_{e^{\phi}}^{G}(BT).$$

\alin{$\Rep_{\phi}(G)$ as a category of modules over a Hecke ``algebra''} \label{into_Hecke_alg}
In \ref{abstract_appart} we introduce a category $[\AM/\bar\GM]$, that only depends on $n$ and on
the residual field $\FM$ of $F$. We start with the category $[\AM/W]$  given by the Coxeter complex
attached to the affine Weyl group $W$ of type $\wt A_{n-1}$ with its natural action by $W$.
Then we enrich it by plugging in, for each
simplex $\sigma$, a group of automorphisms $\bar\GM_{\sigma}$ isomorphic to a well-chosen product of finite $\GL_{n_{i}}(\FM)$.
Finally we consider the $\oZl$-span $\oZl[\AM/\bar\GM]$ of this ``enriched Coxeter
complex'', that we see as ``$\oZl$-algebra with many objects''.
The appendix
recollects some supposedly standard definitions regarding the notion of
modules on such gadgets. 

To any semisimple conjugacy class $s$ in $\GL_{n}(\FM)$, we also
attach an ``algebra with many objects'' $e_{s}\oZl[\AM/\bar\GM]$, and we show
 in \ref{hecke_algebra_s}  that
$\Rep_{\phi}(G)$ is naturally equivalent to the category of ``cartesian modules'' on this
``algebra'' for  $s=s(\phi)$. The proof of this  result takes the previous one about
coefficient systems on the building as a starting point.

The fact that the construction of $e_{s}\oZl[\AM/\bar\GM]$ only depends on $n,\FM$ and $s$
shows that we have in particular constructed an equivalence of categories as claimed in
\ref{main} i).

\alin{Automorphic induction setting}
Now, let $F'$,  $n'$ and $\xi$ be as in \ref{main} ii), and let $\FM'$ be the residual field of
$\FM$,  which is an extension of degree $n'$ of $\FM$. In order to simplify the notation a
bit, put $\mathbf{G}:=\GL_{n}$ and
$\mathbf{G}':={\rm Res}_{F'|F}(\GL_{n'})$. 
 We have a commutative diagram
$$\xymatrix{
\Phi_{\rm inert}(\mathbf{G'},\oQl) \ar[r]^-{\sim} \ar[d]_{\xi_{*}}
& 
\{\hbox{semi-simple conjugacy classes in }\GL_{n'}(\FM')\} \ar[d]^{\bar\xi_{*}}
\\
\Phi_{\rm inert}(\mathbf{G},\oQl) \ar[r]^-{\sim} 
&
\{\hbox{semi-simple conjugacy classes in }\GL_{n}(\FM)\}.
}$$
Here, the second line is the bijection explained in \ref{param-conjcl}, while the first
line is the same bijection for $\GL_{n',F'}$, composed with the  Shapiro bijection
$\Phi_{\rm inert}(\mathbf{G'},\oQl)\leftrightarrow \Phi_{\rm inert}(\GL_{n',F'},\oQl)$ of
\cite[Corollary 2.3.3]{Datfuncto}. The left vertical map is composition with $\xi$, and
the right vertical one  $\bar\xi_{*}$ is the transfer of conjugacy classes that one gets through any embedding
$\GL_{n'}(\FM')\injo \GL_{n}(\FM)$ obtained by choosing an $\FM$-basis of ${\FM'}^{n'}$.
In this setup, the hypothesis that $\xi$ induces an isomorphism $C_{\hat\mathbf{G'}}(\phi')\simto
C_{\hat\mathbf{G}}(\phi)$ translates into the hypothesis that the centralizer in $\GL_{n}(\FM)$ of
any element in the conjugacy class $s'$  corresponding to $\phi'$ is contained in $\GL_{n'}(\FM')$.

Of course we have a similar diagram if we replace ``inert'' by ``$\ell$-inert'' on the
left side, and restrict to conjugacy classes with order prime to $\ell$ on the right side.

\alin{Integral Deligne-Lusztig theory} \label{BRequiv}
If we embed $\GL_{n'}(\FM')$ in $ \GL_{n}(\FM)$ as above, the former group is the group of rational points
of an algebraic $\FM$-subgroup of $\GL_{n}$ that becomes isomorphic to the Levi subgroup
$\GL_{n'}^{f}$ over $\o\FM$. In this situation, any choice of a parabolic subgroup of $\GL_{n}$
with Levi subgroup $\GL_{n'}^{f}$ gives rise to a Deligne-Lusztig variety $Y$ over $\o\FM$ with a
left action of $\GL_{n}(\FM)$ and a right action of $\GL_{n'}(\FM')$.

Now, fix a conjugacy class $s'$ of $\ell'$-order in $\GL_{n'}(\FM')$, denote by $s$ the
image conjugacy class in 
$\GL_{n}(\FM)$, and suppose that $C_{\GL_{n}(\FM)}(s') \subset \GL_{n'}(\FM')$. 

The main result of \cite{BonRou} is that, in this setting,  
the cohomology space $H^{\dim}_{c}(Y,\oZl)e_{s'}$
induces  a Morita equivalence between $e_{s}\oZl \GL_{n}(\FM)$ and $e_{s'}\oZl \GL_{n'}(\FM')$.
All the necessary details are recalled in part \ref{sec:finite-linear-groups}.

\alin{A Morita equivalence for ``Hecke algebras''} Now, the strategy to exhibit an
equivalence of categories in the setting of \ref{main} ii) should be clear. We will
construct a Morita equivalence between the Hecke ``algebras'' $e_{s}\oZl[\AM/\bar\GM]$ and 
$e_{s'}\oZl[\AM'/\bar\GM']$ relative to $(\FM,n,s)$, resp. $(\FM',n',s')$. 
Such an equivalence will be given by a ``bimodule'' in the sense of \ref{bimodules},
i.e. a certain functor $[\AM'/\bar\GM']^{\rm opp}\times
[\AM/\bar\GM]\To{}\oZl\hbox{-Mod}$.

In order to construct this bimodule, we first embed $\iota:\,[\AM'/\bar\GM']\injo
[\AM/\bar\GM]$. This requires some choices, among which that of an $\FM$-basis of $\FM'$.
The value on a pair $(\sigma',\iota\sigma')$ of the desired ``bimodule'' will be given by
some cohomology space $\PC_{\sigma'}^{0}$ as in the previous paragraph ; in particular it
will induce a Morita equivalence between $e_{\iota\sigma'}^{s}\oZl[\bar\GM_{\iota\sigma'}]$
and $e^{s'}_{\sigma'}\oZl[\bar\GM_{\sigma'}]$.
The problem is then to extend this assignment to all pairs $(\sigma',\sigma)$ and make it 
into a functor.

This turned out to be a more subtle matter, and was only solved after we proved in
\cite{cdvdl2}  a property of invariance of the Deligne-Lusztig cohomology under certain change of parabolic subgroups.
The heart of this paper is therefore the  construction of this  ``bimodule'', which is
explained in \ref{cnstruc_bimodule} after some crucial step in Proposition
\ref{prop_finite}. Once the construction is done, the fact that it induces a Morita
equivalence and preserves the cartesian property is comparatively easier.

\alin{Organisation of the paper}
Section 2 reviews the needed input from Deligne-Lusztig theory. It is mostly
expository, with the exception of Proposition \ref{prop_finite} which is a crucial
technical point to be used in section 4.
Section 3 deals in detail with the content of \ref{intro_coeff} above. Section 4 is the
heart of this paper. In particuler, subsection 4.2 explains \ref{into_Hecke_alg}, while 
subsection 4.3 contains the construction of a Morita equivalence in the setting of an
unramified automorphic induction.
Some generalities about ``rings with many objects'' are postponed to the appendix.

\section{Finite linear groups}\label{sec:finite-linear-groups}

In this part, $\FM=\FM_{q}$ denotes a finite field of characteristic $p$, and $\o\FM$ an algebraic
closure of it. 

\subsection{Review of Lusztig theory}

Bold letters $\mathbf{G}$, $\mathbf{L}$, $\mathbf{P}$, etc. will
typically denote algebraic groups over $\o\FM$. A $\FM$-structure on
such a group gives rise to a ``Frobenius isogeny'' typically denoted
by $F$, such that $\FM$-rational points are given, e.g. by $G:=\mathbf{G}(\FM)=\mathbf{G}^{F}$.

\alin{Varieties and functors} 
Let $\mathbf{G}$ be a reductive group over $\o\FM$ with a
$\FM$-rational structure. Assume given a parabolic subgroup
$\mathbf{P}$ with unipotent radical $\mathbf{U}$, and assume that 
$\mathbf{P}$ contains an $F$-stable Levi subgroup $\mathbf{L}$.
The generalized Deligne-Lusztig variety associated to these data is
defined by 
$$Y_{\mathbf{P}}
:= \{ g\mathbf{U}\in 
\mathbf{G}/\mathbf{U}, g^{-1}F(g)\in \mathbf{U}\cdot {{F}(\mathbf{U}})\}.$$
This is a variety over $\o\FM$ with a left action of
$G=\mathbf{G}^{F}$ given by $(\gamma,g\mathbf{U})\mapsto \gamma
g\mathbf{U}$ and a right action of $L=\mathbf{L}^{F}$ given by  $(g\mathbf{U},\delta)\mapsto g\delta\mathbf{U}$.
The cohomology complex 
$R\Gamma_{c}(Y_{\mathbf{P}},\bZl)$ is thus a complex of $(\bZl[G],\bZl[L])$-bimodules and
induces two adjoint functors 
\begin{eqnarray*}
 & &\RC_{\mathbf{L}\subset \mathbf{P}}^{\mathbf{G}} = 
 R\Gamma_{c}(Y_{\mathbf{P}},\bZl)\otimes^{L}_{\bZl L} - 
:\,D^{b}(\bZl L) \To{} D^{b}(\bZl G) \\ &&
{^{*}\RC}_{\mathbf{L}\subset \mathbf{P}}^{\mathbf{G}}: R\Hom_{\bZl G}(R\Gamma_{c}(Y_{\mathbf{P}},\bZl),-):\,
D^{b}(\bZl G) \To{} D^{b}(\bZl L)
\end{eqnarray*}
called ``Lusztig's twisted induction'',
resp. ``restriction'', and where $D^{b}$ means ``bounded derived category''.
These functors satisfy some rather obvious transitivity properties,
and when $\mathbf{P}$ is also $F$-stable, they coincide with
Harish-Chandra's usual parabolic functors.

\alin{Duals and characters} 
Let us fix
  compatible systems of roots of unity $\iota:\,(\QM/\ZM)_{p'}\simto \o\FM_{q}^{\times}$ and
  $\jmath:\,(\QM/\ZM)_{p'}\injo \oZl^{\times}$ as in \ref{param-conjcl}.
If $\mathbf{T}$ is a torus defined over $\FM$ and $\mathbf{T}^{*}$
denotes the dual torus, also defined over $\FM$, there is a group isomorphism
$\mathbf{T}^{*F}\simto \Hom(\mathbf{T}^{F},\o\FM^{\times})$ associated to $\iota$, see
\cite[(8.14)]{CE}. 
Composing with $\jmath\circ\iota^{-1}$ gives  a
bijection $\mathbf{T}^{*F}\simto \Hom(T,\bZl^{\times})$. In other words,
characters of $T$ valued in $\oZl$ are parametrized by rational elements $s$ in the
dual torus $\mathbf{T}^{*}$.
More generally, let us denote by $\mathbf{G}^{*}$ the dual reductive group
of $\mathbf{G}$, which is also defined over $\FM$. Then any rational element $s$ 
in the maximal central torus of $\mathbf{G}^{*}$ determines a
character $\hat s :\, G\To{}\bZl^{\times}$.

\alin{Rational series} \label{defes}
Let $s$ be a semi-simple conjugacy class
in $G^{*}:=\mathbf{G}^{*F}$. An irreducible representation
$\pi\in\Irr{\bQl}{G}$ is said to \emph{belong to the rational series
  attached to $s$} if there is an $F$-stable torus $\mathbf{T}$ in
$\mathbf{G}$, an element $t\in \mathbf{T}^{*F}$ that belongs to $s$,
and a Borel subgroup $\mathbf{B}$ containing $\mathbf{T}$ such that
$\pi$ occurs with non-zero multiplicity in $[\RC_{\mathbf{T}\subset
  \mathbf{B}}^{\mathbf{G}}(\hat t)]$ (here brackets indicate that we
consider the image of the complex in the Grothendieck group). We'll denote by $\EC(G,s)$ the
set of such rational series and by $e_{s}^{G}=e_{s,\oQl}^{G}\in\bQl[G]$ the central idempotent
that cuts out these rational series.
We then have the following
remarkable properties :
\begin{enumerate}
\item 
  the subsets $\EC(G,s)$ form a partition of  $\Irr{\bQl}{G}$,
  and consequently we have the equality  $1=\sum_{s}e_{s}^{G}$ in
  $\bQl[G]$, see \cite[Thm 8.24]{CE}.
\item  if $s$ consists of $\ell'$-elements, then the idempotent
  $e_{s,\oZl}^{G}:=\sum_{s'\sim_{\ell}s}e_{s'}^{G}$, where
  $s'\sim_{\ell}s$ means that $s$ is the $\ell$-regular part of $s'$,
  belongs to $\bZl[G]$, see \cite[Thm. A' and Rem. 11.3]{BonRou}. 
\end{enumerate}

We will denote by $\EC_{\ell}(G,s)$ the set of
  simple modules of $\bZl[G]e_{s,\oZl}^{G}$.
The representations in $\EC(G,1)$ are also called \emph{unipotent}
series, and those in $\EC_{\ell}(G,1)$ are sometimes called
$\ell$-\emph{unipotent}.
 Accordingly, $\bZl[G]e_{1,\oZl}^{G}$ will be called the
\emph{$\ell$-unipotent} summand of $\bZl[G]$.

\alin{Compatibility with induction functors} \label{compatinduc}
Suppose $\mathbf{L}$ is an $F$-stable Levi subgroup contained in a
  parabolic subgroup $\mathbf{P}$ of $\mathbf{G}$.
Any semisimple conjugacy class $t$ in $L^{*}$ gives rise to a
semisimple class $s$ in $G^{*}$. The map $t\mapsto s$ is finite to one.

From transitivity of twisted induction, we see
that 
$$[R\Gamma_{c}(Y_{\mathbf{P}},\bQl)e_{t}^{L}]= 
[e_{s}^{G}R\Gamma_{c}(Y_{\mathbf{P}},\bQl)e_{t}^{L}].$$
For the same reason, we have 
$[e_{s}^{G}R\Gamma_{c}(Y_{\mathbf{P}},\bQl)e_{t'}^{L}]=0
$ if $t'$ is not contained in $s$, and
therefore, putting $e_{s}^{L}:= \sum_{t\mapsto s}e_{t}^{L}$, we get that
$$[R\Gamma_{c}(Y_{\mathbf{P}},\bQl)e_{s}^{L}]= 
[e_{s}^{G}R\Gamma_{c}(Y_{\mathbf{P}},\bQl)].$$
In the particular case that $\mathbf{P}$ itself is $F$-stable, and
denoting by $e_{U}$ the idempotent which averages along the group
$U=\mathbf{U}^{F}$ of rational points of the
radical $\mathbf{U}$ of $\mathbf{P}$, we get the equalities 
\ini\begin{equation}
 e_{s}^{G} e_{U} = e_{s}^{G} e_{U} e_{s}^{L} = e_{U} e_{s}^{L} =: e_{s}^{P},\label{compat2}
 \end{equation}
showing in particular that $e_{s}^{G} e_{U}$ is a central idempotent
in $\bQl[P]$.
In the same way, when $s$ consists of $\ell'$-elements, we have equalities
$$ e_{s,\oZl}^{G} e_{U} = e_{s,\oZl}^{G} e_{U} e_{s,\oZl}^{L} = e_{U} e_{s,\oZl}^{L}=:e_{s,\oZl}^{P}$$
which show that $e_{s,\oZl}^{G} e_{U}$ is a central idempotent in $\bZl[P]$.

\alin{Morita equivalences} \label{Morita}
As above, let $\mathbf{L}$ be an $F$-stable Levi subgroup contained in a
  parabolic subgroup $\mathbf{P}$ of $\mathbf{G}$. We choose an $F^{*}$-stable embedding 
$\mathbf{L}^{*}\injo \mathbf{G}^{*}$ dual to $\mathbf{L}\injo \mathbf{G}$.
Let $t$ be a  semisimple class in $L^{*}$ and $s$ its image in $G^{*}$. 
Let also $\Lambda$ denote either $\oQl$ or $\oZl$, and \emph{suppose that the order of an element
of $t$ is invertible in $\Lambda$}.
The following result is due to Lusztig \cite[Prop 6.6]{LusztigGreen} for $\Lambda=\oQl$
and to Bonnaf\'e and Rouquier \cite[Thm B']{BonRou} for $\Lambda=\oZl$.

\begin{fac}
Assume that $\mathbf{L}^{*}$ contains the centralizer
  $C_{\mathbf{G}^{*}}(t)$ of any element in $t$. Then the complex
  $R\Gamma_{c}(Y_{\mathbf{P}},\Lambda)e_{t,\Lambda}^{L}$ is  
 concentrated in degree $d={\rm dim}(Y_{\mathbf{P}})$ and 
the $(\Lambda[{G}]e_{s,\Lambda}^{G},\Lambda[{L}]e_{s,\Lambda}^{L}))$-bimodule
$H^{d}_{c}(Y_{\mathbf{P}},\Lambda)e_{t,\Lambda}^{L}$ induces a Morita equivalence, in the
sense  that the two adjoint functors
$$ \application{}{\Lambda[L]e_{t,\Lambda}^{L}\hbox{-\rm Mod}}
{\Lambda[{G}]e_{s,\Lambda}^{G}\hbox{-\rm Mod}}
{M}{H^{d}_{c}(Y_{\mathbf{P}},\Lambda)e_{t,\Lambda}^{L}
  \otimes_{\Lambda[L]} M}
\hbox{ and } $$
$$\application{}{\Lambda[{G}]e_{s,\Lambda}^{G}\hbox{-\rm Mod}}
{\Lambda[L]e_{t,\Lambda}^{L}\hbox{-\rm Mod}}
{N}{\Hom_{\Lambda[G]}\left(H^{d}_{c}(Y_{\mathbf{P}},\Lambda)e_{t,\Lambda}^{L}, N\right)}
$$
are equivalences of categories.
\end{fac}

In particular these functors induce bijections from $\EC(L,t)$ to $\EC(G,s)$, and from
$\EC_{\ell}(L,t)$ to $\EC_{\ell}(G,s)$ when appropriate.
Note that in the case $\Lambda=\oQl$, the Morita equivalence follows from this
bijection property, but of course things are much more complicated for $\Lambda=\oZl$.

Note also that in the case where $\mathbf{L}^{*}=C_{\mathbf{G}^{*}}(t)$, the
  element $t$ provides a character $\hat t$ of $L$, and twisting by
  $\hat t$ induces an isomorphism between $\bZl[L]e_{1,\oZl}^{L}$ and
  $\bZl[L]e_{t,\oZl}^{L}$, and a bijection $\EC(L,t)\mapsto \EC(L,1)$. 
Therefore, for groups $\mathbf{G}$ such that  all  centralizers of semisimple
elements in $\mathbf{G}^{*}$ are Levi subgroups, 
 these equivalences reduce the study of $\EC(G,s)$ and
$\bZl[G]e_{s,\oZl}^{G}$ to that of $\EC(C_{\mathbf{G}^{*}}(s)^{*},1)$
and $\bZl[C_{\mathbf{G}^{*}}(s)^{*}]e_{1,\oZl}^{G}$.

\alin{Dependence on the parabolic subgroup} In our later application to $p$-adic groups, it will be
important to know that the equivalences of \ref{Morita} are ``naturally independent'' of the
parabolic sugroup $\mathbf{P}$ that contains  $\mathbf{L}$ as a Levi factor. Note that this is far
from obvious since even the dimension of $Y_{\mathbf{P}}$ depends on $\mathbf{P}$.

We need here an apparently more general notion of ``Deligne-Lusztig varieties''. Namely, start with a collection
$\mathbf{P}_{1},\cdots, \mathbf{P}_{r}$ of parabolic subgroups with Levi component $\mathbf{L}$, and
define
$$Y_{\mathbf{P}_{1},\cdots, \mathbf{P}_{r}}
:= 
\left\{\begin{array}{c}
 (g_{1}\mathbf{U}_{1},\cdots,g_{r}\mathbf{U}_{r})\in 
\mathbf{G}/\mathbf{U}_{1}\times\cdots\times\mathbf{G}/\mathbf{U}_{r} \hbox{ such that}\\
g_{i}^{-1}g_{i+1}\in \mathbf{U}_{i}\cdot \mathbf{U}_{i+1} \hbox{ for } \,i=1,\cdots,r, \hbox{ and} 
 \\ 
g_{r}^{-1}F(g_{1})\in \mathbf{U}_{r}\cdot {{F}(\mathbf{U}_{1})}
\end{array}\right\}
$$ 
This is again a variety over $\o\FM$ with a left action of
$\mathbf{G}^{F}$ given by $(\gamma,(g_{1}\mathbf{U}_{1},\cdots, g_{r}\mathbf{U}_{r}))\mapsto 
(\gamma g_{1}\mathbf{U}_{1},\cdots, \gamma g_{r}\mathbf{U}_{r})$ 
and a right action of $\mathbf{L}^{F}$ given by  $((g_{1}\mathbf{U}_{1},\cdots,
g_{r}\mathbf{U}_{r}),\delta) \mapsto (g_{1}\delta\mathbf{U}_{1},\cdots, g_{r}\delta\mathbf{U}_{r})$.

Now let $\mathbf{P}$, $\mathbf{P'}$ be two parabolic subgroups with Levi component
$\mathbf{L}$. Consider the closed subvariety 
$$ \YC_{\mathbf{P},\mathbf{P'}}:=
\{(g\mathbf{U},g'\mathbf{U'})\in Y_{\mathbf{P},\mathbf{P'}},\,
g^{-1}F(g) \in \mathbf{U}\cdot {^{F}\mathbf{U}}\}.$$
We have a diagram

$$\xymatrix{ Y_{\mathbf{P},\mathbf{P'}}
 & 
\YC_{\mathbf{P},\mathbf{P'}}
\ar[r]^{\pi} \ar@{_{(}->}[l]_{\iota} 
& Y_{\mathbf{P}}
}$$
where
 $\pi$ is the first projection $(g\mathbf{U},g'\mathbf{U'})\mapsto g\mathbf{U}$.
It 
follows from \cite[Corollary 5.3]{cdvdl2}
that  $\pi$ is an affine bundle of relative dimension 
$\delta:=\frac{1}{2}(\dim(Y_{\mathbf{P},\mathbf{P'}}) -
\dim(Y_{\mathbf{P}}))$. Therefore the map $\pi_{!}$ induces an
isomorphism $R\Gamma_{c}(\YC_{\mathbf{P},\mathbf{P'}},\oZl)
\simto
R\Gamma_{c}(Y_{\mathbf{P}},\oZl)[-2\delta](-\delta)$. Composing
with $\iota_{!}$ we get in particular a morphism
\ini\begin{equation}
  \label{PP'} \Psi^{(2)}_{\mathbf{P},\mathbf{P'}}:\,
  R\Gamma_{c}^{\dim}\left(Y_{\mathbf{P},\mathbf{P'}},\oZl\right)
\To{}
  R\Gamma_{c}^{\dim}\left(Y_{\mathbf{P}},\oZl\right).
\end{equation}
where we use the notation 
$R\Gamma_{c}^{\dim}(Y,\oZl):= R\Gamma_{c}(Y,\oZl)[\dim(Y)](\frac 1 2\dim(Y))$
 for a smooth variety $Y$ of pure dimension. In the same way we get a morphism
$\Psi^{(2)}_{\mathbf{P'},F(\mathbf{P})}$.
Since the morphism of varieties
 $Y_{\mathbf{L}\subset \mathbf{P},\mathbf{P'}}^{\mathbf{G}} \To{} Y_{\mathbf{L}\subset \mathbf{P'},
   {{F}(\mathbf{P})}}^{\mathbf{G}}$ defined by
$(g\mathbf{U},g'\mathbf{U'})\mapsto (g'\mathbf{U'},F(g\mathbf{U}))$ is known to induce 
 an equivalence of \'etale sites,  we see that $\Psi^{(2)}_{\mathbf{P'},F(\mathbf{P})}$
provides a morphism
\ini\begin{equation}
  \label{P'P} \Psi^{(1)}_{\mathbf{P},\mathbf{P'}}:\,
  R\Gamma_{c}^{\dim}\left(Y_{\mathbf{P},\mathbf{P'}},\oZl\right)
\To{}
  R\Gamma_{c}^{\dim}\left(Y_{\mathbf{P'}},\oZl\right).
\end{equation}

Let us now choose a $F$-compatible embedding $\mathbf{L}^{*}\injo \mathbf{G}^{*}$. We can transfer
$\mathbf{P}$ and $\mathbf{P'}$ to  dual
parabolic subgroups $\mathbf{P}^{*}$ and $\mathbf{P'}^{*}$ with Levi component $\mathbf{L}^{*}$.
The following result follows from  
Theorem 6.2 of \cite{cdvdl2}\footnote{The reader will find in \cite{indep} a shorter account of this
  result, with notation closer to those of this article}.
\begin{fac}\label{change_par} Let $\Lambda$ denote either $\oZl$ or $\oQl$.
  Let $t$ be a semisimple element in $L^{*}$ with order invertible in $\Lambda$.
 Assume that $C_{\mathbf{G}^{*}}(t)\cap
  \mathbf{P}^{*}=C_{\mathbf{G}^{*}}(t)\cap \mathbf{P'}^{*}$. 
Then the morphisms $\Psi_{\mathbf{P},\mathbf{P'}}^{(2)}$ and $\Psi_{\mathbf{P},\mathbf{P'}}^{(1)}$
induce isomorphisms
$$ R\Gamma_{c}^{\rm dim}(Y_{\mathbf{P}},\Lambda)e_{t,\Lambda}^{L} 
\buildrel\hbox{\tiny{$\sim$}}\over\longleftarrow 
R\Gamma_{c}^{\rm dim}(Y_{\mathbf{P},\mathbf{P'}},\Lambda)e_{t,\Lambda}^{L}
\simto
R\Gamma_{c}^{\rm dim}(Y_{\mathbf{P'}},\Lambda)e_{t,\Lambda}^{L}.$$
\end{fac}

\alin{General linear groups} \label{gln} We assume here that $G$ is
a product of general linear groups over extensions of $\FM$.
Then we have the following additional properties :
\begin{enumerate}
\item 
\emph{$\EC(G,s)$ contains a cuspidal representation if and only if $s$ is
elliptic in $G^{*}$} (by which we mean 
that $s$ does not intersect any
rational parabolic subgroup). 
Moreover in this case, $\EC(G,s)$ consists of one
element, so that we get a bijection
\begin{center}
$\{$elliptic conjugacy classes in $G^{*}\} \leftrightarrow \{$classes
of cuspidal irreps of $G$ over $\bQl$\}.
\end{center}
{\small To see why these properties hold, note first that if $s$ is
  elliptic then for any rational parabolic subgroup $P$ we have
  $e_{s}^{G}e_{U_{P}}=0$ by 
  (\ref{compat2}), thus $\EC(G,s)$ consists
  of cuspidal representations. Picking up an irrep in each $\EC(G,s)$
  we thus get an injection from the set of elliptic conjugacy classes
  to that of cuspidal representations. But a counting argument based
  on Green's classification shows that we get all cuspidals this way.}

\item If $s$ is not elliptic, there is a rational parabolic
  subgroup in $\mathbf{G}$ whose dual intersects $s$. Choose 
 $\mathbf{P}$ minimal for this property, and let $\mathbf{L}$ be a Levi
 complement and $t\in \mathbf{L}^{*}\cap s$. Thus $t$ is elliptic in
 $L^{*}$ and corresponds to a cuspidal irrep $\sigma$ of $L$ by i)
 above. We get in this way a bijection $s\mapsto [L,\sigma]$
\begin{center}
$\{$s/simple conjugacy classes in $G^{*}\} \leftrightarrow
\{$$G$-conj. classes of cuspidal pairs $(L,\sigma)$ over $\bQl$\}.
\end{center}
Moreover, by the compatibility with induction of paragraph
\ref{compatinduc},  $\EC(G,s)$ contains all irreducible subquotients of
  $\RC_{\mathbf{L}\subset \mathbf{P}}^{\mathbf{G}}(\sigma)=
  \ind{P}{G}{\sigma}$.  
But then, comparing the partition of
  $\Irr{\bQl}{G}$ in rational series, and that given by the cuspidal
  support, we see that
  \begin{center}
    $\EC(G,s) = {\rm JH} ( \ind{P}{G}{\sigma} ).$ 
  \end{center}

\item Assume now that $s$ consists of $\ell'$-elements. Then
$\EC_{\ell}(G,s)$ contains a \emph{super}cuspidal representation if and only if $s$ is
elliptic in $G^{*}$. 
Moreover in this case, $\EC_{\ell}(G,s)$ consists of one
element, so that we get a bijection 
\begin{center}
$\{\ell$-regular elliptic conjugacy classes in $G^{*}\}\leftrightarrow
\{$supercuspidal $\oFl$-representations of $G\}$.
\end{center}
{\small To see these properties, observe first that since the
  centralizer of an elliptic element is a torus, the Morita
  equivalence of paragraph \ref{Morita} shows that $\EC_\ell(G,s)$ consists of
  one representation, which by the way is cuspidal. If it were not
  supercuspidal, it would occur as a subquotient of some induced
  representation hence would belong to some $\EC_\ell(G,s')$ with $s'$
  not elliptic, a contradiction. Thus we get an injective map from
  $\ell$-regular elliptic classes to supercuspidals. Again a counting
  argument, relying this time on James' classification, shows it is
  also surjective.
}

\item If $s$ is not elliptic, choose $\mathbf{P}$ minimal among  rational parabolic
  subgroups  whose dual intersects $s$, and choose $\mathbf{L}$  a
  rational Levi
 complement and $t\in \mathbf{L}^{*}\cap s$. Thus $t$, being elliptic
 and $\ell$-regular in
 $L^{*}$, corresponds to a supercuspidal irrep $\sigma$ of $L$ over $\oFl$. 
We get in this way a bijection 
\begin{center}
$\{$s/simple $\ell$-regular conjugacy classes in $G^{*}\}$ \\ $\leftrightarrow
\{$$G$-conj. classes of supercuspidal pairs $(L,\sigma)$ over $\oFl$\}.
\end{center}
Moreover, by the compatibility with induction of paragraph
\ref{compatinduc},  $\EC_\ell(G,s)$ contains all irreducible subquotients of
  $\RC_{\mathbf{L}\subset \mathbf{P}}^{\mathbf{G}}(\sigma)=
  \ind{P}{G}{\sigma}$.  
But then, comparing the partition of
  $\Irr{\oFl}{G}$ as a union of $\EC_\ell(G,s)$, and that given by the supercuspidal
  support, we see that
  \begin{center}
    $\EC_\ell(G,s) = {\rm JH} ( \ind{P}{G}{\sigma} ).$ 
  \end{center}
\end{enumerate}

\subsection{General linear groups in a coordinate-free setting}
\label{sec:gener-line-groups}

The setting here is the following.
Let $V$ be a vector space over $\FM$ of finite dimension $n$.
We put $G:=\Aut_{\FM}(V)$ and we also denote by $\mathbf{G}$ the
corresponding algebraic group over $\o\FM$ and $F$
the Frobenius associated with the $\FM$-structure, so that $G=\mathbf{G}^{F}$.
Further,  let $\FM'$ be a finite subfield of $\o\FM$ containing $\FM$, and
let $n', V', G', \mathbf{G}', F'$ be as
above, but relative to $\FM'$. 
We assume that $|\FM|^{n}=|\FM'|^{n'}$.
We may then see $G'$ as the rational points of some Levi subgroup in
$\mathbf{G}$, and thus use the theory of the previous section to
go from representations of $G'$ to representations of $G$. 

\alin{Varieties}\label{varietiesgln}
Fix an $\FM$-linear isomorphism  $\iota:\, V'\simto V$.
 This induces a group embedding $\iota_{*}:G'\injo G$, which
comes as the $F$-fixed part of an embedding
$\iota_{*}:\,{\rm Res}_{\FM'|\FM}\mathbf{G'}\injo \mathbf{G}$ of algebraic
groups over $\FM$.
The image of this embedding is a Levi subgroup $\mathbf{L}_{\iota}$ of
$\mathbf{G}$ such that $\mathbf{L}_{\iota}^{F}= \iota_{*}(G')$.
More precisely, on the $\o\FM$-vector space $\o
V:=\o\FM\otimes_{\FM}V$, we get from $\iota_{*}$ a $\FM$-linear
action of the field $\FM'$. Thanks to this action, we have 
a direct sum decomposition

$$\o V=\bigoplus_{i=0}^{e-1} \o V_{i}, \hbox{ with } \o V_{i}=\{v\in \o V, \forall a\in\FM',
 \iota_{*}(a)v=F^{i}(a)v\},$$ and this identifies
$\mathbf{L}_{\iota}$ with $\prod_{i}\Aut_{\o\FM}(\o V_{i}).$ 
We denote by $\mathbf{P}_{\iota}$ the parabolic subgroup with Levi
$\mathbf{L}_{\iota}$ associated to the flag 
$$\bar V\supset \bigoplus_{i\geq 1}\o V_{i} \supset \bigoplus_{i\geq
  2}\o V_{i} \supset\cdots\supset \{0\}.$$ Unlike $\mathbf{L}_{\iota}$, this parabolic subgroup is only
defined over $\FM'$, not over $\FM$. Let $\mathbf{U}_{\iota}$ be the unipotent
radical of $\mathbf{P}_{\iota}$ and consider the 
variety
$$ Y_{\iota,V'}^{V} := Y_{\mathbf{P}_{\iota}}.$$
It has a right action of
$G=\mathbf{G}^{F}$ given by $(\gamma,g\mathbf{U}_{\iota})\mapsto \gamma
g\mathbf{U}_{\iota}$ and a left action of $G'$ given by  $(g\mathbf{U}_{\iota},\gamma')\mapsto g\iota_{*}(\gamma')\mathbf{U}_{\iota}$.

\alin{Idempotents and  bimodules} \label{idempotentsgln}
As usual we denote by $\Lambda$ either $\oQl$ or $\oZl$.
We now fix a semisimple conjugacy class $s'$ in $\GL_{n'}(\FM')$ with invertible order in $\Lambda$. This determines a conjugacy
class of rational elements in $\mathbf{G}'$ and thus a central idempotent
$e_{s',\Lambda}^{G'}\in\Lambda[G']$. 
Moreover, choosing an $\FM$-basis of $\FM'$, we get a conjugacy class $s$ in $\GL_{n}(\FM)$, which
does not depend on the choice of this basis. Then, similarly we have an idempotent
$e_{s,\Lambda}^{G}\in\Lambda[G]$.
\begin{center}
  \emph{We assume that the centralizer of an element of $s'$ in $\GL_{n}(\FM)$ is
    contained in $\GL_{n'}(\FM')$}.
\end{center}
Then, as recalled in Fact \ref{Morita}, the $(\Lambda[G],\Lambda[G'])$-bimodule
$$ \PC_{\iota,V'}^{V,s'}=\PC_{\iota,V'}^{V,s', \Lambda} := H^{\dim}_{c}(Y_{\iota,V'}^{V},\Lambda)e_{s',\Lambda}^{G'}
$$
induces a Morita equivalence between $\Lambda[G']e_{s',\Lambda}^{G'}$ and
$\Lambda[G]e_{s,\Lambda}^{G}$.

\alin{Flags and parabolic induction} \label{parabolicgln}
Let $\VC'=(V'_{0}=V'\supset V'_{1}\supset \cdots \supset V'_{r}=0)$ be a flag of $\FM'$-subvector
spaces in $V'$, and let $\mathbf{P}'\subset \mathbf{G'}$ be the associated parabolic
subgroup, with radical $\mathbf{U'}$ and Levi quotient $\mathbf{M'}$, all $F'$-stables.
Denote by $\VC:=\iota_{*}\VC'$ its image as a flag in $V$, and let $\mathbf{P}\subset
\mathbf{G}$, $\mathbf{U}$, $\mathbf{M}$ be similarly defined. We note that, upon
identifying ${\rm Res}_{\FM'|\FM}(\mathbf{G'})$ to a Levi subgroup of $\mathbf{G}$ through
$\iota$, we have the equality $\mathbf{P}\cap {\rm Res}_{\FM'|\FM}(\mathbf{G'})={\rm
  Res}_{\FM'|\FM}(\mathbf{P'})$. Also, $\iota$ identifies ${\rm
  Res}_{\FM'|\FM}(\mathbf{M'})$ to an $F$-stable Levi subgroup of $\mathbf{M}$ and the
projection of $\mathbf{P}_{\iota}\cap \mathbf{P}$ to $\mathbf{M}$ provides a parabolic
subgroup $\mathbf{P^{M}_{\iota}}$ of $\mathbf{M}$ with Levi component ${\rm
  Res}_{\FM'|\FM}\mathbf{M'}$.
We then put 
$$ Y_{\iota,\VC'}^{V}:=Y_{\mathbf{P^{M}_{\iota}}} = \prod_{j} Y_{\iota,
  V'_{j}/V'_{j+1}}^{V_{j}/V_{j+1}},$$
which is a $M\times M'$-variety that we will also consider as a $P\times P'$-variety (since
$\mathbf{P}$ and ${\rm Res}(\mathbf{P}')$ are $F$-stable, we can use  the notation
$P=\mathbf{P}^{F}$ and $P'=\mathbf{P'}^{F'}$).
Recall our notation $e_{s'}^{P'}$ from (\ref{compat2}) and put 
$$ \PC_{\iota,\VC'}^{V,s'}= \PC_{\iota,\VC'}^{V,s',\Lambda}:=
H^{\dim}_{c}(Y_{\iota,\VC'}^{V},\Lambda)e_{s',\Lambda}^{P'}. 
$$
As above, it induces a Morita equivalence between $\Lambda[P']e_{s',\Lambda}^{P'}$ and
$\Lambda[P]e_{s',\Lambda}^{P}$ (note 
that everything boils down to the Levi quotients $M,M'$ where the situation is a product
of instances of the last paragraph).

Later, we will need the following functoriality property of this parabolic
construction. If  $\gamma'\in G'$ with image $\gamma:=\iota_{*}(\gamma')\in G$, then the
conjugation map by $\gamma$ in $\mathbf{G}$ 
$${\rm Int}_{\gamma}:\,Y_{\iota,\VC'}^{V}\To{}Y_{\iota,\gamma'(\VC')}^{V}, \;
m\mathbf{U^{M}_{\iota}}\mapsto  \gamma m\mathbf{U}^{M}_{\iota}\gamma^{-1}$$
 induces an isomorphism of
  varieties which is equivariant with respect
  to  ${\rm Int}_{\gamma}\times {\rm Int}_{\gamma'} :\, P\times P' \simto  {^{\gamma}P}\times {^{\gamma'}}P'$.
This induces in turn an isomorphism
\ini\begin{equation}
{\rm Int}_{\gamma,\VC'}:\,\PC_{\iota,\VC'}^{V,s'} \simto  \PC_{\iota, \gamma'(\VC')}^{V,s'}\label{foncto2}
\end{equation}
with the same equivariance property. Note that when $\gamma'\in P'$, this isomorphism is
given by the natural action of $(\gamma, {\gamma'}^{-1})\in P\times P'$ on $Y_{\iota,\VC'}^{V}$.

The following result is crucial for the main theorem of the paper.

\begin{prop} \label{prop_finite}
There exists a family of isomorphisms of $(\Lambda P_{2},\Lambda P'_{1})$-bimodules
\ini\begin{equation}
\phi_{\VC'_{1}}^{\VC'_{2}}:\,\Lambda P_{2}\otimes_{\Lambda P_{1}}{\PC}_{\iota,\VC'_{1}}^{V,{s'}}  \simto
{\PC}_{\iota,\VC'_{2}}^{V,{s'}}e_{U'_{1}}.
\label{parab1}
\end{equation}
indexed by pairs of flags $\VC_{1}',\VC_{2}'$ in $V'$ such that $\VC_{1}'$ refines
$\VC'_{2}$  (so that $P_{1}\subset P_{2}$ and $P'_{1}\subset
P'_{2}$)
and which satisfy the following properties.
\begin{enumerate}
\item (functoriality) For any $\gamma'\in G'$ with image $\gamma=\iota_{*}(\gamma')$, we
  have $\phi_{\gamma'\VC'_{1}}^{\gamma'\VC'_{2}} = {\rm
    Int}_{\gamma,\VC'_{2}}\circ\phi_{\VC'_{1}}^{\VC'_{2}}\circ {\rm
    Int}_{\gamma,\VC'_{1}}^{-1} $ 
\item (transitivity) If $\VC'_{3}$ is a coarser flag than $\VC'_{2}$, then 
the following diagram commutes :
$$
\xymatrixcolsep{2cm}\xymatrix{
\Lambda P_{3}\otimes_{\Lambda P_{1}} {\PC}_{\iota,\VC'_{1}}^{V,{s'}}  \ar[r]^{\phi_{\VC'_{1}}^{\VC'_{3}}}
\ar@{=}[d] &
{\PC}_{\iota,\VC'_{3}}^{V,{s'}}e_{U'_{1}} \\
\Lambda P_{3}\otimes_{\Lambda P_{2}}\Lambda P_{2}
\otimes_{\Lambda P_{1}}{\PC}_{\iota,\VC'_{1}}^{V,{s'}}  
\ar[r]^{\id\otimes\phi_{\VC'_{2}}^{\VC'_{3}}} &
\Lambda P_{3}\otimes_{\Lambda P_{2}}
{\PC}_{\iota,\VC'_{2}}^{V,{s'}}e_{U'_{1}} \ar[u]_{\phi_{\VC'_{2}}^{\VC'_{3}}}
}$$

\end{enumerate}
\end{prop}

\begin{proof} 

\emph{Step 1. Construction of $\phi_{\VC'_{1}}^{\VC'_{2}}$.} For the sake of clarity, we first assume that $\VC'_{2}$ is
the trivial flag $V'$ and we abbreviate $\VC'=\VC'_{1}$.
We take up the notation of subsection \ref{parabolicgln}.
  Consider the two parabolic subgroups of $\mathbf{G}$
$$ \mathbf{Q}:= (\mathbf{P}_{\iota} \cap \mathbf{P}) \mathbf{U}_{\iota} 
\,\,\hbox{ and } \,\, \mathbf{Q'}:= (\mathbf{P}_{\iota} \cap \mathbf{P})\mathbf{U} .
$$
Both have Levi quotient isomorphic to ${\rm Res}_{\FM'|\FM}(\mathbf{M'})$.  In fact,
putting $\o V_{i,j}:=\o V_{i} \cap \o{\iota(V'_{j})}$, the parabolic subgroup $\mathbf{Q}$,
resp. $\mathbf{Q'}$, 
corresponds to the flag obtained by ordering pairs ${i,j}$ according to the
lexicographical order on the couple $(i,j)$, resp. on the couple $(j,i)$.

Let us temporarily choose a splitting of the flag $\VC'$. This gives rise to splittings
$\mathbf{M'}\injo \mathbf{P'}$ and $\mathbf{M}\injo \mathbf{P}$, that fit in a diagram
of Levi and parabolic subgroups
$$
\xymatrix{
{\rm Res}_{\FM'|\FM}(\mathbf{G'}) \ar@{^{(}-}[r]^-{\mathbf{P_{\iota}}} & \mathbf{G} \\
{\rm Res}_{\FM'|\FM}(\mathbf{M'}) \ar@{^{(}-}[r]_-{\mathbf{P_{\iota}}\cap M} 
\ar@{^{(}-}[u]^{{\rm Res}(\mathbf{P'})}
& \mathbf{M} \ar@{^{(}-}[u]_{\mathbf{P}} 
}$$
so that $\mathbf{Q}$ corresponds to the upper way of composing parabolic subgroups in this
diagram while $\mathbf{Q'}$ corresponds to the lower way.
A good point about general linear groups is that they are isomorphic to their own duals. In fact
the same diagram as above may serve as a dual diagram to itself.
Now, let $t'$ be a semisimple element in the intersection of our semisimple class $s'$ and 
$M' = \mathbf{M'}^{*,F^{*}}$. By hypothesis, the centralizer $C_{\mathbf{G}}(t')$ is contained in
${\rm Res}(\mathbf{G}')$. Moreover, we have $\mathbf{Q}\cap {\rm Res}(\mathbf{G}')=
\mathbf{Q'}\cap {\rm Res}(\mathbf{G}')={\rm Res}(\mathbf{P}')$. Therefore we may apply
Fact \ref{change_par} which provides us with an isomorphism
$$ \Psi_{\mathbf{Q},\mathbf{Q'}}^{(1)}\circ (\Psi_{\mathbf{Q},\mathbf{Q'}}^{(2)}):\,
R\Gamma_{c}^{\rm dim}(Y_{\mathbf{Q'}},\Lambda)e_{{t'},\Lambda}^{M'} \simto
R\Gamma_{c}^{\rm dim}(Y_{\mathbf{Q}},\Lambda)e_{{t'},\Lambda}^{M'}
.$$
Summing on all $t'$ as above up to $M'$-conjugacy, we get an isomorphism in $D^{b}(\Lambda
G\otimes_{\Lambda} \Lambda {M'}^{\rm opp})$
\ini\begin{equation}
\Phi_{\mathbf{Q'}|\mathbf{Q}}:=  \Psi_{\mathbf{Q},\mathbf{Q'}}^{(1)}\circ (\Psi_{\mathbf{Q},\mathbf{Q'}}^{(2)}):\,
R\Gamma_{c}^{\rm dim}(Y_{\mathbf{Q'}},\Lambda)e_{{s'},\Lambda}^{M'} \simto
R\Gamma_{c}^{\rm dim}(Y_{\mathbf{Q}},\Lambda)e_{{s'},\Lambda}^{M'}
.\label{change_Q}
\end{equation}
It is easily checked that these isomorphisms do not depend on the choice of a splitting of $\VC'$.

Now, the transitivity for DL varieties tells us that the multiplication map induces an isomorphism
$$ Y_{\iota,V'}^{V} \times^{G'} G'/U' = 
Y_{\mathbf{P}_{\iota}} \times^{{\rm Res}(\mathbf{G'})^{F}} Y_{{\rm Res}(\mathbf{P}')} \simto
Y_{\mathbf{Q}}.$$ 
Similarly, upon choosing a splitting of the flag $\VC'$, giving rise to splittings
$\mathbf{M'}\injo \mathbf{P'}$ and $\mathbf{M}\injo \mathbf{P}$, we get an isomorphism
$$  G\times^{P}Y_{\iota,\VC'}^{V}
= G/U \times^{M} Y_{\iota,\VC'}^{V}
=   Y_{\mathbf{P}}\times^{M} Y_{\mathbf{P^{M}_{\iota}}} \simto Y_{\mathbf{Q'}}$$
which does not  depend on the splitting.

Taking cohomology and  applying idempotents, the isomorphism  (\ref{change_Q}) becomes our desired
isomorphism
\ini\begin{equation}
\phi_{\VC'}^{V'}:\, \Lambda G\otimes_{\Lambda P} \PC_{\iota,\VC'}^{V,{s'} } \simto \PC_{\iota, V'}^{V,{s'}}e_{U'}.\label{isom1}
\end{equation}

Let us now treat a general pair $(\VC'_{1},\VC'_{2})$ as in the proposition. Let $t'$ be any
conjugacy class in $P'_{2}$ contained in $s'$. Then, repeating the above discussion, we get an
isomorphism similar to (\ref{isom1})
$$
 \Lambda P_{2} \otimes_{\Lambda P_{1}} 
H^{\dim}_{c}(Y_{\iota,\VC'_{1}})e_{{t'},\Lambda}^{P'_{1}} \simto 
H^{\dim}_{c}(Y_{\iota,\VC'_{2}})e_{{t'},\Lambda}^{P'_{2}}e_{U'_{1}}.
$$
Indeed the whole situation is a product of instances of the case just treated. Now, summing over all
$t'$ provides us with the desired isomorphism 
\ini\begin{equation}
\label{isom2} \phi_{\VC'_{1}}^{\VC'_{2}}:\,
\Lambda P_{2} \otimes_{\Lambda P_{1}} \PC_{\iota,\VC'_{1}}^{V,{s'}} \simto 
\PC_{\iota,\VC'_{2}}^{V,{s'}}e_{U'_{1}}.
\end{equation}

\medskip

\emph{Step 2. Property i).} 
This follows quite clearly from the geometric origin of $\phi_{\VC'_{1}}^{\VC'_{2}}$. Indeed we have used
various geometric relations between DL varieties that all have a clear
functorial behavior under conjugacy by $\iota_{*}(\gamma')$, the most involved point being the isomorphism
(\ref{change_Q}) whose geometric construction was recalled in subsection \ref{change_par}.

\medskip
\emph{Step 3. Property ii).} To simplify the discussion a bit we will assume that $\VC'_{3}=V'$. As
pointed out in Step 1, this is not a real loss of generality. 
Let us denote by $\mathbf{Q}_{1}$ and $\mathbf{Q}'_{1}$ the parabolic subgroups involved
in Step 1 and associated to $\VC'_{1}$. Further, let us introduce a third parabolic subgroup with
Levi quotient ${\rm Res}(\mathbf{M'}_{1})$ :
$$\mathbf{Q}_{1}^{2}:= (\mathbf{P}_{\iota}\cap \mathbf{P}_{1})(\mathbf{U}_{\iota}\cap
\mathbf{P}_{2}) \mathbf{U}_{2}.$$
Upon choosing splittings of the flags $\VC'_{1}$ and $\VC'_{2}$, this parabolic subgroup
is the ``intermediate'' way to compose parabolic subgroups in the following diagram
$$
\xymatrix{
{\rm Res}_{\FM'|\FM}(\mathbf{G'}) \ar@{^{(}-}[r]^-{\mathbf{P_{\iota}}} & \mathbf{G} \\
{\rm Res}_{\FM'|\FM}(\mathbf{M'}_{2}) \ar@{^{(}-}[r]_-{\mathbf{P_{\iota}}\cap \mathbf{M}_{2}} 
\ar@{^{(}-}[u]^{{\rm Res}(\mathbf{P}'_{2})}
& \mathbf{M}_{2} \ar@{^{(}-}[u]_{\mathbf{P}_{2}} \\
{\rm Res}_{\FM'|\FM}(\mathbf{M'}_{1}) \ar@{^{(}-}[r]_-{\mathbf{P_{\iota}}\cap \mathbf{M}_{1}} 
\ar@{^{(}-}[u]^{{\rm Res}(\mathbf{P}'_{1}\cap \mathbf{M}'_{2})}
& \mathbf{M}_{1} \ar@{^{(}-}[u]_{\mathbf{P}_{1}\cap \mathbf{M}_{2}}
}.$$
Now, looking at the construction in Step 1, we see that what needs to be proved is the
equality
\ini\begin{equation}
 \Phi_{\mathbf{Q}'_{1}|\mathbf{Q}_{1}} =
\Phi_{\mathbf{Q}_{1}^{2}|\mathbf{Q}_{1}} \circ \Phi_{\mathbf{Q}'_{1}|\mathbf{Q}_{1}^{2}}.\label{telesc}
\end{equation}
involving isomorphisms of type (\ref{change_Q}).
Consider the following diagram 
$$\xymatrix{
&&
R\Gamma_{c}^{\rm dim}(Y_{\mathbf{Q}_{1},\mathbf{Q}'_{1}}) 
\ar@{<-}[d]_{\Psi_{\mathbf{Q}_{1},\mathbf{Q}_{1}^{2},\mathbf{Q}'_{1}}^{(2)}}
&& \\
 & \framebox{1}  & 
R\Gamma_{c}^{\rm dim}(Y_{\mathbf{Q}_{1},\mathbf{Q}_{1}^{2},\mathbf{Q}'_{1}})
 & \framebox{2} & \\
& R\Gamma_{c}^{\rm dim}(Y_{\mathbf{Q}_{1}^{2},\mathbf{Q}'_{1}})
\ar@{<-}[ru]^{\Psi_{\mathbf{Q}_{1},\mathbf{Q}_{1}^{2},\mathbf{Q}'_{1}}^{(1)}}
& \framebox{3} &
R\Gamma_{c}^{\rm dim}(Y_{\mathbf{Q}_{1},\mathbf{Q}_{1}^{2}}) 
\ar@{<-}[lu]_{\Psi_{\mathbf{Q},\mathbf{Q}_{1}^{2},\mathbf{Q}'_{1}}^{(3)}} & \\
R\Gamma_{c}^{\rm dim}(Y_{\mathbf{Q}'_{1}})
\ar@{<-}[ru]_{\Psi_{\mathbf{Q}_{1}^{2},\mathbf{Q}'_{1}}^{(1)}}
\ar@{<-}@/^3pc/[rruuu]^{\Psi_{\mathbf{Q}_{1},\mathbf{Q}'_{1}}^{(1)}}
&  &
 R\Gamma_{c}^{\rm dim}(Y_{\mathbf{Q}_{1}^{2}})
\ar@{<-}[ru]^{\Psi_{\mathbf{Q}_{1},\mathbf{Q}_{1}^{2}}^{(1)}}
\ar@{<-}[lu]_{\Psi_{\mathbf{Q}_{1}^{2},\mathbf{Q}'_{1}}^{(2)}}
&& 
R\Gamma_{c}^{\rm dim}(Y_{\mathbf{Q}_{1}})
\ar@{<-}[lu]^{\Psi_{\mathbf{Q}_{1},\mathbf{Q}_{1}^{2}}^{(2)}}
\ar@{<-}@/_3pc/[lluuu]_{\Psi_{\mathbf{Q}_{1},\mathbf{Q}'_{1}}^{(2)}} 
}$$
When we apply the idempotent $e_{{s'},\Lambda}^{M'_{1}}$ to this diagram, all morphisms
become isomorphisms (and all complexes are concentrated in degree $0$). Therefore, to prove
(\ref{telesc}), it is enough to prove the commutativity of subdiagrams \framebox{1},
\framebox{2} and \framebox{3}. Each of these diagrams is isomorphic to a diagram of the
following form
$$\xymatrix{
R\Gamma_{c}^{\rm dim}(Y_{\mathbf{R}_{1},\mathbf{R}_{2}}) 
\ar[rr]^-{\Psi_{\mathbf{R}_{1},\mathbf{R}_{2},\mathbf{R}_{3}}^{(3)}} &&
R\Gamma_{c}^{\rm dim}(Y_{\mathbf{R}_{1},\mathbf{R}_{2},\mathbf{R}_{3}}) \\
R\Gamma_{c}^{\rm dim}(Y_{\mathbf{R}_{1}})
\ar[rr]_{\Psi_{\mathbf{R}_{1},\mathbf{R}_{3}}^{(2)}}
\ar[u]^{\Psi_{\mathbf{R}_{1},\mathbf{R}_{2}}^{(2)}} &&
R\Gamma_{c}^{\rm dim}(Y_{\mathbf{R}_{1},\mathbf{R}_{3}})
\ar[u]_{\Psi_{\mathbf{R}_{1},\mathbf{R}_{2},\mathbf{R}_{3}}^{(2)}}
}$$
in which the morphisms $\Psi_{\mathbf{R}_{1},\mathbf{R}_{2},\mathbf{R}_{3}}^{(3)}$
and $\Psi_{\mathbf{R}_{1},\mathbf{R}_{2},\mathbf{R}_{3}}^{(2)}$ are defined in a way
similar to what we explained above (\ref{PP'}) for
$\Psi_{\mathbf{R}_{1},\mathbf{R}_{2}}^{(2)}$. 
We refer to section 6.B of \cite{cdvdl2} for more details (see also section 2 of
\cite{indep}). 
In particular, Lemma 6.14 of \cite{cdvdl2} (or Corollary 2.9 of
\cite{indep}) tells us that such a diagram is commutative provided one of the following
inclusion holds.
\begin{center}
  a. $\mathbf{R}_{2}\subset \mathbf{R}_{1}\mathbf{R_{3}}$, or 
b. $F(\mathbf{R}_{1})\subset \mathbf{R}_{1}\mathbf{R_{3}}$, or
c. $\mathbf{R}_{3}\subset \mathbf{R}_{2}F(\mathbf{R_{1}})$, or
d. $\mathbf{R}_{1}\subset \mathbf{R}_{2}F(\mathbf{R_{1}})$.
\end{center}

\emph{Diagram \framebox{2}.} Here we are in the situation $\mathbf{R}_{1}=\mathbf{Q}_{1}$, 
$\mathbf{R}_{2}=\mathbf{Q}_{1}^{2}$ and $\mathbf{R}_{3}=\mathbf{Q}'_{1}$.
From the definition of these parabolic subgroups, we see that 
$$\mathbf{Q}_{1}^{2} = 
(\mathbf{P}_{\iota}\cap \mathbf{P}_{1})(\mathbf{U}_{\iota}\cap
\mathbf{P}_{2}) \mathbf{U}_{2}
\subset
(\mathbf{P}_{\iota}\cap \mathbf{P}_{1})\mathbf{U}_{\iota}
(\mathbf{P}_{\iota}\cap \mathbf{P}_{1})\mathbf{U}_{1}
= \mathbf{Q}_{1}\mathbf{Q}'_{1},$$ 
so we are in case a. and we conclude that diagram
\framebox{2} is commutative. 

\emph{Diagram \framebox{1}.} Taking into account the roundabout definitions of morphisms
$\Psi^{(1)}$ (through certain equivalences of \'etale sites), we are in the situation
$\mathbf{R}_{1}=\mathbf{Q}'_{1}$, $\mathbf{R}_{2}=F(\mathbf{Q}_{1})$ and
$\mathbf{R}_{3}=F(\mathbf{Q}_{1}^{2})$. As above, we have $F(\mathbf{Q}_{1}^{2})\subset
F(\mathbf{Q}_{1})F(\mathbf{Q}'_{1})$, so we are in case c. and we conclude that diagram
\framebox{1} is commutative. 

\emph{Diagram \framebox{3}.} Here we are in the situation 
$\mathbf{R}_{1}=\mathbf{Q}_{1}^{2}$, $\mathbf{R}_{2}=\mathbf{Q}'_{1}$ and
$\mathbf{R}_{3}=F(\mathbf{Q}_{1})$. We have
$$ F(\mathbf{Q}_{1}^{2})= (F(\mathbf{P}_{\iota})\cap \mathbf{P}_{1})
(F(\mathbf{U}_{\iota})\cap \mathbf{P}_{2})\mathbf{U}_{2}
\subset F(\mathbf{Q}_{1})\mathbf{U}_{2}
\subset F(\mathbf{Q}_{1})\mathbf{Q}_{1}^{2},
$$
 so we are in case b. and we conclude that diagram
\framebox{3} is commutative. 
\end{proof}

\section{From level $0$ blocks to coefficient systems} \label{sec:level-0-blocks}

We now denote by $F$ a local non-archimedean field with ring of
integers $\OC$ and residue field
$\FM$. We fix a vector space $V$ over $F$ of dimension $n$ and are
interested in the smooth representation theory of the $p$-adic group
$G=\Aut_{F}(V)$.

\subsection{Review of coefficients systems and systems of idempotents}

\alin{Coefficient systems on simplicial complexes} \label{def_coef_sys}
Recall that a simplicial complex $X=(X,X_{\bullet})$ is a pair of sets with $X_{\bullet}$ consisting 
of finite subsets of $X$, being stable under inclusion and containing all singletons. We'll denote
by $X_{d}\subset X_{\bullet}$ the set of $d$-simplices (of cardinality $d+1$) and identify $X$ to $X_{0}$ (the set
of vertices). We also denote by $[X]$ the small category whose set of objects is $X_{\bullet}$ and
morphisms are given by inclusion, \emph{i.e.}
\begin{center}
  $\Hom_{[X]}(\sigma,\tau)=\{1_{\sigma\tau}\}$ if $\sigma\supseteq\tau$ and
  $\Hom_{[X]}(\sigma,\tau)=\emptyset$ else.
\end{center}

A \emph{coefficient system} on $X$ with values in a category $\CC$ is a functor $\EC:\, [X]\To{}\CC$.
When $\CC$ is an abelian category with arbitrary direct sums,  we may form ``the'' chain
complex of a coefficient system
$$ \CC_{*}(\EC):\,  \cdots\To{} \cdots \bigoplus_{\sigma\in X_{n-1}} \EC_{\sigma} \To{}\cdots\To{} 
\bigoplus_{\sigma\in X_{0}} \EC_{\sigma} $$
once an orientation of $X$ has been chosen. The first homology object $H^{0}(\EC)$ is then
canonically isomorphic to the colimit $\colim_{[X]}\EC$ of the functor $\EC$.

Suppose further that $X$ carries an action of an abstract group $G$ that respects the simplicial structure.
 We then denote by $[X/G]$ the small category whose set of objects is  $X_{\bullet}$  and morphisms
 are given by the action of $G$, \emph{i.e.}
$$\Hom_{[X/G]}(\sigma,\tau)=G_{\sigma\tau}^{\dag}:=\{g\in G,\, \tau\subseteq g\sigma\}.$$
Hence $[X/G]$ contains $[X]$. Now, 
a \emph{$G$-equivariant coefficient system}  on $X$ with values in a category $\CC$ is a functor 
$\EC:\, [X/G]\To{}\CC$. For each $\sigma\in X_{\bullet}$, the object $\EC_{\sigma}$ then gets an
action of the stabilizer $G_{\sigma}^{\dag}$ of $\sigma$ and the  colimit $\colim_{[X]} \EC$  gets an action of
$G$, as well as the complex $\CC_{*}(\EC)$ when $\CC$ is abelian.
If $G$ is a topological group and acts continously on $X$ with the discrete
topology, we say that $\EC$ is \emph{smooth} if for each $\sigma\in X_{\bullet}$, the action of
$G_{\sigma}^{\dag}$ on $\EC_\sigma$ is so. Then the action of $G$ on $\colim_{[X]} \EC$ and $\CC_{*}(\EC)$ is also smooth.

We will generally denote by $G_{\sigma}$ the pointwise stabilizer of $\sigma$, which is a
distinguished subgroup of $G_{\sigma}^{\dag}$. For $\tau\subseteq\sigma$ we have $G_{\sigma}\subset G_{\tau}$, but
in general $G_{\sigma}^{\dag}$ is not contained in $G_{\tau}^{\dag}$.

\def\Latt{{\rm Latt}}

\alin{The Bruhat-Tits building} Here we take up our notation $G=\Aut_{F}(V)$ of the beginning of this section.
We introduce the semi-simple building $BT$ associated to $G$, which we view  as a simplicial complex of dimension $n-1$
with a simplicial  action of $G$. 

Denote by $\Latt(V)$ the set of $\OC_{F}$-lattices in $V$. 
A lattice chain 
is a decreasing sequence $\ZM\To{}\Latt(V)$,
$i\mapsto \LC_{i}$. It is called $d+1$-periodic if $\LC_{i+d}=\varpi\LC_{i}$
for all $i$. Moreover, two such sequences are called
equivalent if one is a shift of the other one.
Denote by $BT_{d}$ the set of equivalence classes of $d$-periodic lattice chains. It turns out that
such an equivalence class is completely determined by the $d+1$ equivalence classes of $1$-periodic
lattice chains that can be extracted from it. Whence an embedding of $BT_{d}$ in the power set of $BT:=BT_{0}$, that
is part of a simplicial complex structure $(BT,BT_{\bullet})$, for which
a $d'$-simplex $\tau$ is a facet of a $d$-simplex $\sigma$ if a
lattice chain in $\tau$ is extracted from one in $\sigma$.

The action of $G$ on $\Latt(V)$ induces a simplicial and continuous action on $BT$.
Let $G^{0}:=\ker(|\det|)\subset G$. 
The group $G_{\sigma}^{0}:=G_{\sigma}^{\dag}\cap G^{0}$ is
the maximal compact subgroup of the stabilizer $ G_{\sigma}^{\dag}$ of a simplex $\sigma$. It is a normal
subgroup and we have $G_{\sigma}=G_{\sigma}^{0}Z$ with $Z$
denoting the center of $G$, while $G_{\sigma}^{\dag}/G_{\sigma}$ is a
cyclic group of order dividing $n$.

\alin{Consistent systems of idempotents and Serre subcategories} \label{consistent}
Let $R$ 
be a commutative ring in which $p$ is invertible
 and let $\HC_{R}(G)$ denote the algebra of locally constant
$R$-valued compactly supported measures on $G$.
We are interested in collections of idempotents $(e_{x})_{x\in
  BT_{0}}$ in $\HC_{R}(G)$ that satisfy the following consistency
properties, taken from \cite[\S 2.1]{MS1}.
\begin{enumerate}
\item $e_{x}e_{y}=e_{y}e_{x}$ whenever $x$ and $y$ are adjacent.
\item $e_{x}e_{z}e_{y}=e_{x}e_{y}$ whenever $z$ belongs the simplicial
  hull (enclos) of $\{x,y\}$.
\item $e_{gx}=ge_{x}g^{-1}$ for all $x$ and all $g\in G$.
\end{enumerate}
Note that condition i) enables one to define an idempotent $e_{\sigma}=\prod_{x\leq\sigma}
e_{x}$ for each simplex $\sigma$. Moreover we have $e_{\sigma}e_{\tau}=e_{\sigma}$ for
$\tau\subseteq \sigma$ and $e_{g\sigma}=ge_{\sigma}g^{-1}$ for $g\in G$.
Note also that the whole system is
determined by a single $e_{x}$, since in our $\GL_{n}$ setting, all
vertices are conjugated.

Our interest in this notion comes from the following result of Meyer
and Solleveld. Recall that $\Rep_{R}(G)$ is the category of smooth
$R$-representations of $G$.
\begin{fac}[\cite{MS1}, Thm 3.1]
 Let $e=(e_{x})_{x\in BT}$ be a consistent system of
  idempotents as above.
 Then the full subcategory   $\Rep^{e}_{R}(G)$ of all objects $V$ in
 $\Rep_{R}(G)$ such that $V=\sum_{x} e_{x}V$ is a Serre sub-category.
\end{fac}
In fact, level decompositions on $\Rep_{R}(G)$ show that
$\Rep_{R}^{e}(G)$ is a direct factor subcategory, or in other words,
that $\HC_{R}(G)e_{x}\HC_{R}(G)$ is a direct factor two-sided ideal of $\HC_{R}(G)$.

\alin{Resolutions and equivalences of categories} \label{resol}
The most important ingredient in the proof of the above Fact is a
certain acyclicity result. Namely, for a representation $V$, define
the coefficient system $(\VC_{\sigma})_{\sigma\in BT}$ by putting
$\VC_{\sigma}:=e_{\sigma}V$ for all $\sigma\in BT_{\bullet}$ and
$\VC(g_{\sigma\tau}):\, e_{\sigma}V\To{} e_{\tau}V$ the map  induced by the
action of $g$ on $V$, for all $g\in G_{\sigma\tau}^{\dag}$ (recall that
$e_{\tau}e_{g\sigma}=e_{g\sigma}$). Then
the complex $\CC_{*}(\VC)$ has a natural augmentation given by
$\bigoplus_{x\in BT_{0}} e_{x}V\To{\sum}V$. 

\begin{fac}[\cite{MS1}, Thm 2.4]
  The complex $\CC_{*}(\VC)$ is acyclic in positive degrees and
  the augmentation map is an isomorphism $H_{0}(\CC_{*}(\VC))\simto \sum_{x}e_{x}V$.
\end{fac}

\def\Coef{{\rm Coef}}

Denote by $\Coef_{R}[BT/G]$ the category of $G$-equivariant coefficient
systems in $R$-modules on the building $BT$.
The above result suggests one to consider the full subcategory of
all coefficient systems 
$\VC$ such that $\VC(1_{\sigma\tau})$
induces an isomorphism $\VC_{\sigma}\simto
e_{\sigma}\VC_{\tau}$ for all $\tau\subseteq\sigma$ (equivalently,
$\VC(g_{\sigma\tau})$ induces an isomorphism $\VC_{\sigma}\simto
e_{g\sigma}\VC_{\tau}$ for all $g\in G_{\sigma\tau}^{\dag}$). In order to be able to do this, one needs to
impose a further axiom on the system $e$, namely
\begin{enumerate}
\item[iv)] for any $\sigma$, $e_{\sigma}$ is supported in $G_{\sigma}$.
\end{enumerate}
Indeed, under this assumption,  since $G_{\sigma} \subset G_{\tau}$ for
$\tau\subseteq\sigma$ and since $\VC_{\tau}$ comes with a smooth action of
$G_{\tau}$, we can  unambiguously make sense of the expression $e_{\sigma}\VC_{\tau}$.
Denote by $\Coef_{R}^{e}[BT/G]$ the category of these coefficients
systems.
Wang has extended the above fact in the following way.
\begin{fact}[\cite{Wang3}, Thm 2.1.9 and Cor. 2.1.10] \label{equivcoef}
For $\VC \in\Coef_{R}^{e}[BT/G]$ the complex $\CC_{*}(\VC)$ is acyclic in
positive degrees. Moreover the functor $\VC\mapsto H_{0}(\CC_{*}(\VC))$
provides an equivalence of categories
$$\Coef_{R}^{e}[BT/G]\simto \Rep_{R}^{e}(G)$$
with quasi-inverse the functor $V\mapsto (e_{\sigma}V)_{\sigma}$.
\end{fact}

We note that Meyer and Solleveld results are valid for any reductive
group over $F$, while Wang's result has been written only for $\GL_{n}$.

\alin{The level $0$ example} \label{level0example}
For a vertex $x$, denote by $G_{x}^{+}$ the kernel of the action map 
$G_{x}^{0}\To{}\Aut_{\FM}(\LC\otimes_{\OC}\FM)$. It is
independent of the choice of $\LC$ and coincides with the 
pro-$p$-radical of $G_{x}$. In particular it determines an  idempotent
$e_{x}^{+}\in \HC_{\ZM[1/p]}(G_{x})$.

The system $(e_{x}^{+})_{x\in BT_{0}}$ is well known to be consistent
in the above sense.
Moreover, for any simplex $\sigma$, the idempotent
$e_{\sigma}^{+}:=\prod_{x\in \sigma} e_{x}^{+}$ is nothing but the
idempotent associated to the pro-$p$-radical $G_{\sigma}^{+}$ of $G_{\sigma}^{0}$,
the latter being also the kernel of the natural map
$G_{\sigma}^{0}\To{} \prod_{i}\Aut_{\FM}(\LC_{i}/\LC_{i+1})$
whenever $\sigma$ is represented by the lattice chain $(\LC_{i})_{i}$.
In particular, property iv) above is also satisfied.

We denote by $\Rep^{0}_{R}(G)$ the
direct factor of $\Rep_{R}(G)$ cut out by the system
$(e_{x}^{+})_{x}$. It is usually called the \emph{level $0$ subcategory}.
Similarly we write $\Coef_{R}^{0}[BT/G]$ for the corresponding
category of equivariant coefficient systems.

\subsection{Level $0$ blocks}

Let $\Lambda$ denote either $\oQl$ or $\oZl$ (or $\oFl$).
Fix  a semisimple conjugacy class $s$ in $\GL_{n}(\FM)$ and assume it has order invertible
in $\Lambda$. 
Since $\GL_{n}$ is its own dual, paragraph \ref{defes} provides us with
a central idempotent $e_{s,\Lambda}^{\GL_{n}(\FM)}$ in
$\Lambda[\GL_{n}(\FM)]$. 

\alin{The systems of idempotents attached to $s$} \label{idempotents_building}
Let $x$ be a vertex. Let us choose a lattice $\LC$ corresponding to
$x$ and an $\OC$-basis of $\LC$. 
Then we get an isomorphism
$$\bar G_{x}:= G_{x}/G_{x}^{+}\simto \Aut_{\FM}(\LC
\otimes_{\OC}\FM)\simto \GL_{n}(\FM)$$
which allows us to pull back  $e_{s,\Lambda}^{\GL_{n}(\FM)}$  to a central idempotent
$e_{x}^{s,\Lambda}\in \Lambda[\bar G_{x}]\subset 
\HC_{\Lambda}(G_{x})$. Since it is central it is in fact independent of
the choices of $\LC$ and the basis. 

\begin{lem}
   The system $(e_{x}^{s,\Lambda})_{x\in BT_{0}}$ 
is consistent in the sense of \ref{consistent} and satisfies
  the additional property iv) of \ref{resol}. Moreover, for any two vertices
  $x,y$ we have
\ini
\begin{equation}
  \label{ortho}
  e_{x}^{+}e_{y}^{s,\Lambda}= e_{x}^{s,\Lambda}e_{y}^{s},
\end{equation}

\end{lem}
\begin{proof}
We will first extend the system to all simplices. Let $\sigma$ be a
simplex represented by a periodic lattice chain
$\cdots\supset \LC_{0}\supset\LC_{1}\supset \cdots$. We thus
get a partial flag  in the vector space
$\LC_{0}\otimes_{\OC}\FM$. After choosing an $\OC$-basis of
$\LC_{0}$ we get a parabolic subgroup $P_{\sigma}$ of
$\GL_{n}(\FM)$ and a
surjection $G_{\sigma}^{0}\twoheadrightarrow P_{\sigma}$. Let
$U_{\sigma}$ be the radical of $P_{\sigma}$ and $e_{U_{\sigma}}$ the
corresponding idempotent. Then by (\ref{compat2}) the idempotent
$e_{s,\Lambda}^{\GL_{n}(\FM)}e_{U_{\sigma}}$ is a central idempotent in
$\Lambda[P_{\sigma}]$, and therefore defines a central idempotent
$e_{\sigma}^{s,\Lambda}\in\HC_{\Lambda}(G_{\sigma})$. Note that changing the
lattice chain or the $\OC$-basis will result in conjugating the
parabolic subgroup $P_{\sigma}$ in $\GL_{n}(\FM)$. Therefore
$e_{\sigma}^{s,\Lambda}$ is independent of these choices. Moreover, by
construction, we have 
\ini\begin{equation}
 e_{\sigma}^{s,\Lambda}=e_{\sigma}^{+}e_{x}^{s,\Lambda}
\hbox{ for any vertex }x\leq \sigma,\label{defesigma}
\end{equation}
and where $e_{\sigma}^{+}$ is the idempotent associated to
$G_{\sigma}^{+}$ as in \ref{level0example}.

We now can prove that these systems are consistent. For the sake of readability, we omit the
superscript $\Lambda$. 
For property i) of \ref{consistent},  let $x,y$ be adjacent vertices and
$\sigma$ the simplex $\{x,y\}$. Then we know that
$e_{x}^{+}e_{y}^{+}=e_{\sigma}^{+}$. Therefore we get
$$ e_{x}^{s}e_{y}^{s}=e_{x}^{s}e_{x}^{+}e_{y}^{+}e_{y}^{s}= 
e_{x}^{s}e_{\sigma}^{+}e_{\sigma}^{+}e_{y}^{s}= e_{\sigma}^{s}e_{\sigma}^{s}=e_{\sigma}^{s}.$$
This proves property i). Moreover, we see in the same way that
$e_{\sigma}^{s}=\prod_{x\leq \sigma}e_{x}^{s}$ for any simplex $\sigma$, so that property iv)
of \ref{resol} also holds. Since property iii) is clear, it only remains to check
property ii). So let $z$ belong to the enclos of $\{x,y\}$. We already know
that  $e_{x}^{+}e_{y}^{+}=e_{x}^{+}e_{z}^{+}e_{y}^{+}$. This implies
the  desired property when $z$ is adjacent to $x$, since in this case
we have
$$ e_{x}^{s}e_{y}^{s}= e_{x}^{s}e_{z}^{+}e_{y}^{s}=
e_{x}^{s}e_{[x,z]}^{+}e_{y}^{s}= e_{[x,z]}^{s}e_{y}^{s}=
e_{x}^{s}e_{z}^{s}e_{y}^{s}
.$$
When $z$ is not adjacent to $x$, we can find a path
$x=x_{0},x_{1},\cdots, x_{r}=z, \cdots ,x_{l}=y$ such that $x_{i+1}$
is adjacent to $x_{i}$ and belongs to the enclos of
$\{x_{i},y\}$ for each $i$, as in \cite[Lemme (2.2.5)]{Wang3}. Then by induction we get property ii).

Finally we prove equality (\ref{ortho}). When $x$ and $y$ are adjacent
we already know that 
$$ e_{x}^{+}e_{y}^{s}=
e_{[x,y]}^{+}e_{y}^{s}=e_{[x,y]}^{s}=e_{x}^{s}e_{y}^{s}$$
as desired. In general, we choose a path $x=x_{0},\cdots x_{l}=y$ as above. 
Since $x_{l-1}$ is in the enclos of $\{x,y \}$ and is adjacent to $y$, we have
$$ e_{x}^{+}e_{y}^{s}= e_{x}^{+}e_{x_{l-1}}^{+}e_{y}^{s}=
e_{x}^{+}e_{x_{l-1}}^{s}e_{y}^{s}.$$
By induction on $l$, we deduce that 
$$ e_{x}^{+}e_{y}^{s}= e_{x}^{s}e_{x_{1}}^{s}\cdots
e_{x_{l-1}}^{s}e_{y}^{s}= e_{x}^{s}e_{y}^{s}$$
as claimed.
\end{proof}

We will denote by $\Rep^{s}_{\Lambda}(G)$ the Serre subcategory of
$\Rep_{\Lambda}(G)$ cut out by the system of idempotents
$(e_{x}^{s,\Lambda})_{x}$. It is pro-generated by the induced object
$\cind{G_{x}^{0}}{G}{e_{x}^{s,\Lambda}\HC_{\Lambda}(G_{x})}$ for any vertex $x$,
and is clearly contained in the level $0$
subcategory $\Rep^{0}_{\Lambda}(G)$ of $\Rep_{\Lambda}(G)$.

\begin{prop}
 The level $0$ subcategory $\Rep^{0}_{\Lambda}(G)$ splits as a direct product
$$ \Rep^{0}_{\Lambda}(G) = \prod_{s\in \GL_{n}(\FM)^{\rm ss,\Lambda}/{\rm
    conj}} \Rep^{s}_{\Lambda}(G)$$ 
where $s$ runs over semisimple conjugacy classes in
$\GL_{n}(\FM)$ which have invertible order in $\Lambda$. 
\end{prop}
\begin{proof} Let $s,s'$ be two distinct conjugacy classes with invertible order in $\Lambda$. 
For simplicity we simply write $e_{x}^{s}=e_{x}^{s,\Lambda}$ and $e_{x}^{s'}=e_{x}^{s',\Lambda}$.
We first
  show that $\Rep^{s}_{\Lambda}(G)$ and $\Rep^{s'}_{\Lambda}(G)$ are orthogonal.
Indeed, let $V$ be an object in $\Rep^{s}_{\Lambda}(G)$. We have to prove
that $e_{x}^{s'}V=0$ for any vertex $x$. We know that $V=\sum_{y} e_{y}^{s}V$, so that
$e_{x}^{s'}V=\sum_{y} e_{x}^{s'}e_{y}^{s}V$. But for all $x,y$, using
(\ref{ortho}) we get
$$
e_{x}^{s'}e_{y}^{s}=e_{x}^{s'}e_{x}^{+}e_{y}^{s}=e_{x}^{s'}e_{x}^{s}e_{y}^{s}=0$$
since $e_{x}^{s'}e_{x}^{s}=0$.

Now, let $V$ be a level $0$ object that is orthogonal to all
subcategories $\Rep^{s}_{\Lambda}(G)$.
For a fixed vertex $x$, we thus have $e_{x}^{s}$ for all $s$. But by
\ref{defes} i) we know that $e_{x}^{+}=\sum_{s} e_{x}^{s}$, so $e_{x}^{+}V=0$ hence
$V=0$. 

This finishes the proof of the claimed decomposition. Note that,
concretely, the projection of a level $0$ object $V$ on
$\Rep^{s}_{\Lambda}(G)$ is canonically isomorphic to the subrepresentation generated by
$e_{x}^{s,\Lambda}V$.
\end{proof}

The link between the decompositions over  $\oZl$ and $\oQl$  is
quite clear : for an $\ell$-regular semisimple class $s$ we have
$$ \Rep^{s}_{\bZl}(G)\cap \Rep_{\bQl}(G) = \prod_{s'\sim_{\ell}s}\Rep^{s'}_{\bQl}(G).$$

\alin{Relation to Bernstein blocks} \label{relation_to_B_blocks}
Here we consider the case $\Lambda=\oQl$. Fix a conjugacy class $s$ as above.
There is unique partition $(n_{1}\geq n_{2}\geq \cdots \geq n_{r})$ of $n$ such
that $s$ contains an \emph{elliptic} element of the diagonal Levi subgroup
$M(\FM)=\GL_{n_{1}}(\FM)\times\cdots\times\GL_{n_{r}}(\FM)$. As in
\ref{gln} i), we thus get from
$s$ a supercuspidal representation $\bar\sigma_{1}\otimes\cdots
\otimes\bar\sigma_{r}$ of this Levi subgroup. After inflating to
$M(\OC)=\prod_{i}\GL_{n_{i}}(\OC)$, extending to $\prod_{i}
F^{\times}\GL_{n_{i}}(\OC)$,  and inducing to $M(F)=\GL_{n_{1}}(F)\times\cdots\times\GL_{n_{r}}(F)$,
we get a supercuspidal representation
$\sigma=\sigma_{1}\otimes\cdots\otimes\sigma_{r}$ of the $p$-adic group
$M(F)$. Choosing further an
$F$-basis of $V$ we get a cuspidal pair $(M,\sigma)$ in $G$ whose
inertial class does not depend on any choice (basis and extension).
In view of the well known construction of level $0$ supercuspidal
representations, this process  sets up a bijection
\begin{center}
  $\{$semisimple conj. classes in $\GL_{n}(\FM)\} \leftrightarrow
  \{$inertial classes of level $0$ supercuspidal pairs in $G\}$  
\end{center}

Now, attached to $[M,\sigma]$ is the Bernstein block
$\Rep^{[M,\sigma]}_{\bQl}(G)$ of $\Rep_{\bQl}(G)$ which consists of
all objects, all irreducible subquotients of which have supercuspidal
support inertially equivalent to  $[M,\sigma]$. Bernstein's
decomposition, when restricted to  the level $0$ subcategory, then reads
$$\Rep^{0}_{\bQl}(G) = \prod_{[M,\sigma]}\Rep^{[M,\sigma]}_{\bQl}(G).$$

\begin{pro}
We have   $\Rep^{s}_{\bQl}(G)=\Rep^{[M,\sigma]}_{\bQl}(G)$. In
particular $\Rep^{s}_{\bQl}(G)$ is  a block. 
\end{pro}
\begin{proof}
In view of the two above decompositions of $\Rep^{0}_{\bQl}(G)$ and the bijection $s
\leftrightarrow [M,\sigma]$, it will be sufficient to prove the
inclusion $\Rep^{s}_{\bQl}(G)\supset\Rep^{[M,\sigma]}_{\bQl}(G)$ for
$s$ corresponding to $[M,\sigma]$.

For this it suffices to prove that any irreducible $\pi$ which occurs
as a subquotient of an induced representation
$\ip{P}{G}{\sigma_{1}\otimes\cdots\otimes\sigma_{r}}$ as above (where
$P$ is  a parabolic
  subgroup with Levi $M$) satisifies $e_{x}^{s}\pi\neq 0$ for some
  vertex $x$.

For simplicity, let us choose an $F$ basis of $V$ such that the
identification $G=\GL_{n}(F)$ takes $M$ to the diagonal Levi subgroup
$\prod_{i}\GL_{n_{i}}(F)$ and $P$ to the corresponding
upper-triangular group. Then consider the vertex $x$
associated to the lattice spanned by this basis, so that
$G_{x}^{0}=\GL_{n}(\OC)$. Finally denote by $\bar P_{x}$ the image
of $P\cap G_{x}^{0}$ in $\bar G_{x}$. Then the Mackey formula shows that,
as a representation of $\bar G_{x}$ we have
$$
\ip{P}{G}{\sigma_{1}\otimes\cdots\otimes\sigma_{r}}^{G_{x}^{+}}\simeq 
\ip{\bar P_{x}}{\bar G_{x}}{\bar \sigma_{1}\otimes\cdots\otimes\bar \sigma_{r}}.
$$
where $\Ip{P}{G}$ denotes parabolic induction.
Therefore $\pi^{G_{x}^{+}}$ is a subquotient of $\ip{\bar P_{x}}{\bar
  G_{x}}{\bar \sigma_{1}\otimes\cdots\otimes\bar \sigma_{r}}$ and by
\ref{gln} ii) we get that $e_{x}^{s}\pi\neq 0$.
\end{proof}

\alin{Compatibility with Langlands' correspondence}\label{compat_LLC}
As in \ref{dual_param_blocks}, the local Langlands correspondence gives a parametrization of blocks
of $\Rep_{\oQl}(\GL_{n}(F))$  by the set $\Phi_{\rm inert}(\GL_{n},\oQl)$ of equivalence classes of
semi-simple representations of $I_{F}$ that extend to $W_{F}$. Moreover, level $0$ blocks
correspond to tame parameters through this parametrization. 
Sofar, we have defined three maps :
$$\xymatrix{
&    \hbox{$\{$inertial classes of level $0$ supercuspidal pairs in $G\}$}
     \\
\Phi_{\rm inert}^{\rm tame}(\GL_{n},\oQl) \ar[r]^-{\phi\mapsto s} \ar@{<->}[ru]^{LLC} & 
\hbox{$\{$semisimple conj. classes in $\GL_{n}(\FM)\}$} \ar[u]_{s\mapsto [M,\sigma]}
}$$
The vertical map was defined in \ref{relation_to_B_blocks}, while the horizontal one was
defined in \ref{param-conjcl}.

\begin{pro}
  This diagram commutes. In particular we have $\Rep_{\phi}(G)=\Rep^{s}_{\oQl}(G)$.
\end{pro}
\begin{proof}
  Let us start with a semisimple element $s\in\GL_{n}(\FM)$, let us denote by $[M,\sigma]$
  the corresponding supercuspidal pair, by $\phi$ the inertial parameter of $[M,\sigma]$,
  and by $s(\phi)$ the conjugacy class associated to $\phi$. We want to prove that $s\in s(\phi)$.

We first reduce to the case where $s$ is elliptic. Indeed, after maybe conjugating we
may assume that $s=s_{1}\times\cdots\times s_{r}$ is elliptic in the Levi subgroup
  $M(\FM)=\GL_{n_{1}}(\FM)\times\cdots\times\GL_{n_{r}}(\FM)$. Then we have $M=\GL_{n_{1}}\times\cdots\times\GL_{n_{r}}$
and  $\sigma=\sigma_{1}\otimes\cdots\otimes\sigma_{r}$ with each $\sigma_{i}$
corresponding to $s_{i}$. Next, putting
  $\phi_{i}:=(\varphi_{\sigma_{i}})_{|I_{F}}$, we have  $\phi=\phi_{1}\oplus\cdots\oplus\phi_{r}$, and the conjugacy class
  $s(\phi)$ is by definition that of a product
  $s(\phi_{1})\times\cdots\times s(\phi_{r})$ in $M(\FM)\subset \GL_{n}(\FM)$. 
Therefore we are reduced to check $s_{i}\in s(\phi_{i})$ for each $i$.

Now assume $s$ is elliptic, so that $M=G$. This case boils down to the well-known compatibility between the
local Langlands correspondence and the Green-Deligne-Lusztig construction (up to some
\emph{unramified} rectifying character). 
\end{proof}

\alin{Relation to Vign\'eras-Helm blocks} Here we consider the case $\Lambda=\oZl$ and we suppose that $s$ is an
$\ell$-regular semisimple conjugacy class of $\GL_{n}(\FM)$. Using the
notation of paragraph \ref{relation_to_B_blocks}, we now get from $s$ a supercuspidal
$\oFl$-representation of $M(\FM)$, thanks to \ref{gln} iii).
Again, after inflating, extending, and inducing, we get an inertial
class $[M,\sigma]$ of supercuspidal pairs over $\oFl$ in $G$, and this
process induces a bijection
\begin{center}
  $\{$ $\ell$-regular semisimple conj. classes in $\GL_{n}(\FM)\}$ \\ $\leftrightarrow
  \{$inertial classes of level $0$ $\oFl$-supercuspidal pairs in $G\}$  
\end{center}

Now, attached to $[M,\sigma]$ is the Vign\'eras-Helm block
$\Rep^{[M,\sigma]}_{\oZl}(G)$ of $\Rep_{\oZl}(G)$ which consists of
all objects, all simple subquotients of which have ``mod $\ell$
inertial supercuspidal support''  equal to  $[M,\sigma]$. The level
$0$ subcategory then decomposes as
$$\Rep^{0}_{\oZl}(G) = \prod_{[M,\sigma]}\Rep^{[M,\sigma]}_{\oZl}(G),$$
the product being over  $[M,\sigma]$'s with level $0$.

\begin{pro}
We have   $\Rep^{s}_{\bZl}(G)=\Rep^{[M,\sigma]}_{\bZl}(G)$. In
particular $\Rep^{s}_{\bZl}(G)$ is  a block.
\end{pro}
\begin{proof}
As in the proof of Proposition \ref{relation_to_B_blocks}, it suffices to prove the
inclusion $\Rep^{s}_{\bZl}(G)\supset\Rep^{[M,\sigma]}_{\bZl}(G)$.
On the other hand, since the LHS is a Serre subcategory and the RHS is
a block, it suffices to find one object in the RHS that belongs to the
LHS. 
Therefore it suffices to prove that any irreducible
$\oFl$-representation $\pi$ in $\Rep^{[M,\sigma]}_{\oFl}(G)$ satisfies
$e_{x}^{s,\oZl}\pi\neq 0$. By definition $\pi$ occurs as a subquotient
of some 
$\ip{P}{G}{\sigma_{1}\otimes\cdots\otimes\sigma_{r}}$, so we may argue
exactly as in the previous proposition, using \ref{gln} iv) instead of ii).
\end{proof}

Let now $\phi\in\Phi_{\ell-\rm inert}(\GL_{n},\oQl)$ be the
$\ell'$-inertia parameter which corresponds to $[M,\sigma]$ as in \ref{dual_param_blocks}, and let $s(\phi)$
be the conjugacy class associated to $\phi$ as in \ref{param-conjcl}.

\begin{pro}
  We have $s=s(\phi)$. In particular, we have $\Rep^{s}_{\oZl}(G)=\Rep_{\phi}(G)$.
\end{pro}
\begin{proof}
  As in the proof of Proposition \ref{compat_LLC}, we reduce to the case where $s$ is elliptic. Then we
  have $M=G$ and   $\sigma$ is the reduction mod $\ell$ of the supercuspidal
  $\oQl$-representation $\wt\sigma$ attached to $s$. Moreover, by compatibility between
  Langlands and Vign\'eras correspondences for supercuspidal representations with respect
  to reduction mod $\ell$, the parameter $\phi$ is the
  restriction to $I_{F}^{\ell}$ of the inertial parameter $\wt\phi$ associated with $\wt\sigma$.   
Proposition \ref{compat_LLC} tells us that $s(\wt\phi)=s$, and \ref{param-conjcl} tells us that $s(\phi)$ is
the $\ell'$-part of $s$, which is $s$ by assumption.
\end{proof}

\section{From Coefficient systems to Hecke ``modules''}

We take up the setting of the last section. So $G={\rm Aut}_{F}(V)$ acts on its building
$BT$, and $\Lambda=\oZl$ or $\oQl$. 
To any semisimple conjugacy class $s\in\GL_{n}(\FM)$ with order invertible in
$\Lambda$, we have associated a system of idempotents $(e^{s,\Lambda}_{x})_{x\in BT}$,
heading to a block $\Rep^{s}_{\Lambda}(G)$ that is equivalent to a certain 
category ${\rm Coef}^{s}_{\oZl}[BT/G]$ of equivariant coefficient systems on $BT$.

In this section we show how ${\rm Coef}^{s}_{\Lambda}[BT/G]$ is equivalent to a category of
``modules'' over a certain Hecke ``$\Lambda$-algebra'', whose definition only involves a
single appartment $A$ and the collections of reductive quotients $\o G_{\sigma}$, $\sigma\in A_{\bullet}$.
In this way, we obtain objects that are closer to the usual modules over Hecke algebras,
except that  our Hecke ``algebras'' will actually be ``rings with many objects''.

\subsection{The enriched Coxeter complex}

\alin{Level $0$ coefficient systems} 
Let $\CC$ be any category. We say that a $\CC$-valued  $G$-equivariant coefficient system $\EC:\,[BT/G]\To{}\CC$ has
\emph{level $0$} if for all $\sigma\in BT_{\bullet}$, the action of $G_{\sigma}^{+}$ on
$\EC_{\sigma}$ is trivial.
Let us introduce a new category $[BT/\bar G]$ whose set of objects is still $BT_{\bullet}$, and
morphisms are given by 
$$\bar G_{\sigma\tau}^{\dag}:=  G_{\tau}^{+} \ba G^{\dag}_{\sigma\tau} /
G_{\sigma}^{+} = G^{\dag}_{\sigma\tau} / G_{\sigma}^{+}.$$
Note that the composition maps  $G_{\sigma\tau}^{\dag}\times G_{\tau\nu}^{\dag}\To{} G_{\sigma\nu}^{\dag}$
induce associative composition maps on quotients 
$\bar G_{\sigma\tau}^{\dag}\times \bar G_{\tau\nu}^{\dag}\To{} \bar G_{\sigma\nu}^{\dag}$
so we indeed get a category $[BT/\bar G]$, which is a quotient category of $[BT/G]$.
 It is then clear that a $G$-equivariant coefficient system $\EC$ has level $0$ if and
 only if it factors through a functor  $\EC:\,[BT/\bar G]\To{}\CC$.

 \begin{defn} \label{def_cartesian_coeff}
   We say that $\EC:\,[BT/\bar G]\To{}\CC$  is \emph{cartesian} if the transition map
   $\EC(\o{1}_{\sigma\tau})$ is an isomorphism $\EC_{\sigma}\simto 
   \EC_{\tau}^{G_{\sigma}^{+}}$ for all $\tau\subseteq\sigma$.
 \end{defn}

\alin{The enriched Coxeter complex  $[\AM/\bar\GM]$} \label{abstract_appart}
Put $I:=\{1,\cdots, n\}$ and $\AM:=\ZM^{I}/\ZM_{\rm diag}$. Two
elements $x$, $y$ in $\AM$ are called \emph{adjacent} if we can find representatives 
$\tilde x$, $\tilde y$ in $\ZM^{I}$ such that $\tilde x(i)\leq \tilde y(i)\leq
\tilde x(i)+1$ for all $i\in I$. 
A subset $\sigma$ of $\AM$ is called a simplex if all its pairs of elements are adjacent. We thus get
a simplicial complex $\AM_{\bullet}$, namely the Coxeter complex of type $\wt A_{n-1}$. The
natural action of the extended affine Weyl group  $W:=\ZM^{I}\rtimes \SG_{n}$ on $\AM$ is simplicial.
We will also denote by  $W^{0}\subset W$  the (non-extended) affine Weyl group, given
by 
$W^{0}=(\ZM^{I})^{0}\rtimes\SG_{n}\subset W$ where $(\ZM^{I})^{0}=\{(x_{i})_{i\in I},\, \sum_{i}x_{i}=0\}$.
For any simplex $\sigma\in\AM$ we put $W_{\sigma}^{0}:= W_{\sigma}^{\dag}\cap
W^{0}$. Then the fixator of $\sigma$ in $W$ decomposes as
$W_{\sigma}=W_{\sigma}^{0}\times \ZM_{\rm diag}$ and the quotient
$W_{\sigma}^{\dag}/W_{\sigma}$ embeds in $W/(W^{0}\ZM_{\rm diag})$ hence is cyclic of order dividing $n$.

Let $\sigma$ be a simplex in $\AM$. If we fix a vertex  $x_{0}\in\sigma$ and a representative $\tilde x_{0}$ of $x_{0}$,
there is a unique ordering $\sigma=\{ x_{0},x_{1},\cdots x_{d}\}$ such that
 we can find representatives $\tilde
x_{i} : I\To{} \ZM$ of $x_{i}$ verifying $\tilde x_{0} \leq \cdots \leq \tilde x_{d}\leq \tilde x_{0}+1$.
We get a  partition  $I=I_{\sigma,0}\sqcup I_{\sigma,1} \sqcup
\cdots \sqcup I_{\sigma,d}$ of $I$ 
defined by $I_{\sigma,k}:=\{i\in I, \tilde x_{k+1}- \tilde x_{k} =1\}$. Up to cycling the indices $k\mapsto
k+1 ({\rm mod }\, d)$, this partition does not depend on the choice of $x_{0}$ and its
representative.
In particular,  the  subset $P_{\sigma}=\{I_{\sigma,0},\cdots, I_{\sigma,d}\}$ of the
power set $\PC(I)$ of $I$ is canonically attached to $\sigma$.
We then  define a finite group of Lie type 
$$\bar\GM_{\sigma}:= \prod_{J\in P_{\sigma}}\Aut_{\FM}(\FM^{J})\subset\Aut_{\FM}(\FM^{I}).$$
If $\tau\subseteq \sigma$ is another simplex, then the partition $P_{\sigma}$ refines $P_{\tau}$, so
that $\bar\GM_{\sigma}$ canonically embeds as a (semi-standard) Levi subgroup of $\bar
\GM_{\tau}$. 
Moreover, if we choose $x_{0}\in \tau$, the associated ordering on $P_{\sigma}$ provides a flag 
$\FM^{I}\supset \FM^{I\setminus I_{\sigma,0}}\supset \cdots \supset \FM^{I_{\sigma,d}}\supset 0$
in $\FM^{I}$ whose intersection with each $\FM^{I_{\tau,k}}$ does not depend on the choice of
$x_{0}\in\tau$.
We thus get a ``canonical'' parabolic subgroup $\bar\PM_{\sigma,\tau}$ of $\bar \GM_{\tau}$ with
Levi component $\bar \GM_{\sigma}$,
and whose unipotent radical we denote by $\bar\UM_{\sigma,\tau}$.

Now, if $w\in W$, its image  $\bar w\in\SG_{n}$ under the quotient map $W\twoheadrightarrow \SG_{n}$ 
 takes $P_{\sigma}$ to $P_{w\sigma}$, hence the permutation matrix $[\bar w]\in
 {\rm Aut}_{\FM}(\FM^{I})$ conjugates $\bar\GM_{\sigma}$ to $\bar\GM_{w\sigma}$. We denote by 
$\bar w_{\sigma}:\,\bar\GM_{\sigma}\simto  \bar\GM_{w\sigma}$ the isomorphism thus obtained. When $w\in W_{\sigma}$, 
the permutation $\bar w$ preserves each subset $J\in P_{\sigma}$ of $I$, 
hence $[\bar w]\in\bar\GM_{\sigma}$.
Moreover, the map $w\mapsto [\bar w]$ is a group
embedding $W_{\sigma}^{0}\injo \bar\GM_{\sigma}$. Finally, if $\tau\subseteq w\sigma$ (\emph{i.e.} $w\in
W^{\dag}_{\sigma\tau}$), we put 
$$\bar w_{\sigma\tau}:\, \bar\GM_{\sigma}\To{\bar w_{\sigma}} \bar\GM_{w\sigma}\injo \bar\GM_{\tau}$$ the composition of the
maps just defined. 
The image $\bar w_{\sigma\tau}(\bar\GM_{\sigma})$ is a Levi subgroup of $\bar\GM_{\tau}$ which is contained in the
parabolic subgroup 
$\bar\PM_{w\sigma,\tau}= \bar w_{\sigma\tau}(\bar\GM_{\sigma}) \bar\UM_{w\sigma,\tau}$.
If $v\in W^{\dag}_{\tau\nu}$ for a third simplex $\nu$, it is easily checked that 
$\bar v_{\tau\nu}\circ \bar w_{\sigma\tau}=\bar{vw}_{\sigma\nu}$, and that
$\bar\UM_{vw\sigma,\nu}=\bar\UM_{v\tau,\nu}.\bar v_{\tau\nu}(\UM_{w\sigma,\tau})$.

For $\sigma,\tau\in \AM_{\bullet}$ we now define 
$$ \bar\GM^{\dag}_{\sigma\tau} :=  (\bar\GM_{\tau} \times W^{\dag}_{\sigma\tau})_{/\sim},$$ 
 where 
$$ (\bar g ,w)\sim (\bar g' ,w') \hbox{ iff } \exists (\bar u,v)\in
\bar\UM_{w\sigma,\tau}\times W_{\tau}^{0}, \, \bar g' =\bar g \bar u v \hbox{ and } w'= v^{-1}w. $$
One checks that the composition rule 
$(\bar\GM_{\nu} \times W^{\dag}_{\tau\nu})\times(\bar\GM_{\tau} \times W^{\dag}_{\sigma\tau})\To{}
\bar\GM_{\nu} \times W^{\dag}_{\sigma\nu}$ given
by $(\bar g_{\nu},v)\circ (\bar g_{\tau},w)=(\bar g_{\nu} \bar
v_{\tau\nu}(\bar g_{\tau}), vw)$
descends to a map 
$$ \bar\GM^{\dag}_{\tau\nu}\times\bar\GM^{\dag}_{\sigma\tau} \To{} \bar\GM^{\dag}_{\sigma\nu}$$
that fulfills the necessary associativity property to justify the following definition.
\begin{defn}
  We denote by $[\AM/\bar\GM]$ the small category whose set of objects is $\AM_{\bullet}$
  and the hom-sets are given by $\Hom(\sigma,\tau)=\bar \GM_{\sigma\tau}^{\dag}$.
\end{defn}

\alin{A functor from $[\AM/\bar \GM]$ to $[BT/\bar G]$}
Let us choose a basis of the $F$-vector space $V$. This allows us to identify $V$ to
$F^{I}$ and $G$ to $\Aut_{F}(F^{I})$.
We define a map $\AM\To{\iota} BT$ by sending $\tilde x\in \ZM^{I}$ to the lattice $\LC(\tilde x):=\bigoplus
\varpi^{\tilde x(i)}\OC_{F}. [i]$ where $\{[i],i\in I\}$ is the canonical basis of $F^I$. This defines
a simplicial isomorphism from $\AM_{\bullet}$ 
to the appartment $A$ associated to the diagonal torus of $G$.
Moreover for any simplex $\sigma$ in $\AM$ there is a canonical identification 
$\bar \GM_{\sigma}=\bar G_{\iota\sigma}^{0}:=G_{\iota\sigma}^{0}/G_{\iota\sigma}^{+}$, and if
$\tau\subseteq\sigma$ the identification 
$\bar\GM_{\tau}= \bar G_{\iota\tau}^{0}$ carries $\bar\PM_{\sigma,\tau}$ to
$G_{\iota\sigma}^{0}/ G_{\iota\tau}^{+}$ and $\bar\UM_{\sigma,\tau}$ to
$ G_{\iota\sigma}^{+}/ G_{\iota\tau}^{+}$.

We also embed $W$ in $G$ by mapping $w=(t,w_{0})$ to the  matrix $[t][w_{0}]$ where 
$[w_{0}]$ is the permutation matrix of $G$ attached to $w_{0}$ and
$[t]={\rm diag}(\varpi^{t_{1}}\cdots \varpi^{t_{n}})$ if $t=(t_{1},\cdots, t_{n})$.
The map $\iota$ is then $W$-equivariant and the map $W_{\iota\sigma}^{0}\To{}\bar G_{\iota\sigma}^{0}$
identifies with the map $W_{\sigma}^{0}\To{}\bar \GM_{\sigma}^{0}$ of the previous paragraph.

\begin{theo}
The multiplication map $\bar\GM_{\tau} \times W^{\dag}_{\sigma\tau} = \bar G_{\iota\tau}^{0} \times
W^{\dag}_{\iota\sigma,\iota\tau} \To{}  \bar G^{\dag}_{\iota\sigma,\iota\tau}$ induces a bijection $\bar \GM^{\dag}_{\sigma\tau}
\simto \bar G_{\iota\sigma,\iota\tau}^{\dag}$. Moreover, the functor 
$$[\AM/\bar \GM]\To{}[BT/\bar G]$$ thus obtained
is an equivalence of categories.
\end{theo}
\begin{proof}
In order to simplify the notation, we  write $\sigma$ for $\iota\sigma$ and $\tau$ for
$\iota\tau$. In view of the definition of $\bar\GM_{\sigma\tau}^{\dag}$ and the foregoing discussion, 
what need to be proved is that the multiplication map of the theorem induces a
bijection
$$ (G_{\tau}^{0} \times W^{\dag}_{\sigma\tau})_{/\sim} \simto \bar G^{\dag}_{\sigma\tau}$$
where the equivalence relation is given by 
$$ ( g ,w)\sim ( g' ,w') \hbox{ iff } \exists (h, u,v)\in G_{\tau}^{+}\times
G_{w\sigma}^{+}\times W_{\tau}^{0}, \,  g' = h  g  u v \hbox{ and } w'= v^{-1}w. $$
Since $W_{\tau}^{0}=G_{\tau}^{0}\cap W$, we may simplify this in
$$ ( g ,w)\sim ( g' ,w') \hbox{ iff } \exists (h, u)\in G_{\tau}^{+}\times
G_{w\sigma}^{+}, \,  g'w' = h  g  u w . $$
But since $G_{\tau}^{+}\subset G_{w\sigma}^{+}$ and $G_{w\sigma}^{+}= wG_{\sigma}^{+}w^{-1}$, we also have
$$ ( g ,w)\sim ( g' ,w') \hbox{ iff } \exists u \in  G_{\sigma}^{+}, \,  g'w' =   gw u. $$
Therefore the multiplication map descends to
$$ (G_{\tau}^{0} \times W^{\dag}_{\sigma\tau})_{/\sim} \simto  (G_{\tau}^{0}W^{\dag}_{\sigma\tau})/G_{\sigma}^{+}\subset
\bar G^{\dag}_{\sigma\tau}$$
and we are left to prove that the inclusion $G_{\tau}^{0}W^{\dag}_{\sigma\tau}\subset
G^{\dag}_{\sigma\tau}$ is an equality. So let $g$ be an element in $G^{\dag}_{\sigma\tau}$, that is
$g\sigma\supset \tau$. Since $\tau\in A_{\bullet}$ we can find $g_{\tau}\in G_{\tau}^{0}$ such that
$g_{\tau}g\sigma\in A_{\bullet}$. Then, since $\sigma\in A_{\bullet}$ too,  there is $w\in W$ such
that $w\sigma=g_{\tau}g\sigma$, so that we have $g^{-1}g_{\tau}^{-1}w\in G_{\sigma}^{\dag}$.
 Using the equality $G^{\dag}_{\sigma}=G_{\sigma}^{0}W^{\dag}_{\sigma}$, we finally write $g$ in the
 form $g_{\tau}^{-1}w g_{\sigma}^{-1}w_{\sigma}^{-1 }$ which belongs to
 $G_{\tau}^{0}W^{\dag}_{\sigma\tau}$ since $wG_{\sigma}^{0}w^{-1}\subset G_{\tau}^{0}$.

We thus have obtained a fully faithful functor $[\AM/\bar\GM]\To{}[BT/\bar G]$. But 
since any simplex in $BT$ is conjugate to some simplex in $A$ under $G$, and therefore isomorphic to
this simplex in $[BT/\bar G]$, this functor is also essentially surjective.
\end{proof}

\begin{coro} \label{equivcoeffmod} Let $\CC$ be any category.
  The restriction functor $\EC\mapsto \EC\circ\iota$ is an equivalence of categories
  between the category of  $\CC$-valued  level $0$ $G$-equivariant coefficient systems on $BT$ and the
  category of functors $[\AM/\bar\GM]\To{}\CC$.
\end{coro}

As in Definition \ref{def_cartesian_coeff}, 
let us say that a functor $\EC:[\AM/\bar\GM]\To{}\CC$ is \emph{cartesian} if
$\EC(\bar 1_{\sigma\tau})$ is an isomorphism $\EC_{\sigma}\simto
\EC_{\tau}^{\UM_{\sigma\tau}}$ for all $\sigma\supset \tau$. 
The above equivalence clearly
preserves cartesian objects on both sides. Using now \ref{equivcoef} and
\ref{level0example}, we get :

\begin{coro} \label{equivlevel0AG}
Let $R$ be any commutative ring with $p\in R^{\times}$.
The functor $V\mapsto (V^{G_{\iota\sigma}^{+}})_{\sigma\in\AM}$ is an equivalence of
categories between $\Rep^{0}_{R}(G)$ and the category of cartesian functors
$[\AM/\bar\GM]\To{} R-\rm Mod$ (with natural transforms as morphisms).
\end{coro}

Note that this implies in particular that the level $0$ subcategory
$\Rep^{0}_{R}(\GL_{n}(F))$ only depends on the residue field $\FM$ of $F$ (and $n$), up to
equivalence. Indeed, the definition of $[\AM/\bar\GM]$ only involves the Coxeter complex
associated to the affine Weyl group $W$, and the residual reductive groups $\bar\GM_{\sigma}$,
$\sigma\in \AM$.
In order to refine the above corollary for summands $\Rep^{s}_{\oZl}(G)$, we enhance
the setting a bit.

\subsection{Hecke ``algebras'' and ``modules''}

We refer to the appendix for the language of ``rings'' (also called ringoids) and
``modules'' that we will
use from now on. As above, $R$ denotes a commutative ring with $p\in R^{\times}$.

\alin{The Hecke ``algebra'' of level $0$} \label{hecke_level_0}
Let us apply the construction of the $R$-algebra of a small category $\AC$ as in \ref{rings}
to the case $\AC=[\AM/\bar\GM]$. We get an ``$R$-algebra''
$R[\AM/\bar\GM]$ with set of objects $\AM_{\bullet}$ and hom-sets 
defined by $R[\AM/\bar\GM](\sigma,\tau)=R[\bar\GM_{\sigma\tau}^{\dag}]$.
We wish to use Corollary \ref{equivlevel0AG} in order to see $\Rep^{0}_{R}(G)$ as a
category of left modules over $R[\AM/\bar\GM]$. Indeed, this corollary provides us with a
full and faithful functor 
$$ V\mapsto \VC,\,\, \Rep^{0}_{R}(G) \To{} R[\AM/\bar\GM]-\rm Mod,$$
with $\VC_{\sigma}:= V^{G_{\iota\sigma}^{+}}$ and the action 
$\VC_{\sigma}\otimes_{R}R[\bar\GM_{\sigma\tau}^{\dag}]\To{}\VC_{\tau}$ obtained by
$R$-linear extension of the action of $\bar\GM_{\sigma\tau}^{\dag}$ induced by that of $G$
on $V$ via the isomorphisms $\bar\GM_{\sigma\tau}^{\dag}\simto \bar G^{\dag}_{\sigma\tau}$.
The image of this functor is formed by ``$R[\AM/\bar\GM]$-modules'' which are cartesian as
functors $[\AM/\bar\GM]\To{} R-\rm Mod$ (as defined above Corollary \ref{equivlevel0AG}).
So we need to understand what this property means in terms of ``$R[\AM/\bar\GM]$-modules''.

Define $ W_{\sigma\tau}^{0}:= W_{\sigma\tau}^{\dag}\cap W^{0}
\hbox{ and }
\bar\GM_{\sigma\tau}^{0}:= (\bar\GM_{\tau} \times W_{\sigma\tau}^{0})/\sim$
where the equivalence relation is the same as in the definition of
$\bar\GM_{\sigma\tau}^{\dag}$ (so, in particular we have
$\bar\GM_{\sigma}^{0}=\bar\GM_{\sigma}$). 
We get a category $[\AM/\bar\GM]^{0}$ and 
an ``$R$-algebra'' $R[\AM/\bar\GM]^{0}$ which is naturally a ``sub-$R$-algebra'' of
$R[\AM/\bar\GM]$. Indeed any ``module'' over $R[\AM/\bar\GM]$ restricts canonically to a
``module'' over $R[\AM/\bar\GM]^{0}$. We refer to Definition \ref{cart_modules} for the
notion of Cartesian ``module''.

\begin{thm}
  The functor $V\mapsto \VC$ above is an equivalence of categories from $\Rep^{0}_{R}(G)$
  to the category of left $R[\AM/\bar\GM]$-modules that are cartesian
 as $R[\AM/\bar\GM]^{0}$-modules.
\end{thm}
\begin{proof}
Let $\sigma,\tau$ be two facets such that $W_{\sigma\tau}^{0}$ is non empty, and fix
$w\in W_{\sigma\tau}^{0}$. The next lemma tells us that
$W_{\sigma\tau}^{0}=W_{\tau}^{0}w$, and this implies in turn an identification of
$\bar\GM_{\tau}\times\bar\GM_{\sigma}^{\rm opp}$-sets
$$ \bar\GM_{\sigma\tau}^{0} = \bar\GM_{\tau}/\bar\UM_{w\sigma,\tau}$$ 
where $g_{\sigma}\in\bar\GM_{\sigma}$ acts on the right by multiplication by $\bar
w_{\sigma}(g_{\sigma})\in\bar\GM_{w\sigma}$. For any left ``$R[\AM/\GM]^{0}$-module''
$\MC$, this implies an isomorphism
$$ \Hom_{R[\bar\GM_{\tau}^{0}]}(R[\bar\GM_{\sigma\tau}^{0}],\MC_{\tau})\simto
\MC_{\tau}^{\bar\UM_{w\sigma,\tau}}$$
which identifies the map $a_{\sigma\tau}^{*}$ of \ref{cart_modules} with the map 
$$\MC_{\sigma}\simto \MC_{w\sigma}\To{} \MC_{\tau}^{\bar\UM_{w\sigma,\tau}}$$
where the first map is $\MC(\bar w_{\sigma})$ and the second one is $\MC(\bar
1_{w\sigma,\tau})$ (the image by the functor $\MC$ of the inclusion
$1_{w\sigma,\tau}:w\sigma\subseteq \tau$).
In particular, we see that $\MC$ is cartesian as a functor $[\AM/\bar\GM]^{0}\To{} R-\rm Mod$
if and only if it is cartesian as a ``$R[\AM/\bar\GM]^{0}$-module''.
\end{proof}

\begin{lem}
 Given two chambers $\delta\supset \sigma$ and $\gamma\supset \tau$, the multiplication in
 $W$ induces two bijections
$W_{\tau}^{0}\times (W_{\delta\gamma}^{\dag}\cap W_{\sigma\tau}^{\dag})\simto
W_{\sigma\tau}^{\dag}$ and
$W_{\tau}^{0}\times (W_{\delta\gamma}^{0}\cap W_{\sigma\tau}^{0})\simto
W_{\sigma\tau}^{0}$. Moreover, $W_{\delta\gamma}^{0}$ is a singleton.
\end{lem}
\begin{proof}
Since $W_{\tau}^{0}$ acts simply transitively on the set of chambers that contain
$\tau$, the multiplication map gives a bijection $W_{\tau}^{0}\times
W_{\delta\gamma}^{\dag}\simto W_{\delta\tau}^{\dag}$. But we have 
$W_{\sigma\tau}^{\dag}\subset W_{\delta\tau}^{\dag}$, and this bijection thus restricts to
a bijection
$W_{\tau}^{0}\times (W_{\delta\gamma}^{\dag}\cap W_{\sigma\tau}^{\dag})\simto
W_{\sigma\tau}^{\dag}$ and to a similar bijection with $\circ$ instead of $\dag$.
Finally,  $W^{0}$ acts simply transitively on the set of
all chambers, so $W_{\delta\gamma}^{0}$ is a singleton.
\end{proof}

\alin{Hecke ``algebras'' attached to semisimple conjugacy classes} \label{hecke_algebra_s}
Now, let $s\in\GL_{n}(\FM)$ be a semisimple conjugacy class of order invertible in $R$,
and assume $R$ is either $\oZl$ or $\oQl$. For each $\sigma\in \AM_{\bullet}$ we get a central
  idempotent $e_{\sigma}^{s}\in R[\bar\GM_{\sigma}]$. Namely, with the
  notation of \ref{defes}  we put
$$e_{\sigma}^{s}:=\sum_{t}e_{t}^{\bar\GM_{\sigma}}$$ where $t$ runs on conjugacy classes in $s\cap 
\prod_{J\in P_{\sigma}}{\rm Aut}_{\FM}(\FM^{J})$ inside ${\rm Aut}_{\FM}(\FM^{I})=\GL_{n}(\FM)$.

\begin{lem}
The idempotent $e_{\sigma}^{s}$ is central in $R[\bar\GM_{\sigma}^{\dag}]$. Moreover, for
any other  $\tau\in\AM_{\bullet}$ we have $e_{\tau}^{s} R[\bar\GM_{\sigma\tau}^{\dag}]
= e_{\tau}^{s}R[\bar\GM_{\sigma\tau}^{\dag}] e_{\sigma}^{s} = R[\bar\GM_{\sigma\tau}^{\dag}] e_{\sigma}^{s}$. 
\end{lem}
\begin{proof}
 Any $w\in W_{\sigma}^{\dag}$ acts on $\bar\GM_{\sigma}=\prod_{J\in P_{\sigma}}{\rm Aut}_{\FM}(\FM^{J})$ by
permuting $P_{\sigma}$, and this action is by conjugation by some
element in ${\rm Aut}_{\FM}(\FM^{I})$. Therefore $w$ normalizes
$e^{s}_{\sigma}$. Since
$\bar\GM_{\sigma}^{\dag}=\bar\GM_{\sigma}W_{\sigma}^{\dag}$, the first assertion follows.

For the second one, let $\Omega_{\sigma\tau}^{\dag}$ be a set of representatives of
$W_{\tau}\ba W_{\sigma\tau}^{\dag}$. Then we have a decomposition of
$\bar\GM_{\tau}\times\bar\GM_{\sigma}^{\rm opp}$ sets
$$ \bar\GM_{\sigma\tau}^{\dag}=  \bigsqcup_{w\in\Omega_{\sigma\tau}^{\dag}}
\bar\GM_{\tau}/\UM_{w\sigma,\tau},$$
where $g_{\sigma}\in\bar\GM_{\sigma}$ acts on the summand 
$\bar\GM_{\tau}/\UM_{w\sigma,\tau}$ by right multiplication by $\bar
w_{\sigma\tau}(g_{\sigma})$.
It follows that
$$ R[\bar\GM_{\sigma\tau}^{\dag}]e_{\sigma}^{s}=  
\bigoplus_{w\in\Omega_{\sigma\tau}^{\dag}}R[\bar\GM_{\tau}/\UM_{w\sigma,\tau}]e_{w\sigma}^{s}=
\bigoplus_{w\in\Omega_{\sigma\tau}^{\dag}}e_{\tau}^{s}R[\bar\GM_{\tau}/\UM_{w\sigma,\tau}]=
e_{\tau}^{s}R[\bar\GM_{\sigma\tau}^{\dag}] ,$$
where the second equality is (\ref{compat2}).
\end{proof}

The lemma shows that the composition map
$R[\bar\GM_{\tau\nu}^{\dag}]\otimes_{R}R[\bar\GM_{\sigma\tau}^{\dag}] \To{}
R[\bar\GM_{\sigma\nu}^{\dag}]$ takes the submodule
$e_{\nu}^{s}R[\bar\GM_{\tau\nu}^{\dag}]e_{\tau}^{s}\otimes_{R} e_{\tau}^{s}R[\bar\GM_{\sigma\tau}^{\dag}]e_{\sigma}^{s}$ into
$e_{\nu}^{s}R[\bar\GM_{\sigma\nu}^{\dag}]e_{\sigma}^{s}$. This justifies the following
definition :

\begin{defn}
  We denote by $e^{s}R[\AM/\bar\GM]$ the  ``$R$-algebra'' whose set of objects is
  $\AM_{\bullet}$ and Hom sets are given
  by $\Hom(\sigma,\tau)=e_{\tau}^{s} R[\bar\GM_{\sigma\tau}^{\dag}]e_{\sigma}^{s}$.
Similarly we have a ``$R$-algebra'' $e^{s}R[\AM/\bar\GM]^{0}$.
\end{defn}

Now, let $\VC=(\VC_{\sigma})_{\sigma\in\AM_{\bullet}}$ be a left
``$R[\AM/\bar\GM]$-module''. For any pair $\sigma,\tau$, the structure map 
$\VC_{\sigma}\otimes_{R}R[\bar\GM_{\sigma\tau}^{\dag}]\To{}\VC_{\tau}$ induces a map
$e_{\sigma}^{s}\VC_{\sigma}\otimes_{R}e^{s}_{\tau}R[\bar\GM_{\sigma\tau}^{\dag}]e^{s}_{\sigma}
\To{} e_{\tau}^{s}\VC_{\tau}$, which equips the collection
 $( e^{s}_{\sigma}\VC_{\sigma})_{\sigma\in\AM_{\bullet}}$ with a structure of left
``$e^{s}R[\AM/\bar\GM]$-module'' that we denote by $e^{s}\VC$. We have a decomposition 
$ \VC=\bigoplus_{s} e^{s}\VC$ of $\VC$ as a direct sum of functors
$\AM/\bar\GM\To{}R-\rm Mod$. Here $s$ runs over semisimple conjugacy classes in
$\GL_{n}(\FM)$ of order
invertible in $R$.
In other words we have a categorical decomposition
$$ R[\AM/\bar\GM]-{\rm Mod} = \prod_{s} 
e^{s}R[\AM/\bar\GM]-{\rm Mod}.$$

Putting Fact \ref{equivcoef}, Lemma \ref{idempotents_building} and Theorem
\ref{hecke_level_0} together, we obtain the following result.

\begin{thm}
  The functor $V\mapsto \VC$ of \ref{hecke_level_0} induces an equivalence of categories from
$ \Rep^{s}_{R}(G)$ to the category of left ``$e^{s}R[\AM/\bar\GM]$-modules'' which are
cartesian as ``$e^{s}R[\AM/\bar\GM]^{0}$-modules''.
\end{thm}

Since the construction of  $e^{s}R[\AM/\bar\GM]$ only involves $\FM$, $n$ and $s$, 
we infer :

\begin{cor}
  Let $F'$ be another local field with residue field isomorphic to $\FM$. The categories
  $\Rep^{s}_{R}(\GL_{n}(F))$ and $\Rep^{s}_{R}(\GL_{n}(F'))$ are equivalent. 
\end{cor}

In particular, taking $s=1$ and $F'$ an unramified extension of $F$, 
  we have constructed an equivalence of categories such as predicted in \ref{main} i).

\subsection{Morita equivalence of Hecke ``algebras''}

In this section, we assume that $n=n'f$ and we denote by $\FM'$ the extension of $\FM$ of
degree $f$. To $(n',\FM')$  we associate the enriched Coxeter complex $[\AM'/\bar\GM']$ as in
\ref{abstract_appart}. In particular, $\AM'=\ZM^{I'}/\ZM_{\rm diag}$, where
$I'=\{1,\cdots, n'\}$, and $W'=\ZM^{I'}\rtimes\SG_{n'}$.
Next, to any  semisimple conjugacy class $s'\in \GL_{n'}(\FM')$ with order invertible
in $R$, we associate the Hecke ``$R$-algebra'' $e^{s'}R[\AM'/\bar\GM']$ as in \ref{hecke_algebra_s}. 

By choosing a basis of $\FM'$ over $\FM$, we get an embedding of $\GL_{n'}(\FM')$ into
$\GL_{n}(\FM)$, and we denote by $s$ the unique conjugacy class  in $\GL_{n}(\FM)$ that
contains $s'$. Note that $s$ does not depend on this choice of basis.

\begin{center}
$(*)$  \emph{We assume that the centralizer in $\GL_{n}(\FM)$ of an element of 
    $s'$ is contained in $\GL_{n'}(\FM')$.}
\end{center}
For simplicity, put
$\RC:=e^{\sG}R[\AM/\bar\GM]$ and
$\RC':=e^{\sG'}R[\AM'/\bar\GM']$.
Our aim is to prove that the categories $\RC-\rm Mod$ and $\RC'-\rm Mod$ are equivalent,
under the above assumption.
For this, we will construct a $(\RC,\RC')$-bimodule $\PC$ and show it
induces a Morita equivalence in the sense of Definition \ref{def_Morita}.

\alin{Embedding $[\AM'/W']$ in $[\AM/W]$} 
We fix a surjective map $\pi: I\To{} I'$ with fibers of cardinality $f$.
Dually, this map induces an embedding $\iota :\ZM^{I'}\injo \ZM^{I}$, as well as an
embedding  $\SG_{n'}\injo \SG_{n}$ which sends a permutation $w'$ to
   the unique permutation $w=\iota(w')$ which induces an increasing bijection
   $\pi^{-1}(r')\simto\pi^{-1}(w'(r'))$  for each $r'\in I'$.
In turn, these two maps induce 
 \begin{itemize}
 \item a group embedding $\iota:\,W'\injo W$, and
 \item a simplicial, dimension preserving, embedding
   $\iota:\,\AM'_{\bullet}\injo\AM_{\bullet}$, which is equivariant
    with respect to the previous embedding.
\end{itemize}

Therefore we have obtained an embedding of categories $[\AM'/W']\To{}[\AM/W]$.
Note that for any $\sigma'\in\AM'_{\bullet}$, we have $P_{\iota\sigma'}=\pi^{-1}(P_{\sigma'})$. 

\alin{Gradings}
  Consider the map $\rho:W\To{}\QM$, $((x_{1},\cdots,x_{n}),w)\mapsto
  \sum_{i=1}^{n}x_{i}$. Then $\rho$ is a group morphism with kernel $W^{0}$ and image
  $\frac 1n\ZM$. For $i\in \QM$, we will put  $W^{i}:=\rho^{-1}(i)$. This notation will
  propagate to all notation involving $W$, 
e.g. we put $W^{i}_{\sigma\tau}=W^{\dag}_{\sigma\tau}\cap W^{i}$. Note that Lemma \ref{hecke_level_0}
tells us that if $W^{i}_{\sigma\tau}$ is not empty, it is a principal homogeneous $W_{\tau}^{0}$-set.
In particular, the multiplication map induces a bijection
$$  W_{\tau\nu}^{j} \times_{W_{\tau}^{0}} W_{\sigma\tau}^{i}\simto W_{\sigma\nu}^{i+j}.$$

For $\sigma,\tau\in\AM_{\bullet}$, the composition $\bar\GM_{\tau}\times
W_{\sigma\tau}^{\dag}\To{}W_{\sigma\tau}^{\dag}\To{\rho}\QM$ factors through
$\bar\GM_{\sigma\tau}^{\dag}$, and we will denote by $\bar\GM_{\sigma\tau}^{i}$ the
inverse image of $i$. Note that $\bar\GM_{\sigma\tau}^{i}$ is also the quotient set of 
$\bar\GM_{\tau}\times W_{\sigma\tau}^{i}$ by the equivalence relation of
\ref{abstract_appart}. If $w\in W_{\sigma\tau}^{i}$, then
$\bar\GM_{\sigma\tau}^{i}\simeq\bar\GM_{\tau}/\bar\UM_{w\sigma,\tau}$. In particular we
see that the multiplication map induces a bijection
$$\bar\GM_{\tau\nu}^{j}\times_{\bar\GM_{\tau}}\bar\GM_{\sigma\tau}^{i}\simto \bar\GM_{\sigma\nu}^{i+j}.$$

The disjoint union $\bar\GM_{\sigma\tau}^{\dag}=\bigsqcup_{i}\bar\GM_{\sigma\tau}^{i}$
induces a grading
$R[\bar\GM_{\sigma\tau}^{\dag}]=\bigoplus_{i}R[\bar\GM_{\sigma\tau}^{i}]$, and in turn a
grading $\RC_{\sigma,\tau}=\bigoplus_{i}\RC_{\sigma,\tau}^{i}$. Note that
$\RC^{0}=e^{s}R[\AM/\bar\GM]^{0}$ so this notation is consistent with the former one.
The multiplication map 
$\RC_{\sigma,\tau}\otimes\RC_{\tau,\nu}\To{}\RC_{\sigma,\nu}$ takes 
$\RC_{\sigma,\tau}^{i}\otimes\RC_{\tau,\nu}^{j}$ into $\RC^{i+j}_{\sigma,\nu}$, so we may
think of $\RC$ as a graded ``ring''. Actually, the above bijection shows that the same
multiplication map induces an isomorphism
$$ \RC_{\sigma,\tau}^{i}\otimes_{\RC_{\tau}^{0}}\RC_{\tau,\nu}^{j}\simto\RC_{\sigma,\nu}^{i+j}.$$

Similarly, there is a map $\rho':\, W'\To{}\QM$ and we will use similar notation as above
for $W'$ and $\RC'$. We have $\rho'=\rho\circ\iota$, so that in particular we have
$\iota(W^{'i})\subset W^{i}$.

\begin{lemme} \label{lemmeW}
  Let $\sigma',\tau'\in \AM'_{\bullet}$. If $W^{i}_{\iota\sigma',\iota\tau'}\neq \emptyset$, then 
$W^{'i}_{\sigma',\tau'}\neq \emptyset$.
\end{lemme}
\begin{proof}
  After maybe conjugating $\sigma'$ by some element in $W^{'0}$ we may assume that $\sigma'$
  and $\tau'$  are contained in a chamber $\delta'$. Let
  $c'$ be the unique element in $W_{\delta'}^{1/n'}$. Then 
$W^{'i}_{\sigma',\tau'}\neq\emptyset \Leftrightarrow \tau'\subset c^{'in'}\sigma'$, because $W_{\tau'}^{0}$ acts transitively on the set of chambers that contain  $\tau'$.
Let us choose a vertex $x'_{0}\in \delta'$, whence a unique labelling
$\delta'=\{x'_{0},\cdots,x'_{n'-1}\}$, as in \ref{abstract_appart}. On this labelling, 
 $c'$ is the circular permutation. Let us
identify $\delta'$ to $\ZM/n'\ZM$ via $x'_{k}\mapsto k \,({\rm mod}\, n')$. This
identifies $c'$ to the map $r\mapsto r+1$. Then $\sigma'$ and $\tau'$ correspond to
subsets $S',T'$ of $\ZM/n'\ZM$ and we have 
$$ W^{'i}_{\sigma',\tau'}\neq\emptyset \Leftrightarrow T'\subset S'+in'.$$

Now let $\delta$ be a chamber in $\AM_{\bullet}$ that contains $\iota\delta'$. The vertex
$x_{0}:=\iota x'_{0}$ gives an ordering $\delta=\{x_{0},\cdots,x_{n-1}\}$ and we have
$\iota x'_{k}= x_{kf}$. Let us identify $\delta$ to $\ZM/n\ZM$ as above. This identifies
the inclusion $\iota\delta'\subset \delta$ to the multipication by $f$ map $\ZM/n'\ZM\injo
\ZM/n\ZM$, hence the simplex $\iota\sigma'$ corresponds to $fS'$ and the simplex $\iota
\tau'$ to $fT'$. As above, we have
$$ W^{i}_{\iota\sigma',\iota\tau'}\neq\emptyset \Leftrightarrow fT'\subset fS'+in,$$
which is equivalent to $T'\subset S'+in'$ since $n=fn'$.
\end{proof}

\alin{Embedding $[\AM'/\bar\GM']$ in $[\AM/\bar\GM]$} 
 Next, we fix an $\FM$-basis $(\alpha_{1},\cdots,\alpha_{f})$ of $\FM'$. 
We associate to it the
$\FM$-linear isomorphism $\iota:\, {\FM'}^{I'}\simto \FM^{I}$ defined as follows.
Let $\eta : I'\times\{1,\cdots,f\}\simto I$ be the unique bijection which commutes with
the projections to $I'$ (i.e. $\pi\circ \eta$ coincides with the first projection
$\pi_{1}$) 
and which is increasing on each fiber of $\pi_{1}$. Then we let $\iota$ be the
$\FM$-linear map which takes $\alpha_{j}e'_{i'}$ to $e_{\eta(i',j)}$, where
 $(e_{i})_{i\in I}$ and $(e'_{i'})_{i'\in I'}$ are the
respective canonical basis of $\FM^{I}$ and ${\FM'}^{I'}$. This isomorphism induces in
turn an embedding
$\iota_{*} : {\rm Aut}_{\FM'}({\FM'}^{I'})\injo {\rm Aut}_{\FM}(\FM^{I})$ and we note that this
embedding restricts to the embedding $\iota:\,\SG_{n'}\injo \SG_{n}$ defined earlier.

For all $J'\subset I'$, the isomorphism $\iota$ restricts to 
  ${\FM'}^{J'}\simto\FM^{\pi^{-1}(J')}$, and therefore, for all
  $\sigma'\in\AM'_{\bullet}$,  $\iota_{*}$ takes
  $\bar\GM'_{\sigma'}$ into $\bar\GM_{\iota\sigma'}$, compatibly with the previously defined
  $W'\injo W$. In other words we now have  an embedding of categories
  $[\AM'/\bar\GM']\injo[\AM/\bar\GM]$.

\alin{Varieties}
As  explained in \ref{varietiesgln}, we can now see $\bar\GM_{\sigma'}$
 as a ``twisted'' Levi subgroups of $\bar\GM_{\iota\sigma'}$ with a
 ``canonical'' choice of a parabolic subgroup, and therefore we get a
 Deligne-Lusztig variety $Y_{\sigma'}$ with an action of
 $\bar\GM_{\sigma}\times\bar\GM_{\sigma'}^{\rm opp}$. We then define
$$ \PC_{\sigma'}^{0}:= H^{\rm
  dim}_{c}(Y_{\sigma'},\Lambda)e^{\bar\GM'_{\sigma}}_{s'}
\,\,\hbox{ in }\,\, 
\RC_{\iota\sigma'}^{0}\otimes_{\Lambda}
({\RC'}^{0}_{\sigma'})^{\rm opp}-\rm Mod
$$
\begin{lemme} \label{keylemma}
There is a family of isomorphisms
$(\phi_{\sigma',\tau'})_{\sigma',\tau'\in\AM'_{\bullet}}$ of
$(\RC_{\iota\tau'}^{0},{\RC'}_{\sigma'}^{0})$-bimodules
 $$\phi_{\sigma',\tau'}:\,\, \PC_{\sigma'}^{0}\otimes_{\RC_{\iota\sigma'}^{0}} \RC_{\iota\sigma',\iota\tau'}
\simto
 \RC'_{\sigma'\tau'}\otimes_{{\RC}'^{0}_{\tau'}}
 \PC_{\tau'}^{0}$$
which satisfy the following transitivity property. For any other
$\nu'\in\AM'_{\bullet}$, the following diagram
$$\xymatrix{ 
&
\RC'_{\sigma',\tau'}\otimes_{{\RC'}^{0}_{\tau'}}  \PC_{\tau'}^{0} \otimes_{\RC^{0}_{\iota\tau'}} \RC_{\iota\tau',\iota\nu'}
\ar[rd]^{\id_{\RC'_{\sigma'\tau'}}\otimes\phi_{\tau',\nu'}}
&
\\
\PC_{\sigma'}^{0}\otimes_{\RC_{\iota\sigma'}^{0}}
\RC_{\iota\sigma',\iota\tau'} \otimes_{\RC_{\iota\tau'}^{0}}
\RC_{\iota\tau',\iota\nu'} \ar[ru]^-{\phi_{\sigma',\tau'}\otimes \id_{\RC_{\iota\tau',\iota\nu'}}\,\,\,\,\;\;} \ar[d] 
&  &  \RC'_{\sigma'\tau'}\otimes_{\RC^{0}_{\tau'}}
 \RC'_{\tau'\nu'}\otimes_{\RC^{0}_{\nu'}}
 \PC_{\nu'}^{0} \ar[d] \\
\PC_{\sigma'}^{0}\otimes_{\RC_{\iota\sigma'}^{0}}\RC_{\iota\sigma',\iota\nu'}
\ar[rr]_{\phi_{\sigma',\nu'}} 
& &
\RC'_{\sigma'\nu'}\otimes_{\RC^{0}_{\nu'}}
 \PC_{\nu'}^{0} 
}$$
is commutative, where the vertical maps are given by composition in
$\RC$ and $\RC'$.
\end{lemme}
\begin{proof}
\emph{Step 1.   Assume that $\tau'\subset\sigma'$.} In this case
we are in the setting of \ref{parabolicgln}. Namely, 
the parabolic subgroup
$\bar\PM_{\iota\sigma',\iota\tau'}=\bar\GM_{\iota\sigma'}\bar\UM_{\iota\sigma,\iota\tau'}$
of $\bar\GM_{\iota\tau'}$ is defined by the image under $\iota_{*}$ of the flag that
defines the parabolic subgroup
$\bar\PM'_{\sigma',\tau'}=\bar\GM'_{\sigma'}\bar\UM'_{\sigma',\tau'}$ of
$\bar\GM'_{\tau'}$. In particular we have
$\bar\PM'_{\sigma',\tau'}=\bar\PM_{\iota\sigma',\iota\tau'}\cap\bar\GM'_{\tau'}$.
Since moreover we have
$\RC^{0}_{\iota\sigma',\iota\tau'}=R[\bar\GM_{\iota\tau'}/\bar\UM_{\iota\sigma',\iota\tau'}]$ and
${\RC'}^{0}_{\sigma',\tau'}=R[\bar\GM'_{\tau'}/\bar\UM'_{\sigma',\tau'}]$, Proposition 
\ref{prop_finite} provides us with
 an isomorphism of $(\RC_{\iota\tau'}^{0},{\RC'}_{\sigma'}^{0})$-bimodules\footnote{Note
   that the tensor products have been switched, because of our convention to write ``left
   modules'' on the right (in accordance with the composition of morphisms in a category
   rather than with the action of a usual ring on a module, see \ref{modules})}
$$ \phi_{\sigma',\tau'}^{0}:\, 
\PC_{\sigma'}^{0}\otimes_{\RC_{\iota\sigma'}^{0}} \RC^{0}_{\iota\sigma',\iota\tau'}
\simto
 {\RC'}^{0}_{\sigma'\tau'}\otimes_{{\RC}'^{0}_{\tau'}}
 \PC_{\tau'}^{0}$$
Moreover, if $\nu'\subset \tau'$, the following diagram commutes by statement ii) of the
same proposition
$$\xymatrixcolsep{6cm}\xymatrix{
\PC_{\sigma'}^{0}\otimes_{\RC_{\iota\sigma'}^{0}}
\RC^{0}_{\iota\sigma',\iota\tau'} \otimes_{\RC_{\iota\tau'}^{0}}
\RC^{0}_{\iota\tau',\iota\nu'} \ar[r]^{(\id_{\RC^{0}_{\sigma'\tau'}}\otimes\phi^{0}_{\tau',\nu'}) \circ
(\phi^{0}_{\sigma',\tau'}\otimes \id_{\RC^{0}_{\iota\tau',\iota\nu'}})} \ar[d]_{\simeq} 
&   \RC^{0}_{\sigma'\tau'}\otimes_{\RC^{0}_{\tau'}}
 \RC^{0}_{\tau'\nu'}\otimes_{\RC^{0}_{\nu'}}
 \PC_{\nu'}^{0} \ar[d]^{\simeq} \\
\PC_{\sigma'}^{0}\otimes_{\RC_{\iota\sigma'}^{0}}\RC^{0}_{\iota\sigma',\iota\nu'}
\ar[r]_{\phi_{\sigma',\nu'}^{0}} 
&
\RC^{0}_{\sigma'\nu'}\otimes_{\RC^{0}_{\nu'}}
 \PC_{\nu'}^{0} 
}$$
Note that in this diagram the vertical maps are isomorphisms, so that we can express it as
an equality 
$$\phi_{\sigma',\nu'}^{0} =
(\id_{\RC^{0}_{\sigma'\tau'}}\otimes\phi^{0}_{\tau',\nu'}) \circ
(\phi_{\sigma',\tau'}^{0}\otimes \id_{\RC^{0}_{\iota\tau',\iota\nu'}}).$$

\medskip

\emph{Step 2.
Assume that $\tau'\in{W'}^{i}\cdot\sigma'$ for some $i\in\QM$.}
Pick $w\in {W'}^{i}$, such that $\tau'=w\sigma'$. We also denote by
$w$ the image $\iota(w)$ of $w$ in $W$.
Conjugation by the permutation matrix $[\bar w]$ in ${\rm Aut}_{\FM}(\FM^{I})$
induces an isomorphism of varieties $\bar w:\, Y_{\sigma'}\simto Y_{\tau'}$
which is equivariant with respect to the isomorphisms 
$\bar w_{\sigma'}:\bar\GM_{\sigma'}\simto\bar\GM_{\tau'}$ and 
$\bar w_{\iota\sigma'}:\bar\GM_{\iota
  \sigma'}\simto\bar\GM_{\iota\tau'}$. 
Recall that the latter isomorphisms are also given by conjugation under $[w]$ in ${\rm
  Aut}_{\FM}(\FM^{I})$, and coincide with the ``outer-conjugacy''  by $w$ seen
as an isomorphism $\sigma'\simto \tau'$ in the category $[\AM'/\bar\GM']$,
 resp.
 $\iota\sigma'\simto \iota\tau'$ in the category $[\AM/\bar\GM]$.
Recall also that the maps 
$$
\application{}{{\bar\GM}^{'-i}_{\tau',\sigma'}}
{\bar\GM'_{\tau'}}{x'}{wx'}
\hbox{ and }
\application{}{{\bar\GM}^{i}_{\iota\sigma',\iota \tau'}}{\bar\GM_{\iota \tau'}}
{x}{x w^{-1}}
$$
are bijections. It follows that the map $(x',y,x)\mapsto (x w^{-1})\cdot (\bar w y)\cdot (wx')$ is a
$\bar\GM'_{\iota \tau'}\times\bar\GM^{'\rm opp}_{\tau'}$-equivariant isomorphism
$$\alpha_{w}:\, {\bar\GM}^{'-i}_{\tau',\sigma'} \times_{\bar\GM'_{\sigma'}} Y_{\sigma'} 
\times_{\bar\GM_{\iota\sigma'}} \bar\GM^{i}_{\iota\sigma',\iota \tau'}
\simto Y_{\tau'}.$$
Moreover, let $v\in W_{\sigma'}^{0}$. We have
$\alpha_{wv}(x',y,x)=\alpha_{w}(vx', \bar v y, xv^{-1})$.
But the automorphism $\bar v$ of $Y_{\sigma'}$ is given by the action
of $([v],[v]^{-1})\in \bar\GM_{\iota\sigma'}\times\bar\GM_{\sigma'}$. 
 Therefore, $(vx', \bar v y, xv^{-1})=(x',y,x)$ in the contracted
 product above, and hence
we have $\alpha_{wv}=\alpha_{w}$.
Since ${W'}_{\sigma'}^{0}$ acts simply transitively on ${W'}_{\sigma'\tau'}^{i}$, we see
that $\alpha_{w}$ does not depend on the choice of $w$. We denote it by $\alpha_{i}$.
Taking the contracted product on the left with
$\bar\GM^{'i}_{\sigma',\tau'}$ and passing to cohomology we get an isomorphism
of $(\RC_{\iota\tau'}^{0},{\RC'}_{\sigma'}^{0})$-bimodules
$$ \phi_{\sigma',\tau'}^{i}:\, 
\PC_{\sigma'}^{0}\otimes_{\RC_{\iota\sigma'}^{0}} \RC^{i}_{\iota\sigma',\iota\tau'}
\simto
 {\RC'}^{i}_{\sigma'\tau'}\otimes_{{\RC}'^{0}_{\tau'}}
 \PC_{\tau'}^{0}$$

These isomorphisms compose as expected. Namely, taking into account that the
multiplication map induces isomorphisms like
$\RC^{i}_{\sigma,\tau}\otimes_{\RC^{0}_{\tau}} \RC^{j}_{\tau,\nu}\simto
\RC^{i+j}_{\sigma,\nu}$, we have for any $\nu'\in {W'}^{j}\tau'$ the following equality
$$ \phi^{i+j}_{\sigma',\nu'}=
(\id_{\RC^{i}_{\sigma'\tau'}}\otimes\phi^{j}_{\tau',\nu'}) \circ
(\phi_{\sigma',\tau'}^{i}\otimes \id_{\RC^{j}_{\iota\tau',\iota\nu'}}).$$


\medskip

\emph{Step 3.Compatibility between steps 1 and 2.} 
Here we pick two simplices $\tau'\subset \sigma'$ and an element $w\in {W'}^{i}$. 
In order to simplify the notation a bit, we put $\sigma=\iota\sigma'$ and $\tau=\iota\tau'$. 
We have a commutative diagram
$$\xymatrixcolsep{2cm}\xymatrix{
\PC^{0}_{\sigma'}\otimes_{{\RC}^{0}_{\sigma}} \RC^{0}_{\sigma,\tau} 
\ar[r]^{\phi^{0}_{\sigma',\tau'}}  \ar[d]_{[w]\otimes (w \circ - \circ w^{-1})}^{\simeq}
&
{\RC'}^{0}_{\sigma',\tau'}\otimes_{{\RC'}^{0}_{\tau'}}\PC^{0}_{\tau'} 
\ar[d]^{ (w \circ - \circ w^{-1})\otimes [w]}_{\simeq} \\
\PC^{0}_{w\sigma'}\otimes_{{\RC}^{0}_{w\sigma}} \RC^{0}_{w\sigma, w\tau} 
\ar[r]_{\phi^{0}_{w\sigma',w\tau'}}  
&
{\RC'}^{0}_{w\sigma',w\tau'}\otimes_{{\RC'}^{0}_{w\tau'}}\PC^{0}_{w\tau'}
}$$
where the vertical maps are given by conjugation by $[w]$ in ${\rm
  Aut}_{\FM}(\FM^{I})$. As in step $2$, these vertical maps are equivariant with respect to conjugation by
$[w]$ and we make them equivariant by tensoring  the first line on the left by
${\RC'}^{-i}_{w\sigma',\sigma'}$ and on the right by $\RC^{i}_{\tau,w\tau}$, as follows :
$$\xymatrixcolsep{2.5cm}\xymatrix{
{\RC'}^{-i}_{w\sigma',\sigma'}\otimes_{{\RC'}_{\sigma'}^{0}}
\PC^{0}_{\sigma'}\otimes_{{\RC}^{0}_{\sigma}} \RC^{0}_{\sigma,\tau} 
\otimes_{\RC^{0}_{\tau}}\RC^{i}_{\tau,w\tau}
\ar[r]^-{\id\otimes\phi^{0}_{\sigma',\tau'}\otimes\id}  \ar[d]_{w\otimes [w]\otimes (w \circ - \circ
  w^{-1})\otimes w^{-1}}^{\simeq}
&
{\RC'}^{-i}_{w\sigma',\sigma'}\otimes_{{\RC'}_{\sigma'}^{0}}
{\RC'}^{0}_{\sigma',\tau'}\otimes_{{\RC'}^{0}_{\tau'}}\PC^{0}_{\tau'} 
\otimes_{\RC^{0}_{\tau}}\RC^{i}_{\tau,w\tau}
\ar[d]^{w\otimes (w \circ - \circ w^{-1})\otimes [w]\otimes w^{-1}}_{\simeq} \\
\PC^{0}_{w\sigma'}\otimes_{{\RC}^{0}_{w\sigma}} \RC^{0}_{w\sigma,w\tau} 
\ar[r]_{\phi^{0}_{w\sigma',w\tau'}}  
&
{\RC'}^{0}_{w\sigma',w\tau'}\otimes_{{\RC'}^{0}_{w\tau'}}\PC^{0}_{w\tau'}
}$$
Now we tensor on the left the whole diagram by ${\RC'}^{i}_{\sigma',w\sigma'}$ and we
simplify terms using isomorphisms of the type
$\RC^{i}_{\sigma\tau}\otimes_{\RC^{0}_{\tau}}\RC^{j}_{\tau\nu}\simto \RC^{i+j}_{\sigma\nu}$
given by the multiplication.
We get a commutative diagram
$$\xymatrixcolsep{1.5cm}\xymatrix{
 \PC^{0}_{\sigma'} \otimes_{{\RC}^{0}_{\sigma}} \RC^{i}_{\sigma,w\tau} 
\ar[rr]^-{\phi_{\sigma',\tau'}^{0}\otimes\id_{\RC^{i}_{\tau,w\tau}}}
\ar[d]_{\phi_{\sigma',w\sigma'}^{i}\otimes\id_{\RC^{0}_{w\sigma,w\tau}}} && 
{\RC'}^{0}_{\sigma',\tau'}\otimes_{{\RC'}^{0}_{\tau'}}\PC^{0}_{\tau'} 
\otimes_{\RC^{0}_{\tau}}\RC^{i}_{\tau,w\tau} 
\ar[d]^{\id_{\RC^{0}_{w\sigma',w\tau'}}\otimes\phi_{\tau',w\tau'}^{i}} \\
{\RC'}^{i}_{\sigma',w\sigma'}\otimes_{{\RC'}^{0}_{w\sigma'}}
\PC^{0}_{w\sigma'}\otimes_{{\RC}^{0}_{w\sigma}} \RC^{0}_{w\sigma,w\tau} 
\ar[rr]_-{\id_{\RC^{'i}_{\sigma',w\sigma'}}\otimes\phi_{w\sigma',w\tau'}^{i}} &&
{\RC'}^{i}_{\sigma',w\tau'}\otimes_{{\RC'}^{0}_{w\tau'}}\PC^{0}_{w\tau'}
}$$
Note that the objects and  morphisms in this diagram only depend on $w\sigma'$ and $w\tau'$,
not on $w$.

\medskip

\emph{Step 4. The desired maps.}
Let now $\sigma',\tau'\in\AM'_{\bullet}$ be any pair of simplices. Assume that
${W'}^{i}_{\sigma',\tau'}\neq \emptyset$. So there is a simplex $\nu'\in {W'}^{i}\sigma'$
such that $\tau'\subset \nu'$. Consider the composition
$$ (\id_{{\RC'}^{i}_{\sigma'\nu'}}\otimes\phi^{0}_{\nu',\tau'}) \circ
(\phi^{i}_{\sigma',\nu'}\otimes \id_{\RC^{0}_{\iota\nu',\iota\tau'}}):\,
\PC_{\sigma'}^{0}\otimes_{\RC_{\iota\sigma'}^{0}} \RC^{i}_{\iota\sigma',\iota\tau'}
\simto
 {\RC'}^{i}_{\sigma'\tau'}\otimes_{{\RC}'^{0}_{\tau'}}
 \PC_{\tau'}^{0}$$
We will show it is independent of the choice of $\nu'$. So, let $\mu'$ be another
simplex in ${W'}^{i}\sigma'$ which contains $\tau'$. By  Step 2 we have a factorization
$\phi_{\sigma',\mu'}^{i}=(\id_{{\RC'}^{i}_{\sigma',\nu'}}\otimes\phi^{0}_{\nu',\mu'})\circ
(\phi^{i}_{\sigma',\nu'}\otimes\id_{\RC^{0}_{\iota\nu',\iota\mu'}})$. Thanks to this
factorization, we are left to prove the following equality
$$ \phi^{0}_{\nu',\tau'} = (\id_{{\RC'}^{0}_{\nu',\mu'}}\otimes\phi^{0}_{\mu',\tau'})\circ
(\phi^{0}_{\nu',\mu'}\otimes\id_{\RC^{0}_{\iota\mu',\iota\tau'}})
$$
Now, since $\mu'\in {W'}^{0}\nu'$ and both contain $\tau'$, there is $w\in {W'}_{\tau'}^{0}$ such that $\mu'=w\nu'$.
Applying the last diagram of Step 3, with $(\nu',\tau')$ instead of $(\sigma',\tau')$ and $i=0$, and
taking into account that $w\tau'=\tau'$, we
get exactly the desired equality.

\medskip

\emph{Step 5. Transitivity.} Pick three simplices $\sigma',\tau',\nu'$. Suppose that
${W'}^{i}_{\sigma'\tau'}$ and ${W'}^{j}_{\tau'\nu'}$ are non-empty and pick $w$, resp. $v$, in
these sets. We then have a diagram
$$ \xymatrixcolsep{2cm}
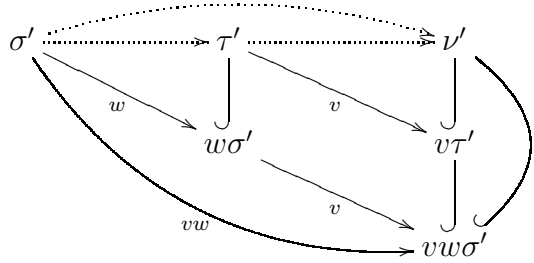
\xymatrix{
\sigma'  \ar@{..>}[r]  \ar[rd]_{w}  \ar@/_1cm/[rrdd]_{vw} \ar@/^0.5cm/@{.>}[rr]
& \tau'  \ar@{..>}[r] \ar[rd]_{v} & \nu' \\
 & w\sigma' \ar@{^{(}-}[u] \ar[rd]_{v} &  v\tau' \ar@{^{(}-}[u] \\
& & vw \sigma' \ar@{^{(}-}[u] \ar@/_1cm/@{^{(}-}[uu]
}$$
We want to check that the morphisms associated to the dotted arrows in Step 4 compose as
expected $ \phi^{i+j}_{\sigma',\nu'}=
(\id_{\RC^{i}_{\sigma'\tau'}}\otimes\phi^{j}_{\tau',\nu'}) \circ
(\phi_{\sigma',\tau'}^{i}\otimes \id_{\RC^{j}_{\iota\tau',\iota\nu'}}).$ By definition, each one of them is the composition of the morphism associated to the diagonal arrow
in Step 2 and the one associated to the vertical arrow in Step 1, along the
corresponding ``triangle''. We saw in Step 2  that the morphisms associated to diagonal
arrows compose as expected. We saw in Step 1 that the morphisms associated to vertical
arrows compose as expected. Therefore it is sufficient to check that the diagram of
morphisms associated to the parallelogram commutes. But this is exactly what Step 3 tells us.

\medskip
\emph{Step 6. Conclusion.} Recall the decompositions
 $\RC_{\iota\sigma',\iota\tau'}=\bigoplus_{i\in\QM}\RC^{i}_{\iota\sigma',\iota\tau'}$ and
$\RC'_{\sigma',\tau'}=\bigoplus_{i\in\QM}\RC^{'i}_{\sigma',\tau'}$.
Define $\phi_{\sigma',\tau'}:=\bigoplus_{i\in \frac 1n \ZM}\phi^{i}_{\sigma',\tau'}$, where
we agree that $\phi^{i}_{\sigma',\tau'}=0$ whenever $W^{'i}_{\sigma'\tau'}=\emptyset$ (and
in particular if $i\notin \frac 1{n'}\ZM$).
The only point that remains to be checked is that $\phi_{\sigma',\tau'}$ is actually an
isomorphism, and for this we need to show that for all $i\in\QM$, we have
$W^{i}_{\iota\sigma',\iota\tau'}\neq\emptyset\Rightarrow W^{'i}_{\sigma',\tau'}\neq\emptyset$.
But this is Lemma \ref{lemmeW}.
\end{proof}

\alin{Construction of a $(\RC,\RC')$-bimodule} \label{cnstruc_bimodule}
For any $\sigma'\in\AM'_{\bullet}$ and $\sigma\in\AM_{\bullet}$ we put
$$ \PC_{\sigma',\sigma} := \PC_{\sigma'}^{0}\otimes_{\RC_{\iota\sigma'}^{0}}
\RC_{\iota\sigma',\sigma} .$$
The composition maps $\RC_{\iota\sigma',\sigma}\otimes\RC_{\sigma,\tau}\To{}\RC_{\iota\sigma',\tau}$
induce an obvious left $\RC$-module structure on the collection
$(\PC_{\sigma',\sigma})_{\sigma\in\AM_{\bullet}}$. On the other hand, define a map
$\RC'_{\sigma',\tau'}\otimes \PC_{\tau',\tau} \To{} \PC_{\sigma',\tau}$ as the following
composition of maps :
\begin{eqnarray*}
  \RC'_{\sigma',\tau'}\otimes \PC_{\tau',\tau} &\To{} &
\RC'_{\sigma',\tau'}\otimes_{\RC^{0}_{\tau'}} \PC_{\tau',\tau} =
\RC'_{\sigma',\tau'}\otimes_{\RC^{0}_{\tau'}}
\PC_{\tau'}^{0}\otimes_{\RC^{0}_{\iota\tau'}}\RC_{\iota\tau',\tau} \\
& \simto & \PC_{\sigma'}^{0}\otimes_{\RC^{0}_{\iota\sigma'}}\RC_{\iota\sigma',\iota\tau'}
\otimes_{\RC^{0}_{\iota\tau'}}\RC_{\iota\tau',\tau} \\
&\To{} & \PC_{\sigma'}^{0}\otimes_{\RC^{0}_{\iota\sigma'}}\RC_{\iota\sigma',\tau}
=\PC_{\sigma',\tau}
\end{eqnarray*}
In the second line, the isomorphism is
$\phi_{\sigma',\tau'}^{-1}\otimes\id_{\RC_{\iota\tau',\tau}}$. 
The big diagram in Lemma \ref{keylemma} insures that these action maps are transitive, so that we get
a right $\RC'$-module structure on the collection
$(\PC_{\tau',\tau})_{\tau'\in\AM_{\bullet}'}$.
On the other hand this structure clearly commutes with the left action of $\RC$, so we
have constructed a $(\RC,\RC')$-bimodule structure on the collection
$(\PC_{\sigma',\sigma})_{\sigma',\AM'_{\bullet},\sigma\in\AM_{\bullet}}$, in the sense of \ref{bimodules}.
We denote it by $\PC$ and we refer to Definition \ref{def_Morita} for the notion of Morita equivalence
in this context.

\begin{theo} \label{mainthm}
The bimodule  $\PC$ induces a Morita equivalence between $\RC$ and $\RC'$.
\end{theo}
\begin{proof}
  We show that the hypothesis of Proposition \ref{equiv_induced} are satisfied. First of
  all, for all $\sigma'\in \AM'_{\bullet}$, the $\RC$-module $\sigma\mapsto
  \PC_{\sigma',\sigma}$ is, by construction, induced from $\iota\sigma'$ (in the sense of
  Definition \ref{def_induced}). Next, the isomorphisms 
$\phi_{\sigma',\tau'}: \PC_{\sigma',\iota\tau'}\simto \RC_{\sigma',\tau'}\otimes_{\RC^{0}_{\tau'}}\PC^{0}_{\tau'}$ 
and
$\phi_{\tau',\tau'}: \PC_{\tau',\iota\tau'}\simto
\RC_{\tau'}\otimes_{\RC^{0}_{\tau'}}\PC^{0}_{\tau'}$ combine to an isomorphism
$\PC_{\sigma',\iota\tau'}\simto
\RC_{\sigma',\tau'}\otimes_{\RC_{\tau'}}\PC_{\tau',\iota\tau'}$, which  
tells us that $\sigma'\mapsto \PC_{\sigma',\iota\tau'}$ is
  induced from $\tau'$. Therefore, hypothesis i) of \ref{equiv_induced} is satisfied.
Now, condition ii) of \ref{equiv_induced} is that $\PC_{\sigma',\iota\sigma'}$ should induce a
Morita equivalence between $\RC_{\iota\sigma'}$ and $\RC'_{\sigma'}$. To show this, consider
the evaluation map for any $\RC_{\iota\sigma'}$-module $M$
$$ \PC_{\sigma',\iota\sigma'} \otimes_{\RC_{\sigma'}}
\Hom_{\RC'_{\iota\sigma'}}(\PC_{\sigma',\iota\sigma'},
M)\To{}M.$$
Since
$\PC_{\sigma',\iota\sigma'}=\PC_{\sigma'}^{0}\otimes_{\RC^{0}_{\iota\sigma'}}\RC_{\iota\sigma'}
\simeq \RC'_{\sigma'}\otimes_{{\RC'}^{0}_{\sigma'}}\PC_{\sigma'}^{0}$,
it coincides with the evaluation map
$$ \PC_{\sigma'}^{0} \otimes_{\RC_{\sigma'}^{0}}
\Hom_{\RC_{\iota\sigma'}^{0}}(\PC_{\sigma'}^{0},
M)\To{}M$$
which is indeed an isomorphism since
$\PC_{\sigma'}^{0}$ induces a Morita equivalence  between
$\RC_{\iota\sigma'}^{0}$ and $\RC_{\sigma'}^{0}$. Similarly, the adjunction map 
$$M'\To{} \Hom_{\RC_{\iota\sigma'}}(\PC_{\sigma',\iota\sigma'},M'\otimes_{\RC'_{\sigma'}}\PC_{\sigma',\iota\sigma'})$$
is an isomorphism for any $\RC'_{\sigma'}$-module $M'$, hence we have the desired Morita
equivalence at $\sigma'$.

It remains to prove condition iii) of \ref{equiv_induced}. To this end, we need to prove that if the idempotent
$e_{\sigma}^{s}$ in $R[\bar\GM_{\sigma}]$ is non zero, then $\sigma$ is conjugate to a
simplex of the form $\iota\sigma'$. Now, $e_{\sigma}^{s}$  is non zero if and only if the
conjugacy class of $s$ in ${\rm Aut}_{\FM}(\FM^{I})$ intersects $\GM_{\sigma}$. Since the
centralizer of $s$ is contained in ${\rm Aut}_{\FM'}({\FM'}^{I'})$, this means that
$\GM_{\sigma}$ contains a maximal torus isomorphic to one of ${\rm Aut}_{\FM'}({\FM'}^{I'})$ and therefore
 any $J\in \PC_{\sigma}$ has cardinality divisible by $f$. It follows that we can find
 $\bar w\in\SG_{n}$ such that $\bar w(\PC_{\sigma})=\PC_{\bar w\sigma}$ is the inverse image of a partition
 of $I'$. Now choose a vertex $x_{0}\in \bar w\sigma$, whence an ordering
 $\bar w\sigma=\{x_{0},\cdots,x_{d}\}$. We can find $t\in \ZM^{I}$ such that $tx_{0}\in
 \iota(\AM')$. Then the fact that $\bar w(\PC_{\sigma})=\PC_{\bar w\sigma}$ is the inverse image of a partition
 of $I'$ implies inductively on $i$ that each $tx_{i}$ is also in $\iota(\AM')$, so that finally
 $t\bar w\sigma\in\iota(\AM'_{\bullet})$.
\end{proof}

Now we need to prove that the equivalences of categories given by this
theorem preserve the property of being \emph{cartesian as $\RC^{0}$-modules (resp. as ${\RC'}^{0}$-modules)}.
To this end, it is useful to express directly the equivalence in terms of
$\RC^{0}$-modules and ${\RC'}^{0}$-modules.
Let us thus consider 
$$ \PC^{0}_{\sigma',\sigma} := \PC_{\sigma'}^{0}\otimes_{\RC_{\iota\sigma'}^{0}}
\RC^{0}_{\iota\sigma',\sigma} $$
for all $\sigma'\in\AM'_{\bullet}$ and $\sigma\in\AM_{\bullet}$. Exactly as in
\ref{cnstruc_bimodule}, this collection of $R$-modules has a clear left action by $\RC^{0}$, and we get a
right  ${\RC'}^{0}$-action using the $0$-weight part of the isomorphisms
$\phi_{\sigma',\tau'}$ of Lemma \ref{keylemma}. Therefore this collection
$\PC^{0}=(\PC^{0}_{\sigma',\sigma})_{\sigma'\in\AM'_{\bullet}, \sigma\in\AM_{\bullet}}$ has a
$(\RC^{0},{\RC'}^{0})$-bimodule structure. 

\begin{theo}
   $\PC^{0}$ induces a Morita equivalence between $\RC^{0}$ and
  ${\RC'}^{0}$. Moreover :
  \begin{enumerate}
  \item the associated equivalences of categories commute with those of Theorem
    \ref{mainthm} via the restriction functors from $\RC$-modules (resp. ${\RC'}$-modules)
    to $\RC^{0}$-modules (resp. ${\RC'}^{0}$-modules).
  \item the associated equivalences of categories preserve cartesian modules.
  \end{enumerate}
\end{theo}
\begin{proof}
  The fact that $\PC^{0}$ induces a Morita equivalence follows from Proposition
  \ref{equiv_induced},  exactly as in Theorem
  \ref{mainthm}. The only change is that we have to choose the $t$ in
  $(\ZM^{I})^{0}$, at the end of that proof. Consequently, also ii) follows from
  Proposition \ref{equiv_induced}. 

To see i), consider the obvious morphism of
$(\RC^{0},{\RC'}^{0})$-bimodules $\PC^{0} \To{} \PC_{|\RC^{0}\times{\RC'}^{0}}$. This
induces the following morphism of ${\RC'}^{0}$-modules, natural in the $\RC$-module $\MC$ 
$$\Hom_{\RC}(\PC,\MC)\To{}\Hom_{\RC^{0}}(\PC^{0},\MC).$$
Let $\sigma'\in\AM_{\bullet}$. 
The  map $\Hom_{\RC}(\PC,\MC)_{\sigma'}\To{}\Hom_{\RC^{0}}(\PC^{0},\MC)_{\sigma'}$ is isomorphic,
by Example \ref{exam_hom}, to the restriction map
$\Hom_{\RC_{\iota\sigma'}}(\PC_{\sigma',\iota\sigma'},\MC_{\iota\sigma'}) \To{} 
\Hom_{\RC^{0}_{\iota\sigma'}}(\PC^{0}_{\sigma',\iota\sigma'},\MC_{\iota\sigma'})$
which is bijective since 
$\PC_{\sigma',\iota\sigma'}\simeq\PC_{\sigma',\iota\sigma'}^{0}\otimes_{\RC_{\iota\sigma'}^{0}}\RC_{\iota\sigma'}$. 
Therefore, the  displayed morphism is an isomorphism of ${\RC'}^{0}$-modules.

On the other hand, we also have a morphism of $\RC^{0}$-modules, natural
in the $\RC'$-module $\MC'$ 
$$\PC^{0}\otimes_{{\RC'}^{0}}\MC'\To{}\PC\otimes_{\RC'}\MC'.$$
For any $\sigma'\in\AM_{\bullet}$, the map
$(\PC^{0}\otimes_{{\RC'}^{0}}\MC')_{\iota\sigma'}\To{}(\PC\otimes_{\RC'}\MC')_{\iota\sigma'}$
is isomorphic, by Lemma \ref{tens_induced} to the map
$\MC'_{\sigma'}\otimes_{{\RC'}_{\sigma'}^{0}}\PC^{0}_{\sigma',\iota\sigma'}
\To{}\MC'_{\sigma'}\otimes_{{\RC'}_{\sigma'}}\PC_{\sigma',\iota\sigma'}$ which is
bijective since $\PC_{\sigma',\iota\sigma'}\simeq
 \RC'_{\sigma'}\otimes_{{\RC'}^{0}_{\sigma'}}\PC_{\sigma',\iota\sigma'}^{0}$.
Since any simplex $\sigma$ with $\RC_{\sigma}\neq 0$ is conjugate to some $\iota\sigma'$
(see the proof of the last theorem),
this shows that the second displayed map is an isomorphism of $\RC^{0}$-modules.
\end{proof}

\appendix

\section{``Rings'' and ``modules''}

\subsection{Generalities}
\alin{``Rings''} \label{rings}
In this paper,  a ``ring'' is a small category $\RC$ such that all
Hom-sets $\RC(x,y)$ are abelian groups, and all composition maps
$\RC(x,y)\times\RC(y,z)\To{}\RC(x,z)$, $(f,g)\mapsto g\circ f$ are bilinear, thus defining linear maps 
$\RC(x,y)\otimes\RC(y,z)\To{}\RC(x,z)$ that satisfy the usual transitivity properties.
Such gadgets appear in the literature under several other
names, like ``ringoids'', ``rings with several objects'', or ``small preadditive categories'', or ``small categories enriched over the
tensor category of abelian groups''. 
Note that for any $x\in\RC$, the abelian group $\RC(x):=\RC(x,x)$ is an ordinary ring with
unit. However, this ring is allowed to be the zero ring, in which case $x$ is called a
``zero object''.

More generally, if $R$ is a commutative ring, an ``$R$-algebra'' will be a small category $\RC$ such that all
Hom-sets $\RC(x,y)$ are $R$-modules, and all composition maps
are $R$-bilinear, thus defining linear maps 
$\RC(x,y)\otimes_{R}\RC(y,z)\To{}\RC(x,z)$ that satisfy the usual transitivity properties.

If $\AC$ is any small category, we denote by $R[\AC]$ the following ``$R$-algebra''. Its
objects are those of $\AC$ and Hom-sets are the free $R$-modules on Hom-sets of $\AC$,
\emph{i.e.} $R[\AC](x,y):=R[\AC(x,y)]$. If we think of $\AC$ as a ``monoid with many
objects'', then $R[\AC]$ is the $R$-algebra of that monoid.

\alin{``Modules''} \label{modules}
Let $\RC$ be an ``$R$-algebra''. A left ``$\RC$-module'' $\MC$ is an
$R$-linear functor $\MC:\, \RC\To{} R-{\rm Mod}$. It thus consists in a collection
$(\MC(x))_{x\in\RC}$ of $R$-modules, together with $R$-linear maps 
 $\MC(x)\otimes_{R}\RC(x,y)\To{}\MC(y)$, $m\otimes f\mapsto \MC(f)(x)$
 that define an action in the sense that
 the following diagrams should be commutative\footnote{Our way of
  displaying these maps suggest that they should define a \emph{right} module structure,
  but in terms of composition of maps, this is really a \emph{left} module structure. For
  example the diagram says that $\MC(x)$ is a left $\RC(x,x)$-module} :
$$\xymatrix{
\MC(x)\otimes_{R}\RC(x,y)\otimes_{R}\RC(y,z) \ar[r] \ar[d] &  \MC(y)\otimes_{R}\RC(y,z) \ar[d] \\
\MC(x)\otimes_{R}\RC(x,z)\ar[r] &\MC(z)
}$$
With natural transforms as morphisms, left ``$\RC$-modules''
form an $R$-linear abelian category which we will denote by  $\RC-\rm Mod$. 
Note that $\MC(x)=\{0\}$ for any zero object $x$.

Similarly, a \emph{right} ``$\RC$-module'' will be a \emph{contravariant} functor $\MC
:\RC\To{} R-\rm Mod$. These objects again form an abelian $R$-linear category that we will
denote by ${\rm Mod}-\RC$. Note that for any $R$-module $M$, the collection $(\Hom_{R}(\MC(x),M))_{x\in \RC}$ gets
a structure of left $\RC$-module, that we denote by $\Hom_{R}(\MC,M)$.

If $\AC$ is a small category, then any functor (resp. contravariant functor)
$\EC:\,\AC\To{} R-\rm Mod$ canonically extends to a left (resp. right) ``$R[\AC]$-module''.

\alin{Restriction and (co)induction} \label{restinduct}
For any object $x$, 
we have an obvious restriction functor $\MC\mapsto \MC(x),$
$ \RC-{\rm Mod} \To{} \RC(x)-{\rm Mod}$.
In the other direction, we can induce a left $\RC(x)$-module $M$  to an $\RC$-module
${\rm Ind}_{x}M$ by putting 
$${\rm Ind}_{x}M(y):=M\otimes_{\RC(x)}\RC(x,y)$$ for any object $y$, with the
action of $\RC$ given by the multiplication in $\RC$. Similarly, we can coinduce $M$ to an
$\RC$-module ${\rm Coind}_{x}M$ defined by
$$ {\rm Coind}_{x}M(y):=\Hom_{\RC(x)}(\RC(y,x),M)$$

\begin{lem}
  The induction functor (resp. the coinduction functor) 
is left adjoint (resp. right adjoint) to the restriction functor ${\rm Res}_{x}$. 
\end{lem}
\begin{proof}
  For an $\RC(x)$-module $M$, we have a tautological isomorphism $\id_{M}\otimes 1:\,M\simto
  {\rm Ind}_{x}M(x)=M\otimes_{\RC(x)}\RC(x)$, and this defines a natural
transform $\id\simto{\rm Res}_{x}\circ {\rm Ind}_{x}$. On the other hand, for an
$\RC$-module $\MC$, the action maps 
$\MC(x)\otimes_{R}\RC(x,y)\To{}\MC(y)$ factor through
$\MC(x)\otimes_{\RC(x)}\RC(x,y)$, due to diagram \ref{modules}, and this defines a natural
transform ${\rm Ind}_{x}\circ{\rm Res}_{x}\To{} \id$. It is immediately checked that these
natural transforms are the unit and counit of an adjunction between ${\rm Res}_{x}$ and
${\rm Ind}_{x}$. The case of ${\rm Coind}_{x}$ is similar.
\end{proof}

There are similar functors ${\rm Ind}_{x}$ and ${\rm Coind}_{x}$  and  similar adjunctions
for \emph{right} modules. 

\begin{DEf}\label{def_induced}
  We say that a left (resp. right) $\RC$-module $\MC$ is \emph{induced from $x$} if it is
  isomorphic to some ${\rm Ind}_{x}M$, or equivalently if  
 the canonical maps 
 \begin{center}
   $a_{x,y}:\,\MC(x)\otimes_{\RC(x)}\RC(x,y)\To{}\MC(y)$ (resp.
   $\RC(y,x)\otimes_{\RC(x)}\MC(x)\To{}\MC(y)$)
 \end{center}
are isomorphisms for all $y\in\RC$. 

Similarly, we say that $\MC$ is \emph{coinduced from $x$} if it is isomorphic to some
${\rm Coind}_{x}M$, or equivalently if the canonical maps 
\begin{center}
  $a_{y,x}^{*}:\,\MC(y)\To{} \Hom_{\RC(x)}(\RC(y,x),\MC(x))$ (resp. $\MC(y)\To{}
  \Hom_{\RC(x)}(\RC(x,y),\MC(x))$)
\end{center}
are isomorphisms for all $y\in\RC$. 
\end{DEf}

\begin{DEf} \label{cart_modules}
   We say that a left $\RC$-module $\MC$ is \emph{cartesian}  or \emph{locally coinduced},
   if all maps $a_{y,x}^{*}$ as above are isomorphisms, whenever $\RC(x,y)$ is non-zero.
\end{DEf}


\alin{``Bimodules''} \label{bimodules}
Let $\RC$, $\RC'$ be two ``$R$-algebras''. A
``$(\RC,\RC')$-bimodule'' is an $R$-bilinear functor ${\RC'}^{\rm opp}\times\RC\To{} R-\rm Mod$. Concretely, this consists in a collection
$(\BC(x',x))_{x\in\RC, x'\in \RC'}$ of $R$-modules together with action maps 
$$ \BC(x',x)\otimes_{R}\RC(x,y)\To{} \BC(x',y)
\hbox{ and }
\RC(y',x')\otimes_{R} \BC(x',x) \To{} \BC(y',x)$$
that make each collection $(\BC(x',x))_{x\in\RC}$ into a left $\RC$-module and each
collection $(\BC(x',x))_{x'\in\RC'}$ into a right $\RC'$-module, and which ``commute''in
the sense that the following diagrams are commutative
$$\xymatrix{
\RC(y',x')\otimes_{R}\BC(x',x)\otimes_{R}\RC(x,y)  \ar[r] \ar[d] &
\RC(y',x')\otimes_{R}\BC(x',y)  \ar[d]\\
\BC(y',x)\otimes_{R}\RC(x,y) \ar[r] &
\BC(y',y)
}$$

\emph{Examples.} 

i) If $\MC$ is a left ``$\RC$-module'' and $\MC'$ is a right ``$\RC'$-module'', we get an
``$(\RC,\RC')$-bimodule'' $\MC'\otimes_{R}\MC$ with underlying collection
$(\MC'\otimes_{R}\MC)(x',x)=\MC'(x')\otimes_{R}\MC(x)$ 

ii) If $\BC$ is a ``$(\RC,\RC')$-bimodule'', its ``dual''  $\BC^{*}$ defined by
$\BC^{*}(x,x'):=\Hom_{R}(\BC(x',x),R)$ is naturally a ``$(\RC',\RC)$-bimodule''.

iii) If $\MC,\NC$ are two ``$\RC$-modules'', we get a ``$(\RC,\RC)$-bimodule''
$\Hom_{R}(\MC,\NC)$ by putting $\Hom_{R}(\MC,\NC)(x',x):=\Hom_{R}(\MC(x'),\NC(x))$ with the obvious actions.

\alin{Ends and coends} \label{ends}
Let $\BC$ be a $(\RC,\RC)$-bimodule. Its \emph{end} is an
$R$-module denoted by
$\int_{\RC}\BC$ and defined by
$$ \int_{\RC} \BC : =  \ker\left( \prod_{x\in\RC} \BC(x,x) \rightrightarrows
  \prod_{y,z\in\RC} \Hom_{R}(\RC(y,z),\BC(y,z))\right)
$$ 
where the first map is a product of the maps $\BC(x,x)\To{}\Hom_{R}(\RC(x,z),\BC(x,z))$ adjoint
to $\BC(x,x)\otimes_{R}\RC(x,z)\To{}\BC(x,z)$ and the second map is a product of the maps
$\BC(x,x)\To{}\Hom_{R}(\RC(y,x),\BC(y,x))$ adjoint to $\RC(y,x)\otimes_{R}\BC(x,x)\To{}\BC(y,x)$.

Similarly, its \emph{coend} is an $R$-module denoted by $\int^{\RC}\BC$ and is defined by
$$ \int^{\RC} \BC := \coker\left(
 \bigoplus_{y,z\in\RC} \RC(z,y)\otimes_{R}\BC(y,z)
\rightrightarrows
 \bigoplus_{x\in\RC} \BC(x,x) 
\right)
$$
where the first map is a sum of maps $\RC(z,y)\otimes_{R}\BC(y,z)\To{}\BC(y,y)$ (giving the
left modules structure) and the second map is a sum of maps
$\RC(z,y)\otimes_{R}\BC(y,z)\To{}\BC(z,z)$ (giving the right module structure).

\begin{exam}\label{exam_hom}
  If $\MC,\NC$ are two $\RC$-modules, we have $\int_{\RC}\HC om_{R}(\MC,\NC)=
  \Hom_{\RC}(\MC,\NC)$.  We note that when $\MC$ is induced from an object $x$, Lemma
  \ref{restinduct} tells us that the canonical map $\int_{\RC}\HC om_{R}(\MC,\NC)\To{}
  \Hom_{R}(\MC(x),\NC(x))$ induces an isomorphism $$\Hom_{\RC}(\MC,\NC)\simto
  \Hom_{\RC(x)}(\MC(x),\NC(x)).$$
\end{exam}

\begin{DEf}
  If $\MC$, resp. $\NC$ is a right, resp. left, ``$\RC$-module'', we define their ``tensor
  product over $\RC$'' as the $R$-module
$$\MC\otimes_{\RC}\NC:=\int^{\RC}\MC\otimes_{R}\NC$$
\end{DEf}

When $\RC$ has a single object and is thus an ordinary ring, this is consistent with  the usual tensor
product over $\RC$. Moreover, classical adjunction properties of tensor products extend to
this context.

\begin{lemme}
There is a canonical isomorphism
$$\Hom_{R}(\MC\otimes_{\RC}\NC,M)\simto \Hom_{\RC}(\NC,\Hom_{R}(\MC,R))$$
natural in the left $\RC$-module $\NC$, the right $\RC$-module $\MC$, and the $R$-module $M$.
\end{lemme}
\begin{proof}
Since $\Hom_{R}(-,M)$ is left exact,  we have 
  \begin{eqnarray*}
    \Hom_{R}(\MC\otimes_{\RC}\NC,M) 
& = & 
\int_{R}  (x',x)\mapsto \Hom_{R}(\MC(x')\otimes_{R}\NC(x),M)  \\
& \simto & \int_{R} (x',x)\mapsto \Hom_{R}(\NC(x), \Hom_{R}(\MC(x'),M)) \\
& = & \Hom_{\RC}(\NC,\Hom_{R}(\MC,R))
  \end{eqnarray*}
where the isomorphism is classical adjunction over $R$.
\end{proof}

\begin{coro} \label{tens_induced}
  If $\NC$ (resp. $\MC$) is induced from the object $x$, the canonical map
$\MC(x)\otimes_{R}\NC(x)\To{}\int^{\RC}\MC\otimes_{R}\NC$
 induces an isomorphism
 $$\MC(x)\otimes_{\RC(x)}\NC(x)\simto\MC\otimes_{\RC}\NC.$$
\end{coro}
\begin{proof}
  For any $R$-module $M$, the previous lemma tells us that the functor $\Hom_R(-,M)$ takes
  the canonical map under consideration to the canonical map 
$$ \int_{\RC} (x',x)\mapsto \Hom_{\RC}(\NC(x),\Hom_{R}(\MC(x'),M)) \To{} \Hom_{R}(\NC(x),\Hom_{R}(\MC(x),M)).$$
But as noted above, the latter map induces an
isomorphism
$$ \Hom_{\RC}(\NC,\Hom_{R}(\MC,M)) \simto \Hom_{\RC(x)}(\NC(x),\Hom_{R}(\MC(x),M)),$$
which means that the map under consideration induces an isomorphism
 $$\Hom_{R}(\MC\otimes_{\RC}\NC,M) \simto\Hom_{R}(\MC(x)\otimes_{\RC(x)}\NC(x),M).$$
\end{proof}

\subsection{Morita equivalences}

\alin{``Bimodules'' as functors between categories of ``modules''} 
If $\CC$ is an $R$-linear category with finite limits and colimits, one
defines the end and coend of an $R$-bilinear functor $\BC:\,\RC^{\rm opp}\times\RC\To{}\CC$ by
the same formulas as in \ref{ends}. These are objects of $\CC$. An example of such a category
$\CC$ is $\RC'-\rm Mod$. Indeed, finite limits and colimits exist in this category and are
computed terms by terms, \emph{e.g.} $(\lim_{I} \MC'_{i})(x')=\lim_{I} \MC'_{i}(x')$.

Now, let $\BC$ be a
``$(\RC,\RC')$-bimodule'' as in \ref{bimodules} and let $\MC$ be an ``$\RC$-module''. Then $\HC
om_{R}(\BC,\MC)$ is an $R$-trilinear functor $\RC^{\rm opp}\times{\RC'}\times
\RC\To{}R-\rm Mod$, and can be seen as an $R$-bilinear functor 
$\HC:\,\RC^{\rm opp}\times\RC\To{}\RC'-\rm Mod$.
The end $\int_{\RC}\HC$ of $\HC$ is therefore a left ``$\RC'$-module'' that we will
denote by $\Hom_{\RC}(\BC,\MC)$.
 In this way we obtain a functor
$$\application{}{\RC-\rm Mod}{\RC'-\rm Mod}{\MC}{\Hom_{\RC}(\BC,\MC)}$$
Concretely, for any $x'\in\RC$, we have
$$\Hom_{\RC}(\BC,\MC)(x')=\Hom_{\RC}(\BC(x',-),\MC)=\int_{\RC}\Hom_{R}(\BC(x',x),\MC(x)).$$

On the other hand, let $\MC'$ be a left ``$\RC'$-module''. We get an $R$-trilinear
functor $\BC\otimes_{R}\MC' :\, \RC\times{\RC'}^{\rm opp}\times\RC'\To{} R-\rm Mod$ whence
an $R$-bilinear functor ${\RC'}^{\rm opp}\times\RC'\To{} \RC-\rm Mod$. As above we get a
functor
$$\application{}{\RC'-\rm Mod}{\RC-\rm Mod}{\MC'}{\BC\otimes_{\RC'}\MC'}$$
with $(\BC\otimes_{\RC'}\MC')(x) = \BC(-,x)\otimes_{\RC'}\MC'=\int^{\RC'}\BC(x',x)\otimes_{R}\MC'(x')$.

\begin{prop}
  The functor $\MC'\mapsto\BC\otimes_{\RC'}\MC'$ is left adjoint to the functor
  $\MC\mapsto \Hom_{\RC}(\BC,\MC)$.
\end{prop}
\begin{proof}
  We first define a counit. Let $\MC$ be a left $\RC$-module. For any $x\in\RC$ and
  $x'\in\RC'$, consider the canonical map
$$\Hom_{\RC}(\BC,\MC)(x')=\Hom_{\RC}(\BC(x',-),\MC) \To{} \Hom_{R}(\BC(x',x),\MC(x)) $$
Taking the sum of the adjoint  maps, we get a map
$$ \bigoplus_{x'\in\RC'}\BC(x',x)\otimes_{R}\Hom_{\RC}(\BC,\MC)(x') \To{}\MC(x).$$
It is tediously but easily checked that this map  factors through the coend, giving a map
$$ (\BC\otimes_{\RC'}\Hom_{\RC}(\BC,\MC))(x) = \BC(-,x)\otimes_{\RC'}\Hom_{\RC}(\BC,\MC)
\To{} \MC(x).$$
When varying $x$, this collection of maps defines a morphism of $\RC$-modules
\ini
\begin{equation}
\BC\otimes_{\RC'}\Hom_{\RC}(\BC,\MC)\To{}\MC\label{adj1}
\end{equation}
which is clearly natural in $\MC$.

Let us now define the unit. Let $\MC'$ be a left $\RC'$-module. For any $x\in\RC$ and
  $x'\in\RC'$, consider the canonical map
$$  \BC(x',x)\otimes_{R} \MC'(x') \To{} \BC(-,x)\otimes_{\RC'}\MC'=(\BC\otimes_{\RC'}\MC')(x) $$
Taking the product of the adjoint maps, we get a map
$$ \MC'(x') \To{} \prod_{x\in\RC}\Hom_{R}(\BC(x',x),(\BC\otimes_{\RC'}\MC')(x)).$$
Another tedious and easy check insures that this map factors through the end, giving  a map
$$ \MC'(x') \To{} \Hom_{\RC}(\BC(x',-),\BC\otimes_{\RC'}\MC')=\Hom_{\RC}(\BC,\BC\otimes_{\RC'}\MC')(x').$$
Now, varying $x$, this collection of maps  defines a morphism of $\RC'$-modules
\ini\begin{equation}
 \MC' \To{} \Hom_{\RC}(\BC,\BC\otimes_{\RC'}\MC')\label{adj2}
\end{equation}
which is clearly natural in $\MC'$.

It remains to check that these maps satisfy the axioms of a counit-unit system. We leave
the reader check in detail that this indeed follows from the classical adjunction between
functors $\BC(x',x)\otimes_{R}-$ and $\Hom_{R}(\BC(x',x),-)$ for each
pair of objects $x',x$.
\end{proof}

\begin{DEf} \label{def_Morita}
  We say that $\BC$ induces a Morita equivalence between $\RC$ and $\RC'$ if the 
  functors of the last proposition are equivalences of categories or, equivalently, if all
  maps (\ref{adj1}) and (\ref{adj2}) are isomorphisms.
\end{DEf}

The following result is relevant for our applications in the main body of this paper.

\begin{prop} \label{equiv_induced}
 Assume  there is a map $\iota:\, \RC'\To{}\RC$ such that the following holds.
  \begin{enumerate}
  \item for all $x'\in\RC'$, the
  left $\RC$-module $x\mapsto \BC(x',x)$ is induced from $\iota x'$
and the right $\RC'$-module $y'\mapsto \BC(y',\iota x')$ is induced from $x'$ in the sense
of Definition \ref{def_induced}.

  \item for all $x'\in \RC'$, the bimodule $\BC(x',\iota x')$ induces a Morita equivalence
    between the $R$-algebras $\RC(\iota x')$ and $\RC'(x')$.
\item  any non zero object $x\in\RC$ is isomorphic to some $\iota x'$ in $\RC$.
  \end{enumerate}
Then $\BC$ induces a Morita equivalence between $\RC$ and $\RC'$. Moreover, the 
 two associated equivalences of categories preserve the propery of
 being cartesian.
\end{prop}
\begin{proof}
We first show that the map  (\ref{adj1}) is an isomorphism for any $\RC'$-module
$\MC'$. To this end we shall evaluate it at any $\iota x'$, for $x'\in\RC'$.
Our assumption i) provides us with  an isomorphism
$$ \Hom_{\RC}(\BC,\MC)(x')\simto \Hom_{\RC(\iota x')}(\BC(x',\iota x'),\MC(\iota x')),$$
as in  example \ref{exam_hom}, as well as an isomorphism
$$ \BC(x',\iota x')\otimes_{\RC'(x')}\Hom_{\RC}(\BC,\MC)(x')
\simto (\BC\otimes_{\RC'}\Hom_{\RC}(\BC,\MC))(\iota x'),$$ 
thanks to Lemma \ref{tens_induced}.
Combining these two
isomorphisms gives the following one
$$(\BC\otimes_{\RC'}\Hom_{\RC}(\BC,\MC))(\iota x') \simeq 
\BC(x',\iota x')\otimes_{\RC'(x')}\Hom_{\RC(\iota x')}(\BC(x',\iota x'),\MC(\iota x')).$$
Tracking backward  definitions  shows that this isomorphism
 identifies  the counit (\ref{adj1}) at $\iota x'$ with the evaluation map of the RHS.
But our hypothesis ii) implies that this evaluation map is an isomorphism. Hence we have
proved that the counit (\ref{adj1}) is an isomorphism at any object of the form  $\iota x'$. By
hypothesis iii), it is an isomorphism at any $x\in\RC$.

The proof that the map (\ref{adj2}) is an isomorphism is very similar, except that we
don't need hypothesis ii) in this case.

It remains to prove the last assertion about cartesian modules. Let $\MC$ be an
$\RC$-module. It suffices to prove that $\MC$ is  cartesian if and only if $\MC'=\Hom_{\RC}(\BC,\MC)$ is cartesian.
To this aim, we have to study the map $a_{y',x'}^{*}$  of
\ref{def_induced} for the $\RC'$-module $\MC'=\Hom_{\RC}(\BC,\MC)$ and any $x',y'\in\RC'$. 
 We leave the reader convince himself that this map is the following  composition of  maps.
\begin{eqnarray*}
a_{y',x'}^{*}:\,\Hom_{\RC}(\BC,\MC)(y') &\simto &
  \Hom_{\RC(\iota y')}(\BC(y',\iota y'),\MC(\iota y')) \\
&\To{a^{*}_{\iota y',\iota x'}} & 
  \Hom_{\RC(\iota y')}(\BC(y',\iota y'),\Hom_{\RC(\iota x')}(\RC(\iota y',\iota x'),\MC(\iota x')))\\
&\simto & 
  \Hom_{\RC(\iota x')}(\BC(y',\iota y')\otimes_{\RC(\iota y')}\RC(\iota y',\iota x'),\MC(\iota x'))\\
&\simto & 
  \Hom_{\RC(\iota x')}(\BC(y',\iota x'),\MC(\iota x'))\\
&\simto & 
  \Hom_{\RC(\iota x')}(\RC'(y',x')\otimes_{\RC'(x')} \BC(x',\iota x'),\MC(\iota x'))\\
&\simto & 
  \Hom_{\RC'(x')}(\RC'(y',x'), \Hom_{\RC(\iota x')}(\BC(x',\iota x'),\MC(\iota x')))\\
& \simto &   \Hom_{\RC'(x')}(\RC'(y',x'), \Hom_{\RC}(\BC,\MC)(x'))\\
\end{eqnarray*}
In the first and last lines, we have applied \ref{exam_hom}, using
 that $y\mapsto \BC(y',y)$ is induced from $\iota y'$.
 The third and sixth lines are classical adjunction properties. 
The fourth line again uses that $y\mapsto \BC(y',y)$ is induced from $\iota y'$.
The fifth line is given by the right $\RC'$-module structure on $\BC(-,\iota x')$.
It is an isomorphism, by our  hypothesis that $y'\mapsto
\BC(y',\iota x')$ is induced from $x'$. It follows from this sequence of maps that
\begin{center}
  \emph{$a^{*}_{\iota y',\iota x'}$ is an isomorphism if and only if $a_{y',x'}^{*}$ is  an
    isomorphism.}
\end{center}
Going back to the definition of cartesian
modules, we see that, thanks to hypothesis iii), 
what remains to prove is 
$$\RC'(y',x')\neq 0\Leftrightarrow\RC(\iota y',\iota x')\neq 0.$$ 
But by hypothesis i) we have the folloving isomorphisms
$$ \RC'(y',x')\otimes_{\RC'(x')}\BC(x',\iota x') \simto\BC(y',\iota x')
{\buildrel\hbox{\tiny{$\sim$}}\over\longleftarrow}
\BC(y',\iota y')\otimes_{\RC(\iota y')}\RC(\iota y',\iota x')$$
and by hypothesis ii), the functors $-\otimes_{\RC'(x')}\BC(x',\iota x')$ and
$\BC(y',\iota y')\otimes_{\RC(\iota y')}-$ are faithful.
\end{proof}


\end{document}